\pdfoutput=1
\documentclass[11pt]{article}
\usepackage{amsmath,amsthm,amssymb} 
\usepackage{flexisym} 
\usepackage{breqn}
\usepackage[english]{babel}
\usepackage{lmodern}
\usepackage{avant}
\usepackage{rotating}
\usepackage{adjustbox}
\usepackage{multirow}
\usepackage{booktabs}
\usepackage{courier}
\usepackage{makecell}

\usepackage{subcaption}
\usepackage[T1]{fontenc}
\usepackage[latin9]{inputenc}

\usepackage{bbold} 
\newcommand{\idd}{\mathbb{1}} 

\usepackage[a4paper]{geometry}
\usepackage{mathtools}   
\usepackage{float}
\usepackage{stmaryrd}
\usepackage{setspace}
\usepackage[authoryear]{natbib}
\PassOptionsToPackage{normalem}{ulem}
\usepackage{ulem}
\usepackage[unicode=true,pdfusetitle,
 bookmarks=true,bookmarksnumbered=false,bookmarksopen=false,
 breaklinks=false,pdfborder={0 0 0},pdfborderstyle={},backref=false,colorlinks=true]
 {hyperref}
\hypersetup{
 pdfborderstyle={},pdfborderstyle={},linkcolor=blue,urlcolor=blue,citecolor=blue}
\usepackage{units}
\usepackage{authblk} 
\allowdisplaybreaks

\makeatletter
\numberwithin{equation}{section}
\numberwithin{figure}{section}
\numberwithin{table}{section}

\theoremstyle{plain}
\newtheorem{thm}{Theorem}
  \theoremstyle{remark}
  \newtheorem{rem}[thm]{Remark}
  \theoremstyle{definition}
  \newtheorem{defn}[thm]{Definition}
  \theoremstyle{plain}
  \newtheorem{prop}[thm]{Proposition}
  \theoremstyle{plain}
  \newtheorem{lem}[thm]{Lemma}
  \theoremstyle{plain}
  
  \theoremstyle{definition}
  \newtheorem{exa}[thm]{Example}
  \theoremstyle{remark}
  \newtheorem{assumption}[thm]{Assumption}

\makeatother

\title{Extended Tetrad Analysis in Factor Modelling: Separability and Uncertainty from Multidimensional Dependence Structures}

\author[1,2]{Mario Angelelli}
{\small
\affil[1]{%
Department of Human and Social Sciences, University of Salento, Lecce, 73100, Italy}%
\affil[2]{%
GNSAGA Research Group, National Institute of Advanced Mathematics (INdAM), Rome, 00185, Italy}%
}
\date{}

\begin{document}
\maketitle


\begin{abstract} 
Geometric representations provide a principled framework for structuring the description of latent constructs and clarifying sources of uncertainty in their dimensional characterisation. We introduce a novel geometric representation of factor models via two subspaces spanned by paired matrices, where determinantal expressions explicitly quantify the contributions of different dimension subsets to the factor structure. This formulation refines rank-based conditions relevant to understanding factor score indeterminacy and the implications of non-uniqueness in instrumental variable estimation for over-identified models. 

By weighting these multidimensional contributions to encode sensitivity to their variation, we extend the definition of tetrads into an algebraic procedure that establishes conditions for identifying variability components attributable to individual dimensions. Focusing on cases where one factor encodes structural information, we derive minimal conditions---expressed through graph planarity---that ensure such dimension-specific identifiability. The proofs yield both formal verification tools and constructive methods for generating counterexamples where these conditions fail. These counterexamples reveal a type of ambiguity, termed contextuality, in which the comparison of dimensional contributions depends on the choice of remaining reference dimensions, violating well-established axioms of order-theoretic consistency. We relate these findings to specific forms of uncertainty examined in the psychometric literature. 

\texttt{Keywords}: Tetrads; Factor Indeterminacy; Instrumental Variables; Dependence Structures; Separability.

\end{abstract}

\section{Introduction} 
\label{sec: Introduction} 

Several key quantities in factor analysis and graphical models, particularly in structural equation modelling (SEM), can be expressed as determinants of submatrices (i.e., minors) of matrices encoding empirical or structural relations \citep{Drton2007,Drton2018}. This reflects the central role of minors in algebraic analyses of covariance structures. A notable example is Confirmatory Tetrad Analysis, where identities involving $2\times 2$ minors of a covariance matrix---under the assumption of a latent factor \citep[Example 6]{Drton2007}---enable the testing of measurement models in SEM, especially by providing evidence to differentiate between effect and causal indicator specifications \citep{Bollen1989,Bollen2000}. Minors are also instrumental in assessing model identifiability \citep{Weihs2018} and in evaluating partial correlations \citep{Steiger1979,Boege2018}, which help examine implications of factor indeterminacy \citep{steiger1979factor,Rigdon2019}. 

Factor indeterminacy is further elucidated by geometric descriptions \citep{heermann1964geometry,krijnen2002construction}, where the non-uniqueness of factor scores entails uncertainty in their correlations or angles with external variables \citep[Appendix A]{rhemtulla2024estimated}. 
We observe that this approach can be naturally extended to \emph{collections} of score realisations by considering the subspace they span, which admits a basis-invariant characterisation through matrix minors \citep[Ch. 3.1]{GKZ1994}, enabling the exploration of alignment with a second subspace via principal angles or related indices. In particular, we focus on the determinantal expansion \citep[Sec. 0.8.7]{Horn2012} 
 \begin{equation}
\det(\mathbf{L}\cdot\mathbf{R})=\sum_{\mathcal{I}\in \wp_{k}[n]}\Delta_{\mathbf{L}}(\mathcal{I})\cdot\Delta_{\mathbf{R}}(\mathcal{I})
\label{eq: Cauchy-Binet expansion}
 \end{equation} 
where $\mathbf{L},\mathbf{R}^{\mathtt{T}}$ are $k\times n$ matrices with $k\leq n$, $[n]=\{1,\dots,n\}$, and $\wp_{k}[n]=\{\mathcal{I}\subseteq [n]:\,|\mathcal{I}|=k\}$ is the collection of all $k$-element subsets of $[n]$; $\Delta_{\mathbf{L}}(\mathcal{I})$ and $\Delta_{\mathbf{R}}(\mathcal{I})$ denote the $k\times k$ subdeterminants (i.e., maximal minors) of $\mathbf{L}$ and $\mathbf{R}$, extracted from columns and rows indexed by $\mathcal{I}$, respectively. 
This expansion 
induces an inner product between vectors of maximal minors, 
yielding, after normalisation, a similarity measure associated with principal angles between the corresponding $k$-dimensional subspaces, as used in canonical correlation analysis \citep{wolf2003learning,Hamm2008}. 

Building on such connections with canonical correlations, these expressions also apply in the context of instrumental variable (IV) estimation \citep[Sec. 2.3]{Bowden1985}. IVs are widely used in causal modelling \citep{Brito2002,Weihs2018}, and selected indicators can act as model-implied instruments in factor analysis \citep{bollen2024model}. In \emph{over-identified} models, instruments outnumber endogenous variables, whereas a just-identified model contains equal numbers of each. Over-identification thus permits multiple just-identified submodels, each derived by selecting a different subset of instruments. Estimators for these submodels \citep[Sec. 1.2]{Bowden1985} can be combined via a weight matrix, and this combination can be expressed through the expansion (\ref{eq: Cauchy-Binet expansion}), revealing a form of estimator arbitrariness and specification-induced effects \citep{rotemberg1983instrument}. 

This work introduces a novel geometric representation of such models, beginning with an invariant characterisation of a factor $\mathbf{R}$---which may, for instance, represent a score matrix or support estimation from empirical observations---via the subspace it spans, parameterised by its maximal minors \citep[Ch. 3.1.C]{GKZ1994}; $\mathbf{R}$ is assumed generic, in which case all such minors are non-zero. We then focus on factorisations incorporating \emph{structural} information through a second factor (e.g., a design or weight matrix), whose generated subspace is analogously characterised by the maximal minors of $\mathbf{L}$. These minors determine the pattern of non-zero---i.e., \emph{observable}---contributions in (\ref{eq: Cauchy-Binet expansion}), indexed by 
\begin{equation} 
\mathfrak{G}(\mathbf{L})=\left\{ \mathcal{I}\in\wp_{k}[n]:\,\Delta_{\mathbf{L}}(\mathcal{I})\neq0\right\}. 
\label{eq: matroid of non-zero minors}
\end{equation}
We operate directly on the bases defining the independence structure (\ref{eq: matroid of non-zero minors})---each of which may reflect, for instance, a different choice of reference dimensions 
or instruments used in estimation \citep[Eq. (6)]{rotemberg1983instrument}---through the corresponding terms $\Delta_{\mathbf{L}}(\mathcal{I})\cdot \Delta_{\mathbf{R}}(\mathcal{I})$ from (\ref{eq: Cauchy-Binet expansion}). By treating these non-zero terms as the primitive elements of our representation, we capture how each basis in (\ref{eq: matroid of non-zero minors}) contributes to the association between $\mathbf{R}$ and $\mathbf{L}$. More generally, this formulation extends to any rank-$k$ factorisation $\mathbf{L}\cdot \mathbf{R}$, even when $\mathbf{R}\in\mathbb{R}^{n\times s}$ and $k\neq s$, given that each maximal minor involved in parameterising the subspace spanned by $\mathbf{L}\cdot \mathbf{R}$ can be expressed via (\ref{eq: Cauchy-Binet expansion}), as discussed later. 

Each basis' contribution is analysed under varying weights assigned to the corresponding minor product, subject to 
determinantal constraints. This approach supports a form of model sensitivity analysis specific to our representation (Section \ref{subsec: relation to sensitivity analysis in monomial models}), where each weight reflects the variation attributable to a given minor product---and, by extension, to the corresponding subset of components or dimensions. We introduce a class of weight assignments decomposing over individual dimensions coherently across bases and investigate minimal conditions on (\ref{eq: matroid of non-zero minors}) under which these assignments are the only admissible weights. We refer to this property as the \emph{identifiability} of the factor model representation, and our analysis highlights a graphical characterisation of such conditions in relation to graph planarity (see Remark \ref{rem: analogy with Kuratowski}). We then derive expressions generalising tetrads within our representation, enabling an algebraic procedure to recover component-specific contributions when these structural conditions are satisfied. Concurrently, this procedure guides the construction of counterexamples in which the failure of any one of these conditions leads to a loss of identifiability. By ordering bases according to the norm of their weights, such counterexamples reveal how the non-separability of component-wise contributions introduces ambiguity into the model, violating a condition of context independence analogous to the Independence of Irrelevant Alternatives (IIA) axiom \citep{Luce2005}. As such, these violations indicate a form of contextuality---that is, a dependence on the basis chosen for pairwise comparisons of component-specific contributions---that can be discussed in relation to foundational issues of latent score indeterminacy, as elements included in multiple collections pertaining to different domains of measures may acquire contextual meaning (see Remark \ref{rem: self-compatibility score matrices}). 

The results presented here lay methodological foundations for future developments, deepening the conceptual and practical implications of geometric formulations of measurement compatibility across contexts---in line with the previous considerations---and their graphical characterisations. Additionally, our treatment of structural factors could be broadened to accommodate alternative interpretations. In particular, the connection with principal angles between subspaces suggests extending this representation to analyse cross-covariances or cross-correlations between scores and external variables, thereby supporting the evaluation of 
uncertainty in score assignments. 

The paper is organised as follows. Section \ref{sec: previous work} sets the notation and formalises the conditions underlying identifiability. Section \ref{sec: theoretical framework} presents and contextualises the main results in relation to the state of the art, focusing on specific facets of uncertainty arising in the psychometric literature, especially IV methods (Section \ref{subsec: relation to uncertainty in IV estimation}), factor indeterminacy (Section \ref{subsec: relation to structural equations under partial information and factor indeterminacy}), and ambiguity or deviations from order-theoretic consistency conditions such as the IIA axiom (Section \ref{subsec: counterexamples and relations to choice theory}). From Section \ref{sec: principal minors and hyperdeterminants}, we provide the proofs of the technical results, first focusing on a precondition for identifiability in a given basis, then showing in Section \ref{sec: propagation of separability between different bases} how this property \emph{propagates} within $\mathfrak{G}(\mathbf{L})$. Section \ref{sec: counterexamples} presents counterexamples illustrating violations of these assumptions and linking them to structural inconsistency as failures of this IIA-type condition. Finally, Section \ref{sec: conclusion} summarises the findings and outlines directions for future work.

\section{Preliminaries: Structural Conditions and Graphical Interpretation}
\label{sec: previous work}

We start formalising a weight assignment to each term in (\ref{eq: Cauchy-Binet expansion}) through the mapping 
\begin{equation}
    \mathcal{I}\mapsto \Delta_{\mathbf{L}}(\mathcal{I})\cdot \Delta_{\mathbf{R}}(\mathcal{I}) \cdot \mathbf{t}^{\Psi(\mathcal{I})},\quad \mathcal{I}\in\mathfrak{G}(\mathbf{L})
    \label{eq: toric deformation}
\end{equation}
where $\Psi(\mathcal{I})\in\mathbb{Z}^{d}$, $d\in\mathbb{N}$, $\mathbf{t}$ is a $d$-tuple of parameters, and $\mathbf{t}^{\mathbf{e}}=\prod_{u=1}^{d} t_{u}^{e_{u}}$. 
This parameterisation can be specified to explore base-specific weights $c_{\mathcal{I}}\neq 0$ by considering $d = |\mathfrak{G}(\mathbf{L})|$, ordering $\mathfrak{G}(\mathbf{L})$, and defining $\Psi$ as an indicator function---namely, $\Psi(\mathcal{I})_{u} = 1$ if $\mathcal{I}$ is the $u$-th basis in the ordered tuple and $0$ otherwise, then setting $t_{u}=c_{\mathcal{I}}$. Other specifications allow investigating algebraic interdependencies among such weights, consistent with the independence structure encoded in $\mathfrak{G}(\mathbf{L})$, and we consider the form (\ref{eq: toric deformation}) because it preserves that structure. Specifically, for $\Delta_{\mathbf{R}}(\mathcal{I})\neq 0$ (see Assumption \ref{claim: no Y=-1} below), the product (\ref{eq: toric deformation}) vanishes for generic $\mathbf{t}$ if and only if $\Delta_{\mathbf{L}}(\mathcal{I})=0$ holds, reflecting a linear dependence among columns of $\mathbf{L}$ indexed by $\mathcal{I}$. The results below also apply to parameterisations admitting a monomial representation in algebraically independent parameters after a suitable change of variables. 
Next, we look for $\mathbf{t}$-dependent matrices $\mathbf{L}(\mathbf{t})$ and $\mathbf{R}(\mathbf{t})$ that generate the minor products (\ref{eq: toric deformation}) as 
$h(\mathcal{I}) := \Delta_{\mathbf{L}(\mathbf{t})}(\mathcal{I})\cdot\Delta_{\mathbf{R}(\mathbf{t})}(\mathcal{I})$, where $\mathfrak{G}(\mathbf{L}(\mathbf{1}))=\mathfrak{G}(\mathbf{L})$ and $\mathcal{I}\in\mathfrak{G}(\mathbf{L})$---hereafter, `$:=$' denotes assignment. Note that the set of such terms 
is invariant under transformations 
\begin{equation}
\left(\mathbf{L}(\mathbf{t}),\mathbf{R}(\mathbf{t})\right)\mapsto \left(\mathbf{L}(\mathbf{t})\cdot \mathbf{D}(\mathbf{t})^{-1},\mathbf{D}(\mathbf{t})\cdot \mathbf{R}(\mathbf{t})\right)
\label{eq: gauge invariance}
\end{equation}
where $\mathbf{D}(\mathbf{t})$ is a generalised $n\times n$ permutation matrix depending on $\mathbf{t}$. We can obtain determinantal forms of the type (\ref{eq: toric deformation}) by taking a vector $\mathbf{d}(\mathbf{t})=(\mathbf{t}^{\mathbf{p}_{u}}:\,u\in[n])$ with $\mathbf{p}_{u}\in\mathbb{Z}^{d}$ and setting 
\begin{equation} 
\mathbf{L}(\mathbf{t}) = \mathbf{L}(\mathbf{1}) \cdot \mathrm{diag}\left(\mathbf{d}(\mathbf{t})\right),\quad  \mathbf{R}(\mathbf{t}) = \mathbf{R}(\mathbf{1}) \text{ constant}. 
\label{eq: separable deformations}
\end{equation} 
As anticipated, we say that identifiability holds if the contributions derived through (\ref{eq: separable deformations}), up to a common monomial factor, are the only weights consistent with the determinantal constraints. 

\subsection{Notation}
\label{subsec: Notation} 

Using $\setminus$ for set difference and defining $\mathcal{I}^{\mathtt{C}}:=[n]\setminus\mathcal{I}$, we denote the symmetric difference between sets $\mathcal{I}$ and $\mathcal{J}$ as $\mathcal{I}\Delta\mathcal{J}:=(\mathcal{I}\setminus \mathcal{J})\cup(\mathcal{J}\setminus \mathcal{I})$, and we introduce the notation 
 \begin{equation}
\mathcal{I}_{\alpha_{1}\alpha_{2}\dots}^{i_{1}i_{2}\dots}:=\mathcal{I}\backslash\{i_{1},i_{2},\dots\}\cup\{\alpha_{1},\alpha_{2},\dots\},\quad i_{1},i_{2},\dots\in\mathcal{I},\,\alpha_{1},\alpha_{2},\dots\in\mathcal{I}^{\mathtt{C}}. 
\label{eq: 1-exchange}
 \end{equation} 
The set $\mathfrak{G}(\mathbf{L})$ is a \emph{matroid} \citep{Oxley2011}, a combinatorial structure that generalises independence relations and satisfies the exchange relation 
 \begin{equation}
\text{for all }\mathcal{A},\mathcal{B}\in\mathfrak{G}(\mathbf{L}),\,\alpha\in\mathcal{A}\setminus\mathcal{B}:\quad\text{there exists }\beta\in\mathcal{B}\setminus\mathcal{A} \text{ such that }\mathcal{A}_{\beta}^{\alpha}\in\mathfrak{G}(\mathbf{L})
\label{eq: exchange relation}
 \end{equation}
which is equivalent to the following symmetric exchange property \citep{Brualdi1969}
\begin{equation} 
\text{for all }\mathcal{A},\mathcal{B}\in\mathfrak{G}(\mathbf{L}),\,\alpha\in\mathcal{A}\setminus\mathcal{B}:\quad\text{there exists }\beta\in\mathcal{B}\setminus\mathcal{A} \text{ such that }\mathcal{A}_{\beta}^{\alpha},\mathcal{B}_{\alpha}^{\beta}\in\mathfrak{G}(\mathbf{L}). 
\label{eq: symmetric exchange relation}
\end{equation}
We also introduce the binary relations $\triangledown_{\mathcal{H}}$ on $[n]$ for $\mathcal{H}\in\mathfrak{G}(\mathbf{L})$: 
\begin{equation}
\alpha\triangledown_{\mathcal{H}}\beta \quad \overset{{\scriptstyle \mathrm{def}}}{\Leftrightarrow} \quad \mathcal{H}_{\alpha}^{\beta}\in\mathfrak{G}(\mathbf{L})\text{ or }\mathcal{H}_{\beta}^{\alpha}\in\mathfrak{G}(\mathbf{L}). 
\label{eq: distinguishability relation on columns}
\end{equation}

While psychometric applications primarily focus on real matrices, the results presented here extend to complex matrices as well. Accordingly, we formulate them over $\mathbb{C}$ to streamline the proofs and enhance generalisability. Let $\mathbb{C}[\mathbf{t},\mathbf{t}^{-1}]$ be the algebra of polynomials in the variables $t_{1},t_{1}^{-1},\dots,t_{d},t_{d}^{-1}$. The non-zero monomials are the \emph{units} of this algebra, meaning they have a multiplicative inverse. This ensures that the weighting (\ref{eq: toric deformation}) is reversible. Let $\mathbb{F}:=\mathbb{C}(t_{1},\dots,t_{d})$ be the field of fractions of $\mathbb{C}[\mathbf{t},\mathbf{t}^{-1}]$ (i.e., the algebra of all ratios of polynomials in $\mathbb{C}[\mathbf{t},\mathbf{t}^{-1}]$, where division is allowed whenever the denominator is not identically zero), and take matrix entries in a common field extension of $\mathbb{F}$. For any polynomial $P$ in $\mathbb{C}[\mathbf{t},\mathbf{t}^{-1}]$, we define $\mathrm{Supp}(P)$ as the set of non-zero monomials composing $P$, with 
the associated mapping  
 $\Psi(P):=\left\{ \mathbf{e}\in\mathbb{Z}^{d}:\,\text{there exists } c_{\mathbf{e}}\in\mathbb{C}\setminus\{0\} \text{ such that } c_{\mathbf{e}}\cdot\mathbf{t}^{\mathbf{e}}\in\mathrm{Supp}(P)\right\}$. 
\begin{rem} 
\label{rem: notation for Psi}
When the polynomial $P$ is a unit in $\mathbb{C}[\mathbf{t},\mathbf{t}^{-1}]$, i.e., $|\Psi(P)|=1$, we explicitly write $\Psi(P)=:\{\Psi^{(1)}(P)\}$. Under the parameterisation (\ref{eq: toric deformation}), each product $h(\mathcal{I})$ is a monomial in the parameters $\mathbf{t}$. We then set 
$\Psi(\mathcal{I}):=\Psi^{(1)}(h(\mathcal{I}))$ for any $\mathcal{I}\in\mathfrak{G}(\mathbf{L})$ 
to simplify notation whenever no ambiguity arises. 
\end{rem}
\begin{defn} 
\label{def: keys} 
The set 
$\chi(\mathcal{I}\mid_{\alpha\beta}^{ij}):=\left\{ h(\mathcal{I})\cdot h(\mathcal{I}_{\alpha\beta}^{ij}),h(\mathcal{I}_{\alpha}^{i})\cdot h(\mathcal{I}_{\beta}^{j}),h(\mathcal{I}_{\beta}^{i})\cdot h(\mathcal{I}_{\alpha}^{j})\right\}
$ 
is called \emph{observable} if $\chi(\mathcal{I}\mid_{\alpha\beta}^{ij})\neq\{0\}$; in that case, we use the same attribute for the corresponding index set $\mathfrak{c}:=\{i,j\}\times\{\alpha,\beta\}$. Let $\mathfrak{c}_{r}:=\{i,j\}$ and $\mathfrak{c}_{c}:=\{\alpha,\beta\}$ denote the projections of $\mathfrak{c}$ onto 
$\mathcal{I}$ and 
$\mathcal{I}^{\mathtt{C}}$, respectively. 
An observable set $\chi(\mathcal{I}\mid_{\alpha\beta}^{ij})$, or the corresponding index set $\{i,j\}\times\{\alpha,\beta\}$, is called a \emph{key} if $0\notin\chi(\mathcal{I}\mid_{\alpha\beta}^{ij})$; it is called a \emph{weak key} if at least three pairs $(l,\gamma)\in\{i,j\}\times\{\alpha,\beta\}$ satisfy $h(\mathcal{I}_{\gamma}^{l})\neq0$. Finally, an observable set is called \emph{separable} (or \emph{distinguishable}) if $\left|\Psi\left(\chi(\mathcal{I}\mid_{\alpha\beta}^{ij})\setminus\{0\}\right)\right|=1$, and $\mathcal{I}\in\mathfrak{G}(\mathbf{L})$ is \emph{separable} if all observable sets $\chi(\mathcal{I}\mid_{\alpha\beta}^{ij})$ based on $\mathcal{I}$ are separable. 
\end{defn} 

Note that basic solutions (\ref{eq: separable deformations}) satisfy the separability condition. 
\begin{exa}
\label{exa: keys} 
To illustrate these definitions, we consider 
{\small 
\begin{equation}
\mathbf{L}_{\mathrm{ex}}:=
\begin{array}{c}
\begin{pmatrix}1 & 0 & 0 & 0 & 1 & 1 & 3\\
0 & 1 & 0 & 0 & 2 & 2 & 4\\
0 & 0 & 1 & 1 & 4 & 5 & 0
\end{pmatrix} \\ 
\begin{array}{ccccccc}
{\scriptstyle i} & {\scriptstyle j} & {\scriptstyle m} & {\scriptstyle \alpha} & {\scriptstyle \beta} & {\scriptstyle \gamma} & {\scriptstyle \delta}\end{array}
\end{array}
\end{equation}
}
where we explicitly label the columns and assume that no maximal minor in a paired matrix $\mathbf{R}_{\mathrm{ex}}$ identically vanishes. Setting $\mathcal{I}:=\{i,j,m\}$, we observe that $\chi(\mathcal{I}\mid_{\alpha\beta}^{ij})=\{0\}$, so the set $\{i,j\}\times\{\alpha,\beta\}$ is not observable. In contrast, the sets $\mathfrak{c}_{\beta\gamma}:=\{i,j\}\times\{\beta,\gamma\}$, $\mathfrak{c}_{\beta\delta}:=\{i,j\}\times\{\beta,\delta\}$, and $\mathfrak{c}_{\gamma\delta}:=\{i,j\}\times\{\gamma,\delta\}$ are observable. Both $\mathfrak{c}_{\beta\delta}$ and $\mathfrak{c}_{\gamma\delta}$ are keys, as $0\notin\chi(\mathcal{I}\mid_{\beta\delta}^{ij})\cup\chi(\mathcal{I}\mid_{\gamma\delta}^{ij})$. On the other hand, $\chi(\mathcal{I}\mid_{\beta\gamma}^{ij})$ is a weak key since $\Delta_{\mathbf{L}_{\mathrm{ex}}}(\{m,\beta,\gamma\})=0$. Finally, $\chi(\mathcal{I}\mid_{\alpha\delta}^{im})$ is an observable set that is not a weak key. 
\end{exa} 

The $k$-dimensional subspace of $\mathbb{C}^{n}$ generated by a full-rank matrix $\mathbf{M}\in\mathbb{C}^{k\times n}$ is parameterised by its \emph{Pl\"{u}cker coordinates}, namely, the tuple $\left(\Delta_{\mathbf{M}}(\mathcal{I}):\,\mathcal{I}\in\wp_{k}[n]\right)$ \citep[Ch. 3.1.C]{GKZ1994}. These homogeneous coordinates embed the set of $k$-dimensional subspaces of $\mathbb{C}^{n}$ (the $(k,n)$-\emph{Grassmannian}) into a projective space; indeed, a change of basis in $\mathbb{C}^{k}$, acting on $\mathbf{M}$ by left-multiplication, scales all Pl\"{u}cker coordinates by a common factor but leaves the subspace generated by $\mathbf{M}$ unchanged. 
\begin{rem} 
\label{rem: Grassmann-Plucker coordinates and factor score representations}
Pl\"{u}cker coordinates make explicit the duality between the subsets $\mathcal{I}$ and $\mathcal{I}^{\mathtt{C}}$ partitioning $[n]$, as the coordinates of a $k$-subspace---and the associated matroid (\ref{eq: matroid of non-zero minors})---can be equivalently expressed starting from its orthogonal complement. Such a duality is suited to the description of a factor model with loading matrix ${\bf\tilde{\Lambda}}=\left({\bf{\Lambda}}|\idd_{n-k}\right)$ via $k$-dimensional spaces, where $\bf{\Lambda}$ encodes common factor loadings and $\idd_{n-k}$ denotes the $(n-k)$-dimensional identity matrix (unique factor loadings). Specifically, the orthogonal complement to the subspace of $\mathbb{R}^{n}$ spanned by ${\bf\tilde{\Lambda}}$ can be represented by $\left(\idd_{k}|-\bf{\Lambda}^{\mathtt{T}}\right)$ \citep[Prop. 2.2.23]{Oxley2011}, and its Pl\"{u}cker coordinates, at complementary index sets and after compatible normalisation, coincide with those of ${\bf\tilde{\Lambda}}$ up to a sign \citep[Lemma 11.4]{lukowski2023positive}. For each $\mathcal{I}$, this sign depends solely on the partition $(\mathcal{I};\mathcal{I}^{\mathtt{C}})$; thus, it equals the corresponding sign arising from $\mathbf{R}$ expressed via its orthogonal complement. Therefore, the two signs cancel out in each minor product (\ref{eq: toric deformation}), which remains invariant under this duality. We also note that the projective invariance of Pl\"{u}cker coordinates with respect to changes of basis in $\mathbb{R}^{k}$ allows one to isolate this geometric freedom from rank-based sources of factor indeterminacy. 
\end{rem} 
The Pl\"{u}cker coordinates satisfy the three-term Grassmann-Pl\"{u}cker relations \citep[Ch. 3.1.D]{GKZ1994}
 \begin{align} 
 \hspace{-.5cm}
\Delta_{\mathbf{M}}(\mathcal{I})\cdot\Delta_{\mathbf{M}}(\mathcal{I}_{\alpha\beta}^{ij}) & = c_{1} \cdot \Delta_{\mathbf{M}}(\mathcal{I}_{\alpha}^{i})\cdot\Delta_{\mathbf{M}}(\mathcal{I}_{\beta}^{j}) + c_{2} \cdot \Delta_{\mathbf{M}}(\mathcal{I}_{\beta}^{i})\cdot\Delta_{\mathbf{M}}(\mathcal{I}_{\alpha}^{j}), 
\label{eq: three-term Grassmann-Plucker relations} \\ 
 \hspace{-.5cm}
c_{1} :=  \mathrm{sign}\left((i-j)(\alpha-\beta)(i-\beta)(\alpha-j)\right), & \, c_{2} := \mathrm{sign}\left((i-j)(\alpha-\beta)(i-\alpha)(j-\beta)\right).
\label{eq: signs for permutation}
\end{align}
The attribute ``observable'' in Definition \ref{def: keys} refers to non-trivial information in $\chi(\mathcal{I}\mid_{\alpha\beta}^{ij})$ 
derived from the Grassmann-Pl\"{u}cker relations. Such constraints can be used to extract information about 
the functions 
\begin{equation}
Y(\mathcal{I})_{\alpha\beta}^{ij}:=c_{1}c_{2}\cdot\frac{\Delta_{\mathbf{R}(\mathbf{t})}(\mathcal{I}_{\alpha}^{i})}{\Delta_{\mathbf{R}(\mathbf{t})}(\mathcal{I}_{\beta}^{i})}\cdot\frac{\Delta_{\mathbf{R}(\mathbf{t})}(\mathcal{I}_{\beta}^{j})}{\Delta_{\mathbf{R}(\mathbf{t})}(\mathcal{I}_{\alpha}^{j})}, \quad i,j\in\mathcal{I},\,\alpha,\beta\in\mathcal{I}^{\mathtt{C}}  
\label{eq: cross-section}
 \end{equation} 
referred to as \emph{extended tetrads} or $Y$-\emph{terms}; we say that $Y(\mathcal{I})_{\alpha\beta}^{ij}$ is observable if $\chi(\mathcal{I}|_{\alpha\beta}^{ij})$ is observable. From (\ref{eq: three-term Grassmann-Plucker relations}), the $Y$-terms transform as follows under changes of basis obtained by a \emph{single-index exchange}: 
\begin{align}
Y(\mathcal{I}_{\alpha}^{i})_{i\beta}^{\alpha j} & = -c_{2}\frac{\Delta_{\mathbf{R}(\mathbf{t})}(\mathcal{I})\cdot\Delta_{\mathbf{R}(\mathbf{t})}(\mathcal{I}_{\alpha\beta}^{ij})}{\Delta_{\mathbf{R}(\mathbf{t})}(\mathcal{I}_{\beta}^{i})\cdot\Delta_{\mathbf{R}(\mathbf{t})}(\mathcal{I}_{\alpha}^{j})}=-Y(\mathcal{I})_{\alpha\beta}^{ij}-1,
\label{eq: Grassmann-Plucker translation exchange, vertical}\\
Y(\mathcal{I}_{\beta}^{i})_{\alpha i}^{\beta j} & = -c_{1}\cdot\frac{\Delta_{\mathbf{R}(\mathbf{t})}(\mathcal{I}_{\alpha}^{i})}{\Delta_{\mathbf{R}(\mathbf{t})}(\mathcal{I})}\cdot\frac{\Delta_{\mathbf{R}(\mathbf{t})}(\mathcal{I}_{\beta}^{j})}{\Delta_{\mathbf{R}(\mathbf{t})}(\mathcal{I}_{\alpha\beta}^{ij})}=-\frac{1}{1+\left(Y(\mathcal{I})_{\alpha\beta}^{ij}\right)^{-1}}.
\label{eq: Grassmann-Plucker translation exchange, diagonal}
\end{align} 
For each change of basis $\mathcal{I}\mapsto\mathcal{J}:=\mathcal{I}_{\gamma}^{l}\in\mathfrak{G}(\mathbf{L})$, where $(l,\gamma)\in\mathfrak{c}\subseteq\mathcal{I}\times\mathcal{I}^{\mathtt{C}}$, we set 
 \begin{equation}
\mathfrak{c}_{\mathcal{J}}:=(\mathfrak{c}_{r})_{\gamma}^{l}\times(\mathfrak{c}_{c})_{l}^{\gamma}.
\label{eq: weak key under local exchange}
 \end{equation}

\begin{rem}
\label{rem: motivation naming keys}
Example \ref{exa: keys} clarifies the terminology ``(weak) key'': the properties of the key $\mathfrak{c}_{\gamma\delta}$ allow us to access other bases in $\mathfrak{G}(\mathbf{L})$, starting from $\mathcal{I}$, by performing a label switching $u\leftrightarrows \omega$ with $u\in\{i,j\}$ and $\omega\in\{\gamma,\delta\}$. This accessibility is 
limited to indices in $\mathfrak{c}_{\gamma\delta}$ and 
extends to transformed bases via (\ref{eq: weak key under local exchange}). This is not the case for the weak key $\mathfrak{c}_{\beta\gamma}$, as $\mathcal{I}_{\beta\gamma}^{ij}\notin\mathfrak{G}(\mathbf{L})$. Finally, (weak) keys are characterised by the existence of a basis $\mathcal{J}$ obtained from $\mathcal{I}$ through either (\ref{eq: Grassmann-Plucker translation exchange, vertical}), (\ref{eq: Grassmann-Plucker translation exchange, diagonal}), or the identity map, satisfying 
\begin{equation}
\prod_{(s,\sigma)\in\mathfrak{c}_{\mathcal{J}}}h(\mathcal{J}_{\sigma}^{s})\neq0.
\label{eq: local context for representation of 4-square}
\end{equation}
This property, which will be used later, does not hold for $\mathfrak{c}_{\alpha\delta}$ in Example \ref{exa: keys}.
\end{rem} 
It is also straightforward to verify  
\begin{equation}
Y_{\alpha\beta}^{ij}\cdot Y_{\beta\gamma}^{ij} =  -Y_{\alpha\gamma}^{ij},\quad 
Y_{\alpha\beta}^{im}\cdot Y_{\alpha\beta}^{mj} = -Y_{\alpha\beta}^{ij}.
\label{eq: associativity}
\end{equation}
In the degenerate cases where $i=j$ or $\alpha=\beta$, we set $Y_{\alpha\beta}^{ij}=-1$ for consistency. Iterating (\ref{eq: associativity}), for all $i,j,m\in\mathcal{I}$ and $\alpha,\beta,\omega\in\mathcal{I}^{\mathtt{C}}$, we obtain the decomposition 
\begin{equation}
     Y_{\alpha\beta}^{ij} = -Y_{\alpha\omega}^{ij}\cdot Y_{\omega\beta}^{ij} 
    = -Y_{\alpha\omega}^{im}\cdot Y_{\alpha\omega}^{mj}\cdot Y_{\omega\beta}^{im}\cdot Y_{\omega\beta}^{mj}.
    \label{eq: quadrilateral decomposition}
\end{equation} 

\subsection{\label{subsec: Assumptions} Assumptions}

\subsubsection{No Trivial Dependence Relations in $\mathbf{L}(\mathbf{t})$} 
\label{par: non-trivial dependence pattern} 

We assume that each column of $\mathbf{L}(\mathbf{t})$ belongs to at least one basis in $\mathfrak{G}(\mathbf{L})$. Thus, for every $\mathcal{B}\in\mathfrak{G}(\mathbf{L})$ and $\alpha\notin\mathcal{B}$, there exists $\mathcal{A}\in\mathfrak{G}(\mathbf{L})$ such that $\alpha\in\mathcal{A}$ and, by (\ref{eq: symmetric exchange relation}), at least one $i\in\mathcal{B}$ for which $\Delta_{\mathbf{L}(\mathbf{t})}(\mathcal{B}_{\alpha}^{i})\neq0$. 
\emph{Null} columns do not satisfy these conditions and thus contribute no information about the corresponding rows in $\mathbf{R}(\mathbf{t})$ in (\ref{eq: Cauchy-Binet expansion}); 
they are therefore excluded without loss of relevant information. Dually, we assume that for each $\mathcal{I}\in\mathfrak{G}(\mathbf{L})$ and $i\in\mathcal{I}$, there exists at least one $\alpha\in\mathcal{I}^{\mathtt{C}}$ such that $\Delta_{\mathbf{L}(\mathbf{t})}(\mathcal{I}_{\alpha}^{i})\neq0$, ensuring that no $i\in[n]$ belongs to \emph{every} basis. 

Some of the results presented here extend those in \citet{Angelelli2025}, allowing for greater sparsity in the structural matrix $\mathbf{L}(\mathbf{t})$ to explore minimal conditions for separability. In what follows, we impose the dimensional bound $\max\{n-k,k\}\geq 5$ to exclude cases where deviations from separability arise due to small values of $n$, which in turn restrict $k$; see, e.g., \citet[Example 4]{Angelelli2025}. 

\subsubsection{Generic $Y$-terms}
\label{subsubsec: no Y=-1} 

The matrix $\mathbf{R}(\mathbf{t})$ is assumed to be generic, making the dependence pattern defined by terms $h(\mathcal{I})=0$ fully determined by $\wp_{k}[n]\setminus\mathfrak{G}(\mathbf{L})$ and $Y$-terms generic as well in the following sense: 
\begin{assumption}
\label{claim: no Y=-1}
No $Y$-term (\ref{eq: cross-section}) vanishes identically as a function of $\mathbf{t}$; specifically, $\Delta_{\mathbf{R}(\mathbf{t})}(\mathcal{I})\neq0$ for all $\mathcal{I}\in\wp_{k}[n]$ at $\mathbf{t}=\mathbf{1}$ and for a generic choice of $\mathbf{t}$. 
\end{assumption}
From (\ref{eq: Grassmann-Plucker translation exchange, vertical}), this assumption implies that $Y_{\alpha\beta}^{ij}\notin\{0,-1\}$
for all $\mathcal{I}\in\wp_{k}[n]$, $i,j\in\mathcal{I}$, and $\alpha,\beta\in\mathcal{I}^{\mathtt{C}}$ with $i\neq j$ and $\alpha\neq\beta$. We still allow $Y_{\alpha\beta}^{ij}=-1$ only for the degenerate cases where $i=j$ or $\alpha=\beta$. By applying a permutation matrix $\mathbf{D}$ mapping $\mathcal{I}$ to $[k]$ via (\ref{eq: gauge invariance}) and a change of basis in $\mathbb{C}^{k}$, the column span of $\mathbf{D}\cdot\mathbf{R}(\mathbf{1})$ can be represented by $\left(\idd_{k}|\mathbf{r}\right)^{\mathtt{T}}$, with $\mathbf{r}\in\mathbb{C}^{k\times (n-k)}$ \citep[Sec. 0.3.4]{Horn2012}. Then, the excluded value $Y(\mathcal{I})_{\alpha\beta}^{ij} = -1$ would yield a vanishing $2\times 2$ minor of $\mathbf{r}$ by (\ref{eq: Grassmann-Plucker translation exchange, vertical}), thereby mirroring tetrad constraints in our setting. In this sense, our work extends tetrad analysis by incorporating observability (as per Definition \ref{def: keys}) and explicitly formalising the role of the reference basis $\mathcal{I}\in\mathfrak{G}(\mathbf{L})$. 

\subsubsection{Local Separability Condition and its Relation to Graph Planarity} 
\label{subsubsec: local separability condition and its relation to graph planarity} 

Before stating the last condition, we fix the following:  
\begin{defn}
\label{def: adjoint pair of null-sets}
Every $\mathcal{A}\subseteq\mathcal{I}$ and $\mathcal{H}\subseteq\mathcal{I}^{\mathtt{C}}$, $\mathcal{I}\in\wp_{k}[n]$, generate a pair of dual sets 
\begin{equation} 
\mathcal{N}_{\mathcal{I};\mathcal{A}} := \left\{ \gamma\in\mathcal{I}^{\mathtt{C}}:\,h(\mathcal{I}_{\gamma}^{i})=0 \text{ for all } i\in\mathcal{A}\right\},
\;
\mathcal{N}^{\mathcal{I};\mathcal{H}} := \left\{ m\in\mathcal{I}:\,h(\mathcal{I}_{\alpha}^{m})=0 \text{ for all }\alpha\in\mathcal{H}\right\}
.
\label{eq: sets of null arrays}
\end{equation}
To simplify notation, we omit the subscript $\mathcal{I}$ when no ambiguity arises and define   
$\mathcal{N}(\mathcal{A};\mathcal{H}):=\left\{ (m,\omega):\,\omega\in\mathcal{N}_{\mathcal{I};\mathcal{A}}\text{ or }m\in\mathcal{N}^{\mathcal{I};\mathcal{H}}\right\}$.  
Given a key $\mathfrak{c}=\mathfrak{c}_{r}\times\mathfrak{c}_{c}$, 
we denote the associated set $\mathcal{N}(\mathfrak{c}_{r};\mathfrak{c}_{c})$
as $\mathcal{N}(\mathfrak{c})$. 
The sets $\mathcal{N}_{\mathcal{A}}$ and $\mathcal{N}^{\mathcal{H}}$ in (\ref{eq: sets of null arrays}) are dual in the sense that they satisfy the adjunction property 
$\mathcal{H}\subseteq\mathcal{N}_{\mathcal{A}}\Leftrightarrow\mathcal{A}\subseteq\mathcal{N}^{\mathcal{H}}$. 
Thus, $\mathcal{A}\subseteq\mathcal{I}$ and $\mathcal{H}\subseteq\mathcal{I}^{\mathtt{C}}$ are adjoint if $\mathcal{H}\subseteq\mathcal{N}_{\mathcal{A}}$. 
\end{defn}
\begin{rem}
\label{rem: invariance null arrays} For every $\mathcal{A}\subseteq\mathcal{I}$ and $i\in\mathcal{A}$, whenever $h(\mathcal{I}_{\alpha}^{i})\neq0$, the set $\mathcal{N}_{\mathcal{I};\mathcal{A}}$ in (\ref{eq: sets of null arrays}) remains invariant under the change of basis $\mathcal{I}\mapsto\mathcal{J}:=\mathcal{I}_{\alpha}^{i}$. Indeed, $\alpha\notin\mathcal{N}_{\mathcal{I};\mathcal{A}}$, and for all $\beta\in\mathcal{N}_{\mathcal{I};\mathcal{A}}$, we have $h(\mathcal{J}_{\beta}^{\alpha})=h(\mathcal{I}_{\beta}^{i})=0$ by definition. For the other indices $j\in\mathcal{A}^{i}$, we find $h(\mathcal{J}_{\beta}^{j})=h(\mathcal{I}_{\alpha\beta}^{ij})=0$, so  
$\mathcal{N}_{\mathcal{I};\mathcal{A}}\subseteq\mathcal{N}_{\mathcal{J};\mathcal{A}_{\alpha}^{i}}$. By symmetry under the exchanges $i\leftrightarrows \alpha$ and $\mathcal{I}\leftrightarrows \mathcal{J}$, we deduce that 
$\mathcal{N}_{\mathcal{I};\mathcal{A}}=\mathcal{N}_{\mathcal{J};\mathcal{A}_{\alpha}^{i}}$.
Dually, for every $\mathcal{H}\subseteq\mathcal{I}^{\mathtt{C}}$ and $\alpha\in\mathcal{H}$, the set $\mathcal{N}^{\mathcal{I};\mathcal{H}}$ is invariant under the change of basis $\mathcal{I}\mapsto\mathcal{I}_{\alpha}^{m}$ whenever $h(\mathcal{I}_{\alpha}^{m})\neq0$. In particular, transformations (\ref{eq: weak key under local exchange}) preserve the set $\mathcal{N}(\mathfrak{c})$ and map a (weak) key to a new (weak) key, at least one of which satisfies (\ref{eq: local context for representation of 4-square}). 
\end{rem} 
\begin{defn} 
\label{def: explained key} A key $\mathfrak{c}$ is called \emph{planar} 
if the following conditions simultaneously hold: 
\begin{equation} 
\mathcal{N}^{\mathfrak{c}_{c}}=\mathcal{I}\setminus\mathfrak{c}_{r}, \quad \mathcal{N}_{\mathfrak{c}_{r}}=\mathcal{I}^{\mathtt{C}}\setminus\mathfrak{c}_{c}.
\label{eq: equivocal condition}
\end{equation}
Otherwise, it is referred to as \emph{non-planar}. 
\end{defn} 
\begin{rem} 
\label{rem: planarity preservation}
For a key $\mathfrak{c}=\{i_{1},i_{2}\}\times\{\alpha_{1},\alpha_{2}\}$ in a basis $\mathcal{I}$ and $(m,\omega)\in\mathcal{N}(\mathfrak{c})$ such that $\mathcal{J}:=\mathcal{I}_{\omega}^{m}\in\mathfrak{G}(\mathbf{L})$, the change of basis $\mathcal{I}\mapsto\mathcal{J}$ preserves planarity: specifically, $\mathfrak{c}$ remains a key in $\mathcal{J}$, as $\Delta_{\mathbf{L}(\mathbf{t})}(\mathcal{I}_{\omega\alpha_{w}}^{mi_{u}})\neq0$ for all $u,w\in\{1,2\}$, and $\Delta_{\mathbf{L}(\mathbf{t})}(\mathcal{I}_{\omega\alpha_{1}\alpha_{2}}^{mi_{1}i_{2}})\neq0$. Moreover, if $\mathfrak{c}$ is planar in $\mathcal{I}$, then by definition $h(\mathcal{I}_{\gamma}^{i_{u}})=h(\mathcal{I}_{\alpha_{w}}^{g})=0$ for all $u,w\in\{1,2\}$, $g\in\mathcal{I}\setminus\mathfrak{c}_{r}$, and $\gamma\in\mathcal{I}^{\mathtt{C}}\setminus\mathfrak{c}_{c}$. This yields $h(\mathcal{J}_{m}^{i_{u}})=h(\mathcal{I}_{\omega}^{i_{u}})=0$ when $g=m$, $h(\mathcal{J}_{\alpha_{w}}^{\omega})=h(\mathcal{I}_{\alpha_{w}}^{m})=0$ when $\gamma=\omega$, and hence $h(\mathcal{I}_{\omega\gamma}^{mi_{u}})=h(\mathcal{I}_{\omega\alpha_{w}}^{mg})=0$ in the remaining cases. Hence, $\mathfrak{c}$ remains planar in $\mathcal{J}$. Applying the inverse change of basis $\mathcal{J}\mapsto\mathcal{I}$ yields the converse: if $\mathfrak{c}$ is planar in $\mathcal{J}$, then it is planar in $\mathcal{I}$. Also note that 
$\Delta_{\mathbf{L}(\mathbf{t})}(\mathcal{I}_{\omega\alpha_{w}\gamma}^{mi_{1}i_{2}})=\Delta_{\mathbf{L}(\mathbf{t})}(\mathcal{I}_{\omega\alpha_{1}\alpha_{2}}^{mgi_{u}})=0$ for all $g\neq m$ and $\gamma\neq\omega$, as each such minor contains two proportional columns or rows, respectively. 
\end{rem}
\begin{rem}
\label{rem: analogy with Kuratowski}
In our context, the attribute ``planar'' is opposed to the term ``non-planar,'' which, in turn, is motivated by a correspondence between a non-planar key and one of the graphs $K_{3,3}$ or $K_{5}$ (defined below) that characterise obstructions to graph planarity in Kuratowski's theorem \citep[Thm. 2.3.8]{Oxley2011}. Before formalising this correspondence, we give an intuition of how index sets can be seen as nodes of a graph, connected via $\mathfrak{G}(\mathbf{L})$. Assume the existence of a non-planar key $\mathfrak{c}:=\{i_{1},i_{2}\}\times\{\alpha_{1},\alpha_{2}\}$, allowing us to find an index, say $i_{3}\in\mathcal{I}$, with $\mathcal{I}_{\alpha_{1}}^{i_{3}}\in\mathfrak{G}(\mathbf{L})$. 

First, suppose that $0\in\chi(\mathcal{I}\mid_{\alpha_{1}\alpha_{2}}^{i_{u}i_{3}})$ for some $u\in\{1,2\}$, say $h(\mathcal{J})=0$ where $\mathcal{J}:=\mathcal{I}_{\alpha_{1}\alpha_{2}}^{i_{2}i_{3}}$. We define $\mathfrak{g}_{r}:=\{i_{1},\alpha_{1},\alpha_{2}\}$, $\mathfrak{g}_{c}:=\{i_{2},i_{3}\}$, 
and introduce the following correspondence $k_{0}$: 
\begin{equation} 
k_{0}(u,\omega):=\mathcal{J}\setminus\{u\}\cup\{\omega\},\quad k_{0}(u,\mathfrak{g}_{c})=k_{0}(\mathfrak{g}_{c},u):=\mathcal{J}\setminus\mathfrak{g}_{r}\cup\mathfrak{g}_{c}\cup\{u\},\quad u,\omega\in\mathfrak{g}_{r}\cup\mathfrak{g}_{c}. 
\label{eq: obstruction planarity, K_33} 
\end{equation} 
For all $u,w\in\mathfrak{g}_{r}\cup\mathfrak{g}_{c}\cup\{\mathfrak{g}_{c}\}$, 
we say that $u$ and $w$ are related if either $k_{0}(u,w)$ or $k_{0}(w,u)$ yields a basis in $\mathfrak{G}(\mathbf{L})$ via (\ref{eq: obstruction planarity, K_33}). It is directly verified that this relation between elements of $\mathfrak{g}_{r}\cup\mathfrak{g}_{c}\cup\{\mathfrak{g}_{c}\}$ corresponds to $K_{3,3}$, the complete bipartite graph in which three nodes (here, the elements of $\mathfrak{g}_{r}$) are each connected to three others (the elements of $\mathfrak{g}_{c}\cup\{\mathfrak{g}_{c}\}$). 
The same argument holds, by a change of basis, if $\mathcal{J}:=\mathcal{I}_{\alpha_{2}}^{i_{3}}\notin\mathfrak{G}(\mathbf{L})$. In this case, we set $\mathfrak{g}_{r}:=\{i_{1},i_{2},\alpha_{2}\}$, $\mathfrak{g}_{c}:=\{\alpha_{1},i_{3}\}$, and proceed as before. 

When $0\notin \chi(\mathcal{I}\mid_{\alpha_{1}\alpha_{2}}^{i_{1}i_{3}})\cup \chi(\mathcal{I}\mid_{\alpha_{1}\alpha_{2}}^{i_{2}i_{3}})$, an analogous correspondence can be established with $K_{5}$, the complete graph with five nodes, by associating pairs of indices in $\mathfrak{g}_{r}\cup\mathfrak{g}_{c}$ with bases in $\mathfrak{G}(\mathbf{L})$ as follows:   
\begin{equation}
k_{1}(\alpha_{1},\alpha_{2}):=\mathcal{I},
\quad k_{1}(i_{u},\alpha_{w}):=\mathcal{I}_{\alpha_{w}}^{i_{u}},
\quad k_{1}(i_{s},i_{u}):=\mathcal{I}_{\alpha_{1}\alpha_{2}}^{i_{s}i_{u}},\quad s,u\in[3] ,\, s\neq u,\, w\in\{1,2\}  
\label{eq: obstruction planarity, K_5}
\end{equation}
so that all such pairs with distinct components are related. 

We can encompass both relations within a larger graph based on index sets. Let $N:=\wp_{0}[n]\cup\wp_{1}[n]\cup\wp_{2}[n]$ and take $\mathcal{H}\in\wp_{k+2}[n]$. For $\hat{u},\hat{w}\in N$, we use their symmetric difference $\hat{u}\Delta\hat{w}$ to define the relation 
\begin{equation} 
\hat{u}\underset{{\scriptstyle \mathcal{H}}}{\sim}\hat{w} \quad \overset{{\scriptstyle \mathrm{def}}}{\Leftrightarrow} \quad \mathcal{H}\setminus(\hat{u}\Delta\hat{w})\in\mathfrak{G}(\mathbf{L}),\quad\hat{u},\hat{w}\in N.
\label{eq: joint encoding reducing to matroid}
\end{equation} 
Setting $\mathcal{H}:=\mathcal{I}_{\alpha_{1}\alpha_{2}}$, we extend the relation defined by $k_{0}$ in (\ref{eq: obstruction planarity, K_33}) whether $\mathcal{I}_{\alpha_{1}\alpha_{2}}^{i_{2}i_{3}}\notin\mathfrak{G}(\mathbf{L})$ by considering  
\begin{equation}
N_{1}:=\left\{ \{i_{1},i_{2}\},\{\alpha_{1},\alpha_{2}\},\{i_{1},i_{3}\}\right\} ,\quad N_{2}:=\left\{ \{\alpha_{1},i_{1}\},\{\alpha_{2},i_{1}\},\emptyset\right\} 
\label{eq: joint encoding for K33}
\end{equation} 
since $\hat{u}\underset{{\scriptstyle \mathcal{H}}}{\sim}\hat{w}$ for all $\hat{u}\in N_{1}$ and $\hat{w}\in N_{2}$, thus including the graph $K_{3,3}$. Similarly, we recover the relation induced by $k_{1}$ in (\ref{eq: obstruction planarity, K_5}) and the associated graph $K_{5}$ when $\mathcal{I},\mathcal{I}_{\alpha_{w}}^{i_{u}},\mathcal{I}_{\alpha_{1}\alpha_{2}i_{u}}^{i_{1}i_{2}i_{3}}\in\mathfrak{G}(\mathbf{L})$ for all $u\in[3]$ and $w\in\{1,2\}$, noting that $\hat{u}\underset{{\scriptstyle \mathcal{H}}}{\sim}\hat{w}$ for all distinct $\hat{u},\hat{w}\in\left\{\{i_{1}\},\{i_{2}\},\{i_{3}\},\{\alpha_{1}\},\{\alpha_{2}\}\right\}$.  
\end{rem}
\begin{assumption}
\label{claim: existence 3-key} 
There exists a basis $\mathcal{I}\in\mathfrak{G}(\mathbf{L})$ that exhibits a non-planar key. 
\end{assumption}

\section{Main Contribution and Positioning within the Literature}
\label{sec: theoretical framework}

The results in the next sections identify Assumption \ref{claim: existence 3-key} as a local and structural property enabling global identifiability. This assumption is local since it refers to a substructure of $\mathfrak{G}(\mathbf{L})$ induced by a submatrix of $\mathbf{L}(\mathbf{1})$; it is structural since it does not depend on the set function $\Psi$ in Remark \ref{rem: notation for Psi} and the specific values of non-zero entries in $\mathbf{L}(\mathbf{1})$ at given $\mathfrak{G}(\mathbf{L})$. These findings are formalised in Theorem \ref{thm: equality phases monomial}, which follows from Theorem \ref{thm: equality phases monomial generalised} presented in Section \ref{sec: propagation of separability between different bases}. 
\begin{thm}
\label{thm: equality phases monomial} 
Let $\mathbf{L}(\mathbf{t}),\mathbf{R}(\mathbf{t})^{\mathtt{T}}$ be two $k\times n$ matrices of complex functions of $d$ indeterminates $\mathbf{t}$, with $\mathfrak{G}(\mathbf{L}(\mathbf{t})) = \mathfrak{G}(\mathbf{L}(\mathbf{1}))$ and $\mathfrak{G}(\mathbf{R}(\mathbf{1}))=\wp_{k}[n]$ for generic $\mathbf{t}$, and let $\Psi:\,\mathfrak{G}(\mathbf{L}(\mathbf{1}))\longrightarrow\mathbb{Z}^{d}$ be a mapping satisfying 
\begin{equation}
\Delta_{\mathbf{L}(\mathbf{t})}(\mathcal{I})\cdot\Delta_{\mathbf{R}(\mathbf{t})}(\mathcal{I})= g_{\mathcal{I}} \cdot \mathbf{t}^{\Psi(\mathcal{I})},\quad \mathcal{I}\in\mathfrak{G}(\mathbf{L}(\mathbf{1})), \, g_{\mathcal{I}}\in\mathbb{C}\setminus\{0\}. 
\label{eq: monomial terms of Cauchy-Binet expansion}
 \end{equation}
When $\max\{n-k,k\}\geq 5$, Assumption \ref{claim: existence 3-key} guarantees the existence of an element $\mathbf{m_{0}}\in\mathbb{Z}^{d}$ and a mapping $\psi:\,[n]\longrightarrow\mathbb{Z}^{d}$ such that 
\begin{equation}
\Delta_{\mathbf{L}(\mathbf{t})}(\mathcal{I})\cdot\Delta_{\mathbf{R}(\mathbf{t})}(\mathcal{I})=\mathbf{t}^{\mathbf{m}_{0}}\cdot\Delta_{\mathbf{L}(\mathbf{1})}(\mathcal{I})\cdot\Delta_{\mathbf{R}(\mathbf{1})}(\mathcal{I})\cdot\prod_{\alpha\in\mathcal{I}}\mathbf{t}^{\psi(\alpha)},\quad\mathcal{I}\in\wp_{k}[n].
\label{eq: separability of set function}
\end{equation}
\end{thm} 
This result applies to any mapping from $\mathfrak{G}(\mathbf{L})$ to a given list of $|\mathfrak{G}(\mathbf{L})|$ monomials, accounting for permutations acting via (\ref{eq: gauge invariance}). 
Thus, under Assumption \ref{claim: existence 3-key}, the pair $(\mathbf{L}(\mathbf{t}),\mathbf{R}(\mathbf{t}))$ induces the same expansion (\ref{eq: Cauchy-Binet expansion}) as the pair $(\mathbf{L}(\mathbf{1})\cdot\mathrm{diag}(\mathbf{t}^{\psi(\alpha)})_{\alpha\in[n]}, \mathbf{R}(\mathbf{1}))$, up to a common unit $\mathbf{t}^{\mathbf{m}_{0}}$ that is irrelevant for Pl\"{u}cker coordinates but plays a role in 
representing matrices encoding different score realisations (see Remark \ref{rem: self-compatibility score matrices}). Theorem \ref{thm: equality phases monomial} should also be interpreted in light of Remark \ref{rem: analogy with Kuratowski}, as a non-planar key prevents non-separability, just as the existence of a subgraph 
equivalent to $K_{3,3}$ or $K_{5}$ obstructs planarity. 

Finally, we show in Theorem \ref{thm: retrieval of canonical forms} that a weaker condition than 
Assumption \ref{claim: existence 3-key} 
suffices to retrieve a ``\emph{canonical}'' matricial form for $\mathbf{R}$ and $\mathbf{L}$ derived from the list of matrix products (\ref{eq: monomial terms of Cauchy-Binet expansion}), given a reference basis $\mathcal{I}\in\mathfrak{G}(\mathbf{L})$. 
The rest of the section discusses how these results relate to specific forms of uncertainty in factor modelling, IV estimation, and ambiguity---as framed by order-theoretic inconsistency. 

\subsection{Relation to Uncertainty in Instrumental Variable Estimation}
\label{subsec: relation to uncertainty in IV estimation}

IVs are a fundamental tool for addressing endogeneity---that is, deviations from the orthogonality condition between explanatory variables and error terms. Specifically, IVs $\mathbf{Z}$ are correlated with the endogenous variables $\mathbf{X}$ but uncorrelated with the error term; 
IV methods also admit a causal formulation \citep{Pearl2009}. Algebraic expressions for partial correlations via minor ratios support IV generalisation and identifiability analysis across a broad class of graphical models \citep{Brito2002,Weihs2018}. 

IVs are also useful in factor analysis, where variables are partitioned into scaling (or reference) 
and non-scaling 
ones \citep{bollen2024model}. Once factor loadings and intercepts for scaling variables are fixed, the factor model can be reduced 
by algebraic elimination of latent factors 
\citep[Eqs. (2)-(3)]{bollen2024model}. Indicators acting as IVs are then selected to address error-variable correlations generated in the previous step. 
These methods include test statistics for instrument validity in over-identified cases---where the number of IVs $n$ exceeds the number of endogenous variables $k$. 

A consequence of over-identification is the non-uniqueness of estimators from various just-identified submodels, which can be aggregated via a weight matrix $\mathbf{W}\in\mathbb{R}^{k\times n}$ as in \citet[Prop. 1]{rotemberg1983instrument}: 
\begin{equation}
\hat{\beta} = \sum_{\mathcal{I}\in\wp_{k}[n]} \alpha_{\mathcal{I}}\cdot \hat{\beta}|_{\mathcal{I}}, \quad  
\hat{\beta}|_{\mathcal{I}} = (\mathbf{Z}|_{\mathcal{I}}^{\mathtt{T}}\cdot \mathbf{X})^{-1}\cdot \mathbf{Z}|_{\mathcal{I}}^{\mathtt{T}}\cdot \mathbf{y}, \quad  
\alpha_{\mathcal{I}} = \frac{\Delta_{\mathbf{W}}(\mathcal{I})\cdot \Delta_{\mathbf{Z}^{\mathtt{T}}\cdot \mathbf{X}}(\mathcal{I})}{\det(\mathbf{W}\cdot \mathbf{Z}^{\mathtt{T}}\cdot \mathbf{X})} 
    \label{eq: IV model} 
\end{equation} 
assuming $\det(\mathbf{W}\cdot \mathbf{Z}^{\mathtt{T}}\cdot \mathbf{X}) \cdot \det(\mathbf{Z}|_{\mathcal{I}}^{\mathtt{T}}\cdot \mathbf{X}) \neq 0$ for every $\mathcal{I}\in\wp_{k}[n]$. 
Here, $\hat{\beta}|_{\mathcal{I}}$ is the estimator for the just-identified model defined by instruments indexed by elements of $\mathcal{I}$. In this framework, over-identification introduces estimator arbitrariness under misspecification, which can be examined by varying weights in the decomposition (\ref{eq: IV model}) \citep[Sec. II]{rotemberg1983instrument}. This expression explicitly mirrors (\ref{eq: Cauchy-Binet expansion}), with each term reflecting the contribution of a just-identified submodel to the aggregated estimate. In our setting, a factor $\mathbf{W}(\mathbf{t})$ modulates the relative contributions of the different just-identified models; weights in (\ref{eq: toric deformation}) rescale these contributions without changing which submodels enter the combination, as the independence structure $\mathfrak{G}(\mathbf{L})$ is preserved.

\subsection{Relation to Factor Score Indeterminacy and Tetrad Analysis}
\label{subsec: relation to structural equations under partial information and factor indeterminacy}

Different forms of indeterminacy may affect factor models, meaning that factor scores cannot be uniquely reconstructed from observed variables and the covariance structure. Even when common and unique factor loadings are given, multiple solutions for a factor model can arise due to a rank condition \citep[Eqs. (5)-(6)]{steiger1979factor}. This uncertainty source has a geometric nature \citep{heermann1964geometry} with practical effects in psychological measurements, as indeterminacy from the space of factor score solutions represents a critical element affecting the analysis of conceptual constructs \citep{Rigdon2019}. 

Recently, alternatives to classical factor models have been proposed; relaxing assumptions related to the sparsity pattern of the matrix of (both common and unique) loadings allows defining novel optimisation problems to address both indeterminacy and numerical issues \citep{Elden2019}. Such approaches suggest that examining matrix decompositions and the geometry of solution spaces in greater generality can benefit both the analysis of factor score indeterminacy and the exploration of variant models. This approach is in line with the rationale underlying the present work, which exploits all information available from a structural (e.g., loading) matrix, encoding not only the sparsity of $\mathbf{L}$ but also higher-dimensional contributions via the independence structure $\mathfrak{G}(\mathbf{L})$.

In this regard, the notion of observability derived from $\mathfrak{G}(\mathbf{L})$ is central, as it acts as a filter to access partial information about the relations satisfied by the $Y$-terms. Note that previous studies on factor indeterminacy introduced alternative definitions of observability \citep{krijnen1998conditions}. Even that definition relies on the independence structure in the loading matrix $\bf{\Lambda}$, along with properties of the unique factors' covariance matrix. These conditions characterise cases where factors can be determined through predictors \citep[Result 1]{krijnen1998conditions}. The formalism we propose---together with its associated notion of observability---underpins \emph{compatibility conditions} for matrix representations involving multiple score realisations, as detailed below. 

\begin{rem}
\label{rem: self-compatibility score matrices} 

The issue of factor indeterminacy has been dealt with even in relation to the extension of a model through the inclusion of additional variables. In such cases, compatibility criteria can be established 
\citep[Thm. 2]
{mulaik1978effect} to address misalignments between sets of additional variables drawn from distinct domains or universes of content. Such misalignment also relates to a type of contextual effect; specifically, it may arise even when reference variables are shared across collections that do not meet compatibility conditions---see \citet[Sec. 5]{mulaik1978effect} for further details. 

We begin detailing indeterminacy, compatibility, and contextuality within our framework by extending it to 
a generic matrix $\mathbf{X}\in\mathbb{R}^{n\times s}$ with $s\neq k$. Indeed, when $\mathbf{L}\cdot \mathbf{X}\in\mathbb{R}^{k\times s}$ has rank $k\leq s$, its row space admits a representation via its Pl\"{u}cker coordinates. Furthermore, the expansion (\ref{eq: Cauchy-Binet expansion}) applies to each $\mathcal{H}\in \wp_{k}[s]$: 
\begin{equation} 
    \Delta_{\mathbf{L}(\mathbf{t}) \cdot \mathbf{X}(\mathbf{t})}(\mathcal{H}) = \det(\mathbf{L}(\mathbf{t}) \cdot \mathbf{X}(\mathbf{t})|_{\mathcal{H}}) = \sum_{\mathcal{I}\in\wp_{k}[n]} \Delta_{\mathbf{L}(\mathbf{t})}(\mathcal{I}) \cdot \Delta_{\mathbf{X}(\mathbf{t})|_{\mathcal{H}}}(\mathcal{I}),\quad \mathcal{H}\in \wp_{k}[s]  
    \label{eq: higher-order Plucker coordinates expansion} 
\end{equation} 
where $\mathbf{X}(\mathbf{t})|_{\mathcal{H}}$ is the submatrix of $\mathbf{X}(\mathbf{t})$ formed by the columns indexed by $\mathcal{H}$, with explicit dependence on $\mathbf{t}$. This form can be specified to support dual representations of factor models. In one perspective, $k$ is the number of observed variables, and $\mathbf{X}$ encodes $s$ realisations of factor scores. The dual specification takes $k$ as the number of common factors, capturing factor indeterminacy due to a non-trivial kernel---a rank-$k$ subspace of $\mathbb{R}^{n}$ orthogonal to the $(n-k)$-dimensional row span of the block matrix ${\bf\tilde{\Lambda}}:=({\bf\Lambda}|\idd_{n-k})$ of common factor loadings and uniquenesses. This kernel generates differences between solutions of the factor equation \citep[Eq. (6)]{Rigdon2019}, extending to matrices $\mathbf{F}\in\mathbb{R}^{n\times N}$ of $N\geq k$ score realisations \citep[p. 167]{Elden2019}. For each $\mathcal{H}\in\wp_{k}[N]$ and any two solutions $\mathbf{F}_{1}$ and $\mathbf{F}_{2}$, the columns of $\mathbf{F}_{1}|_{\mathcal{H}}-\mathbf{F}_{2}|_{\mathcal{H}}$, when linearly independent, span the same subspace as the rows of $\mathbf{L}:=(\idd_{k}|-\bf{\Lambda}^{\mathtt{T}})$, the ``dual'' representation of the structural matrix discussed in Remark \ref{rem: Grassmann-Plucker coordinates and factor score representations}. Thus, working in the space of such score differences provides an invariant for this family of solutions via the Pl\"{u}cker coordinates of $\mathbf{L}$. Any generic matrix $\mathbf{X}\in\mathbb{R}^{n\times k}$---e.g., external variables or solutions to alternative factor models---may align differently with factor scores across solutions, 
yet their coupling with the subspace of solution differences, represented by $\mathbf{L}$, remains invariant, generalising to $s\geq k$ realisations via (\ref{eq: higher-order Plucker coordinates expansion}). 
While any term multiplying all summands in (\ref{eq: higher-order Plucker coordinates expansion}) is immaterial for each given $\mathcal{H}$, such weights entail compatibility constraints when distinct $k$-subsets of realisations from $\mathbf{X}(\mathbf{t})$ are analysed together, as we now discuss. 

Although the preceding argument offers a description suited to the intrinsic indeterminacy arising from rank conditions, it presumes knowledge of the structural matrix, which itself may be affected by model uncertainty. Indeed, ``true'' factor scores are unobserved (even assuming they can be meaningfully defined), and uncertainty may affect any pair of $k$ score realisations extracted from two solutions, as well as their difference. As above, this uncertainty carries over to the structural matrix $\mathbf{L}$, even when its qualitative pattern $\mathfrak{G}(\mathbf{L})$ is assumed known. The multiplicative parameterisation (\ref{eq: toric deformation}) accommodates different interpretations of such uncertainty---from freely specifiable scaling across realisations to sensitivity analysis probing the structural conditions implied by $\mathfrak{G}(\mathbf{L})$, as discussed in the next section. Theorem \ref{thm: equality phases monomial} provides a compatibility condition linking these weighting parameters 
to the scaling of $s>k$ realisations encoded in $\mathbf{X}(\mathbf{t})$. With Assumption \ref{claim: no Y=-1} satisfied for any admissible selection $\mathbf{X}(\mathbf{t})|_{\mathcal{H}}$ with $\mathcal{H}\in\wp_{k}[s]$, the differences $\Psi_{\mathcal{H}}\left(h(\mathcal{I}_{\alpha}^{i})\right)-\Psi_{\mathcal{H}}\left(h(\mathcal{I})\right)\in\mathbb{Z}^{d}$ relative to the subset $\mathcal{H}$ for $\mathcal{I},\mathcal{I}_{\alpha}^{i}\in\mathfrak{G}(\mathbf{L})$, with corresponding weight ratios $\mathbf{t}^{\Psi_{\mathcal{H}}\left(h(\mathcal{I}_{\alpha}^{i})\right)}/\mathbf{t}^{\Psi_{\mathcal{H}}\left(h(\mathcal{I})\right)}$, are key to assessing the identifiability of the factor model representation. For each fixed $\mathcal{H}$, the proof of Theorem \ref{thm: equality phases monomial generalised} shows that, under separability, these differences enable the construction of a vector of scaling exponents $(\psi_{\mathcal{H}}(1),\dots,\psi_{\mathcal{H}}(n))$ yielding a diagonal representation of the minor products, up to a common monomial factor, with $\delta_{\mathcal{H}}(\mathbf{t}) :=\mathrm{diag}(\mathbf{t}^{\psi_{\mathcal{H}}(1)},\dots, \mathbf{t}^{\psi_{\mathcal{H}}(n)})$. The independence of such weight ratios on the basis $\mathcal{I}\in\mathfrak{G}(\mathbf{L})$, entailed by separability (Lemma \ref{lem: independence on base-set}), reflects consistency in comparing the contributions of the $i$-th and $\alpha$-th dimensions, ensuring that these component-specific comparisons do not depend on the context---i.e., on the remaining elements of the basis $\mathcal{I}$. In this sense, when separability fails, contextual inconsistency in dimension-wise scaling may emerge as a consequence of the factor structure geometry, made explicit by the proposed subspace representation. In Section \ref{subsec: counterexamples and relations to choice theory}, we address such contextual effects of basis selection on dimension-wise comparisons through reversals in orderings induced by weights. 

When $k<s$, we can examine the independence of such weight ratios derived from $\Psi_{\mathcal{H}}$ even across realisations by allowing $\mathcal{H}$ to vary over $\wp_{k}[s]$. Under separability, a common diagonal matrix $\delta(\mathbf{t})$ can be chosen if these exchange differences are also independent of $\mathcal{H}$. In this case, each contribution to the $\mathcal{H}$-indexed Pl\"{u}cker coordinate in (\ref{eq: higher-order Plucker coordinates expansion}) takes the form $\Delta_{\mathbf{L}(\mathbf{1})\cdot \delta(\mathbf{t})}(\mathcal{I})\cdot\Delta_{\mathbf{R}(\mathbf{1})}(\mathcal{I})\cdot\mathbf{t}^{\mathbf{m_{0}}(\mathcal{H})}$, with $\mathbf{R}(\mathbf{1})=\mathbf{X}(\mathbf{1})|_{\mathcal{H}}$, as implied by (\ref{eq: separable deformations}) and Theorem \ref{thm: equality phases monomial} under Assumption \ref{claim: existence 3-key}. The common unit $\mathbf{t}^{\mathbf{m_{0}}(\mathcal{H})}$ is independent of $\mathcal{I}$, but it may vary with $\mathcal{H}$. By fixing any $\mathcal{I}\in\mathfrak{G}(\mathbf{L})$ and choosing a representative under the transformations (\ref{eq: gauge invariance}) satisfying $\Delta_{\mathbf{L}(\mathbf{t})}(\mathcal{I}) =\Delta_{\mathbf{L}(\mathbf{1})\cdot\delta(\mathbf{t})}(\mathcal{I})$, this form gives $\Delta_{\mathbf{X}(\mathbf{t})|_{\mathcal{H}}}(\mathcal{I})=\Delta_{\mathbf{X}(\mathbf{1})|_{\mathcal{H}}}(\mathcal{I})\cdot\mathbf{t}^{\mathbf{m_{0}}(\mathcal{H})}$. Then, the three-term Grassmann-Pl\"{u}cker relations applied to the generic matrix $\mathbf{X}(\mathbf{t})|^{\mathcal{I}}$, formed by selecting the rows of $\mathbf{X}(\mathbf{t})$ indexed by $\mathcal{I}$, impose the separability of $\mathbf{m_{0}}$: 
$
    \mathbf{m_{0}}(\mathcal{H})+\mathbf{m_{0}}(\mathcal{H}_{\alpha\beta}^{ij}) = \mathbf{m_{0}}(\mathcal{H}_{\alpha}^{i})+\mathbf{m_{0}}(\mathcal{H}_{\beta}^{j}) = \mathbf{m_{0}}(\mathcal{H}_{\alpha}^{j})+\mathbf{m_{0}}(\mathcal{H}_{\beta}^{i})
$ 
holds for each $\mathcal{H}\in \wp_{k}[s]$, 
$i,j\in\mathcal{H}$, and $\alpha,\beta\in\mathcal{H}^{\mathtt{C}}$. 
Thus, the proof of Theorem \ref{thm: equality phases monomial generalised} applies to $\mathbf{X}(\mathbf{t})|^{\mathcal{I}}$ and $\mathbf{m_{0}}$ as well. 
These implications convey information on the compatibility of $\mathbf{m_{0}}(\mathcal{H})$ across different choices of $\mathcal{H}\in\wp_{k}[s]$ and show how the weighting parameters $\mathbf{t}$ reveal scaling consistency across dimensions and realisations that are not evident when restricting to the unweighted case $\mathbf{t}=\mathbf{1}$. 
\end{rem} 

Finally, note that minors-based methods providing algebraic invariants have been used to generalise the tetrad approach in SEM 
\citep{Bollen2000} 
and obtain test statistics \citep{Drton2007} as well as graphical characterisations of vanishing minors \citep{Sullivant2010,Drton2018}. Here, extended tetrads (\ref{eq: cross-section}) support identifiability analysis and the construction of a canonical form of the factor $\mathbf{R}(\mathbf{t})$---and hence, the structural matrix $\mathbf{L}(\mathbf{t})$ compatible with our representation, starting from (\ref{eq: toric deformation}); see Theorem \ref{thm: retrieval of canonical forms}. Under Assumption \ref{claim: existence 3-key}, determinantal configurations (\ref{eq: separable deformations}) are precisely those yielding constant $Y$-terms (see Remark \ref{rem: constant Y-terms}), thus preserving the factor $\mathbf{R}$ in such a canonical form under weighting of the type (\ref{eq: toric deformation}). Remarkably, the main hypothesis underlying Theorem \ref{thm: retrieval of canonical forms} is also instrumental in examining the propagation of separability between different bases (Section \ref{sec: propagation of separability between different bases}).

\subsection{Relation to Algebraic Modelling in Sensitivity Analysis} 
\label{subsec: relation to sensitivity analysis in monomial models}  

The weighting scheme (\ref{eq: toric deformation}) can be viewed as a deformation $(\mathbf{L},\mathbf{R})\mapsto \left(\mathbf{L}(\mathbf{t}),\mathbf{R}(\mathbf{t})\right)$, introducing a monomial parameterisation into our factorisation framework. Monomial parameterisations are widely used in graphical models; see \citet{Leonelli2022} for details. This formulation facilitates sensitivity analysis by enabling the study of a special class of algebraic perturbations in model parameters and their effects on a probability distribution \citep{Leonelli2022}. 
Similar parameterisations arise in other graphical models. For example, in certain SEMs, covariances are expressed as sums of monomials associated with treks in a graph, encoding different effects in the model \citep{Weihs2018} and supporting estimation procedures \citep{ernst2023note}. 
In our setting, sensitivity analysis refers to 
canonical forms of factors derived from extended tetrads ($Y$-terms), which we examine via multiplicative perturbations (or scaling) of the terms in (\ref{eq: Cauchy-Binet expansion}) or (\ref{eq: IV model}). Moreover, the function $\Psi$ in (\ref{eq: monomial terms of Cauchy-Binet expansion}) helps identify which $Y$-terms remain constant and which vary, possibly vanishing at specific evaluation points $\mathbf{t}_{0}$. Indeed, $\Psi$ can also be used to control the simultaneous vanishing of terms (\ref{eq: toric deformation}) from a family $\mathfrak{F}\subseteq\mathfrak{G}(\mathbf{L})$ whenever there exists an index $u\in[d]$ such that $\Psi(\mathcal{H})_{u} > 0$ if $\mathcal{H}\in\mathfrak{F}$ and $\Psi(\mathcal{H})_{u} = 0$ otherwise, by evaluating $\mathbf{t}$ at $t_{u}=0$. More generally, constraints on $\Psi$ implied by Theorem \ref{thm: equality phases monomial} reflect on the properties of families of bases yielding vanishing terms at specific values of $\mathbf{t}$. 

Any non-zero choice of parameters $\mathbf{t}$ in (\ref{eq: toric deformation}) makes the deformation \emph{reversible}, meaning that it is invertible within a specified algebraic structure---here, the polynomial algebra $\mathbb{C}[\mathbf{t},\mathbf{t}^{-1}]$. These deformations based on the units in $\mathbb{C}[\mathbf{t},\mathbf{t}^{-1}]$ preserve the original information, as applying the inversion $\mathbf{t^{-1}}$ starting from the deformed minor products returns the original ones. As we shall see in Section \ref{subsec: principal minors and hyperdeterminants}, a key quantity for assessing the compatibility of such deformations with determinantal conditions and Assumption \ref{claim: no Y=-1} is the \emph{hyperdeterminant} \citep[Ch. 14]{GKZ1994}. This quantity emerges from the coupling of Grassmann-Pl\"{u}cker relations for $\mathbf{R}(\mathbf{t})$ and $\mathbf{L}(\mathbf{t})$. 
We remark that, beyond their role in studying simultaneous algebraic relations, hyperdeterminants are also instrumental in analysing independence \citep{Boege2018} and entropy \citep{Shadbakht2008} of Gaussian random variables. 


\subsection{Relation to the IIA Axiom and Ambiguity}
\label{subsec: counterexamples and relations to choice theory}

For a generic point $\mathbf{t}_{0}$, the function $\Psi$ in (\ref{eq: toric deformation}) induces an order on $\mathfrak{G}(\mathbf{L})$ by comparing norms $\Vert \mathbf{t}_{0}^{\Psi(\mathcal{I})} \Vert$ for $\mathcal{I}\in\mathfrak{G}(\mathbf{L})$. This scalarisation and the derived order depend only on $\mathfrak{G}(\mathbf{L})$ and $\Psi$, in line with our structure-based approach. For the separable configurations expressed by (\ref{eq: separable deformations}) or (\ref{eq: separability of set function}), $\Psi$ is \emph{additive} in component-wise contributions via the function $\psi$, which induces an order on the corresponding label set $[n]$. More generally, a function $\varphi:\,[n]\longrightarrow\mathbb{Z}^{d}$ returns, via scalarisation, a means to compare dimension-specific contributions through the norms $\Vert\mathbf{t}_{0}^{\varphi(\alpha)}\Vert$ for $\alpha\in[n]$, and we can examine the consistency of this relation on $[n]$ with the ordering on $\mathfrak{G}(\mathbf{L})$. 

The counterexample in Section \ref{subsubsec: Reduction to Principal Minor Assignment} 
shows that when $\Psi$ is non-additive, dimension-specific comparisons may depend on the choice of basis in $\mathfrak{G}(\mathbf{L})$ used for assessment. This dependence is analogous to failures of context independence of preferences underlying the IIA axiom \citep[Chap. 1.C]{Luce2005}. These orderings, while not expressing preference in our setting, reflect comparative relations between contributions to uncertainty structured within the model.  
The context-independence condition considered here requires that the ordering between two components remain unaffected by the auxiliary dimensions completing the basis with each of them. 

A related form of uncertainty is ambiguity, as effectively represented by Ellsberg's urn model \citep[Sec. 4]{Aerts2018}. This model can be characterised by the non-additive behaviour of subjective weights assigned to events \citep[Sec. 3.4]{Fishburn1986}. Our construction aligns with prior studies linking epistemic uncertainty to order inequivalence across alternative representations of conceptual models encoding Ellsberg-type ambiguity \citep{
angelelli2024}. 

Finally, we note a potential connection between the (non-)additivity of $\Psi$ and selected notions in Dempster-Shafer evidence theory; see \citet[Sec. 2]{Cuzzolin2020} for background. For any $\mathbf{n}_{0}\in\mathbb{Z}^{d}$, the shift $\Psi+\mathbf{n}_{0}$ generates a new configuration preserving (non-)separability. Then, for $\mathbf{n}_{0}\neq \mathbf{0}$, an appropriate selection of $\mathbf{t}_{0}$ with $\Vert\mathbf{t}_{0}^{\mathbf{n}_{0}}\Vert\notin\{0,1\}$ and $G\in\mathbb{Z}$ with $|G|$ large enough ensures $\Vert \mathbf{t}_{0}^{G\cdot \mathbf{n}_{0}}\Vert > \Vert \mathbf{t}_{0}^{-\Psi(\mathcal{I})}\Vert $ for all $\mathcal{I}\in\mathfrak{G}(\mathbf{L})$, yielding positive log-weights 
$\log(\Vert\mathbf{t}_{0}^{\Psi(\mathcal{I})+G\cdot \mathbf{n}_{0}}\Vert)$. With appropriate scaling $\mathbf{t}_{0}\mapsto \mathbf{t}_{0}^{A}$ for $A>0$, these log-weights can be normalised to form the masses of a \emph{basic probability assignment} (BPA) whose focal elements---subsets of $[n]$ with non-zero mass; see \citet[Def. 4]{Cuzzolin2020}---are indexed by $\mathfrak{G}(\mathbf{L})$. Under the assumptions of Theorem \ref{thm: equality phases monomial}, one may further construct a \emph{Bayesian} BPA (whose focal elements are singletons; see \citealp[Sec. 2]{Cuzzolin2010}) with masses $\log(\Vert\mathbf{t}_{0}^{a\cdot (\psi(\alpha)+(\mathbf{m}_{0}+g\cdot \mathbf{n_{0}})/k)}\Vert)$, with suitable $g,a\in\mathbb{R}$ ensuring positivity and normalisation, where the associated belief function \citep[Eq. (2.2)]{Cuzzolin2020}, restricted to $\mathfrak{G}(\mathbf{L})$, is expressed via $\Psi$. These observations suggest alignment with existing geometric frameworks for epistemic or non-probabilistic uncertainty, motivating further study. 

\section{Local Separability} 
\label{sec: principal minors and hyperdeterminants}

This section examines how the assumptions in Section \ref{subsec: Assumptions} impose \emph{local} constraints on the scaling of the determinantal terms---specifically, with a focus on the extended tetrads (\ref{eq: cross-section}) relative to a given basis $\mathcal{I}$. 

\subsection{Preliminary Lemmas} 
\label{sec: preliminary lemmas} 

The following lemmas and remarks provide auxiliary results for use in subsequent sections. 
\begin{lem}
\label{lem: chain of non-zero} Let $\mathcal{H},\mathcal{K}\in\mathfrak{G}(\mathbf{L})$ with $r:=|\mathcal{H}\backslash\mathcal{K}|$. Then, there exists a finite sequence $\mathcal{L}^{(0)}:=\mathcal{H}$, $\mathcal{L}^{(1)},\dots$, $\mathcal{L}^{(r)}:=\mathcal{K}$ of elements of $\mathfrak{G}(\mathbf{L})$ such that $|\mathcal{L}^{(u-1)}\Delta\mathcal{L}^{(u)}|=2$,
$u\in[r]$. 
\end{lem}
\begin{proof}
It follows from the exchange property of matroids (\ref{eq: exchange relation}); see \citet[Lemma 6]{Angelelli2019}. 
\end{proof}
\begin{lem}
\label{lem: algebraic forms Y-terms}
For observable sets $\chi(\mathcal{I}\mid_{\alpha\beta}^{ij})$, at $h(\mathcal{I})\cdot h(\mathcal{I}_{\alpha\beta}^{ij})=0$ we find that 
\begin{equation}
Y_{\alpha\beta}^{ij}=-\frac{h(\mathcal{I}_{\alpha}^{i})\cdot h(\mathcal{I}_{\beta}^{j})}{h(\mathcal{I}_{\beta}^{i})\cdot h(\mathcal{I}_{\alpha}^{j})}
\label{eq: special rational, monomial}
\end{equation}
is a non-zero monomial, while at $h(\mathcal{I}_{\alpha}^{i})\cdot h(\mathcal{I}_{\beta}^{j})=0$ we get 
 \begin{equation}
Y_{\alpha\beta}^{ij}=\frac{h(\mathcal{I})\cdot h(\mathcal{I}_{\alpha\beta}^{ij})}{h(\mathcal{I}_{\beta}^{i})\cdot h(\mathcal{I}_{\alpha}^{j})}-1.
\label{eq: special rational, binomial}
 \end{equation}
\end{lem}
\begin{proof}
Multiplying (\ref{eq: three-term Grassmann-Plucker relations}) for $\mathbf{L}(\mathbf{t})$ and $\mathbf{R}(\mathbf{t})$ side by side, and taking into account that $Y_{\alpha\beta}^{ij}\neq0$, we obtain  
 \begin{equation}
h(\mathcal{I})\cdot h(\mathcal{I}_{\alpha\beta}^{ij})\cdot Y_{\alpha\beta}^{ij}=\left(h(\mathcal{I}_{\alpha}^{i})\cdot h(\mathcal{I}_{\beta}^{j})+h(\mathcal{I}_{\beta}^{i})\cdot h(\mathcal{I}_{\alpha}^{j})\cdot Y_{\alpha\beta}^{ij}\right)\cdot(Y_{\alpha\beta}^{ij}+1). 
\label{eq: quartic from Grassmann-Plucker, 2}
\end{equation}
Then the thesis follows by direct computation, using $Y_{\alpha\beta}^{ij}\notin\{0,-1\}$ as per Assumption \ref{claim: no Y=-1}. 
\end{proof}
\begin{lem}
\label{lem: case B^(ij)_(ab) square} An observable set $\chi(\mathcal{I}\mid_{\alpha\beta}^{ij})$ is separable if and only if $Y_{\alpha\beta}^{ij}\in\mathbb{C}$. 
\end{lem}
\begin{proof}
The thesis follows from (\ref{eq: special rational, monomial})-(\ref{eq: special rational, binomial}) when $0\in \chi(\mathcal{I}\mid_{\alpha\beta}^{ij})$. Otherwise, 
by (\ref{eq: quartic from Grassmann-Plucker, 2}), $Y_{\alpha\beta}^{ij}$ is a root of a 
polynomial with coefficients in $\mathbb{C}[\mathbf{t},\mathbf{t}^{-1}]$. If $\chi(\mathcal{I}\mid_{\alpha\beta}^{ij})$ is separable, then this polynomial simplifies to one with constant coefficients and, hence, constant roots. Conversely, if $Y_{\alpha\beta}^{ij}\in\mathbb{C}$, in particular $Y_{\alpha\beta}^{ij}\notin\{0,-1\}$ by Assumption \ref{claim: no Y=-1}, then (\ref{eq: quartic from Grassmann-Plucker, 2}) expresses a linear dependence among the three monomials in $\chi(\mathcal{I}\mid_{\alpha\beta}^{ij})$, with coefficients in $\mathbb{C}\setminus\{0\}$. This condition 
can hold only if $\chi(\mathcal{I}\mid_{\alpha\beta}^{ij})$ is separable. 
\end{proof}
\begin{rem}
    \label{rem: algebraic key yields constant} 
If $\{i,j\}\times\{\alpha,\beta\}$ is a key generating a term $Y_{\alpha\beta}^{ij}\in\mathbb{F}$, then the discriminant of the quadratic polynomial (\ref{eq: quartic from Grassmann-Plucker, 2}) is a perfect square in $\mathbb{C}[\mathbf{t},\mathbf{t}^{-1}]$. By \citet[Lemma 4]{Angelelli2025}, it follows that $Y_{\alpha\beta}^{ij}\in\mathbb{C}$. 
\end{rem}

\subsection{Separable Sets from Non-Planar Keys}
\label{subsec: principal minors and hyperdeterminants}

\begin{lem}
\label{lem: constant 2-key under conditions}
Let $\mathcal{I}\in\mathfrak{G}(\mathbf{L})$ be a basis such that $h(\mathcal{I}_{\omega}^{s})=0$ for at least one pair $(s,\omega)\in\mathcal{I}\times\mathcal{I}^{\mathtt{C}}$, and $\mathfrak{c}:=\{i_{1},i_{2}\}\times\{\alpha_{1},\alpha_{2}\}\subseteq  \mathcal{I}\times\mathcal{I}^{\mathtt{C}}$ be a key. Then, $Y_{\gamma_{1}\gamma_{2}}^{i_{1}i_{2}},Y_{\alpha_{1}\alpha_{2}}^{l_{1}l_{2}}\in\mathbb{C}$ for any $\gamma_{1},\gamma_{2}\notin\mathcal{N}_{\mathfrak{c}_{r}}$ and  $l_{1},l_{2}\notin\mathcal{N}^{\mathfrak{c}_{c}}$. 
\end{lem}
\begin{proof}
First, we examine the consequences of the existence of $\delta\in\mathcal{I}^{\mathtt{C}}$ and $w\in\{1,2\}$ such that $0\in\chi(\mathcal{I}\mid_{\delta\alpha_{w}}^{i_{1}i_{2}})$. Under this assumption, at least one pair in $\mathfrak{c}$, say $(i_{2},\alpha_{2})$, generates a basis $\mathcal{J}:=\mathcal{I}_{\alpha_{2}}^{i_{2}}$ satisfying $h(\mathcal{J}_{\delta}^{j})=0$ for some 
$j\in\{i_{1},\alpha_{2}\}$. The absence of null columns (Section \ref{par: non-trivial dependence pattern}) ensures the existence of $m\in\mathcal{J}$ such that $h(\mathcal{J}_{\delta}^{m})\neq0$. Furthermore, when $\delta\notin\mathcal{N}_{(\mathfrak{c}_{\mathcal{J}})_{r}}$, we can, by definition, select such an index $m$ in $(\mathfrak{c}_{\mathcal{J}})_{r}=\{i_{1},\alpha_{2}\}$. This configuration yields $Y(\mathcal{J})_{\beta\delta}^{mj}\in\mathbb{F}$ for all $j\in\{i_{1},\alpha_{2}\}$ and $\beta\in\{\alpha_{1},i_{2}\}$, 
as all such $Y$-terms are either equal to $-1$ (when 
$m=j$) or are derived from (\ref{eq: special rational, binomial}). Applying (\ref{eq: quadrilateral decomposition}), we obtain $Y(\mathcal{J})_{\alpha_{1}i_{2}}^{i_{1}\alpha_{2}}\in\mathbb{F}$, and from Remark \ref{rem: algebraic key yields constant}, we find that $Y(\mathcal{J})_{\alpha_{1}i_{2}}^{i_{1}\alpha_{2}}\in\mathbb{C}$. By transposition, i.e., exchanging the roles of indices in $\mathcal{I}$ and $\mathcal{I}^{\mathtt{C}}$, the same holds if there exists $g\in\mathcal{I}$ such that $0\in\chi(\mathcal{I}|_{\alpha_{1}\alpha_{2}}^{i_{w}g})$. 

Now, to prove the thesis, it suffices to verify that  $Y_{\alpha_{1}\gamma}^{i_{1}i_{2}}\in\mathbb{C}$ for all $\gamma\notin\mathcal{N}_{\mathfrak{c}_{r}}$, as this implies $Y_{\gamma_{1}\gamma_{2}}^{i_{1}i_{2}}=-Y_{\gamma_{1}\alpha_{1}}^{i_{1}i_{2}}\cdot Y_{\alpha_{1}\gamma_{2}}^{i_{1}i_{2}}\in\mathbb{C}$. An analogous argument yields $Y_{\alpha_{1}\alpha_{2}}^{l_{1}l_{2}}\in\mathbb{C}$, $l_{1},l_{2}\notin\mathcal{N}^{\mathfrak{c}_{c}}$. 
Suppose that $0\in\chi(\mathcal{I}|_{\gamma\alpha_{w}}^{i_{1}i_{2}})$ for some $w\in\{1,2\}$ and set $\delta:=\gamma$ 
in the previous argument. The compatibility of the condition $Y(\mathcal{J})_{\alpha_{1}i_{2}}^{i_{1}\alpha_{2}}\in\mathbb{C}\setminus\{-1\}$ (from Assumption \ref{claim: no Y=-1}) with the expression derived from (\ref{eq: special rational, binomial}) for both $Y(\mathcal{J})_{\alpha_{1}\gamma}^{i_{1}\alpha_{2}}$ and $Y(\mathcal{J})_{i_{2}\gamma}^{i_{1}\alpha_{2}}$ implies that $Y(\mathcal{J})_{\beta_{1}\beta_{2}}^{i_{1}\alpha_{2}}\in\mathbb{C}$ for all $\beta_{1},\beta_{2}\in\{\alpha_{1},i_{2},\gamma\}$. Returning to $\mathcal{I}$ using (\ref{eq: Grassmann-Plucker translation exchange, vertical})-(\ref{eq: Grassmann-Plucker translation exchange, diagonal}), we have $Y_{\alpha_{1}\alpha_{2}}^{i_{1}i_{2}},Y_{\alpha_{2}\gamma}^{i_{1}i_{2}}\in\mathbb{C}$, and by (\ref{eq: associativity}), we obtain $Y_{\alpha_{1}\gamma}^{i_{1}i_{2}}\in\mathbb{C}$. Thus, we may have $Y_{\alpha_{1}\gamma}^{i_{1}i_{2}}\notin\mathbb{C}$ only if $\{i_{1},i_{2}\}\times\{\alpha_{1},\gamma\}$ is a key. Applying the previous reasoning to this key, 
$Y_{\alpha_{1}\gamma}^{i_{1}i_{2}}\notin\mathbb{C}$ would imply that we cannot find indices $\delta$ or $g$ as above, deducing that $\{j_{1},j_{2}\}\times\{\alpha_{1},\gamma\}$ and $\{i_{1},i_{2}\}\times\{\beta_{1},\beta_{2}\}$ are also keys for all $j_{1},j_{2}\in\mathcal{I}$ and $\beta_{1},\beta_{2}\in\mathcal{I}^{\mathtt{C}}$. From $h(\mathcal{I}_{\omega}^{s})=0$, we deduce that $Y_{\alpha_{1}\omega}^{i_{1}s}$, $Y_{\gamma\omega}^{i_{1}s}$, $Y_{\alpha_{1}\omega}^{i_{2}s}$, and $Y_{\gamma\omega}^{i_{2}s}$ 
take the form (\ref{eq: special rational, binomial}),  
which, by (\ref{eq: quadrilateral decomposition}), implies that $Y_{\alpha_{1}\gamma}^{i_{1}i_{2}}\in\mathbb{F}$. By Remark \ref{rem: algebraic key yields constant}, we conclude that $Y_{\alpha_{1}\gamma}^{i_{1}i_{2}}\in\mathbb{C}$. 
\end{proof}

\begin{prop}
\label{prop: constant 2-key}
We have $Y_{\gamma_{1}\gamma_{2}}^{i_{1}i_{2}},Y_{\alpha_{1}\alpha_{2}}^{l_{1}l_{2}}\in\mathbb{C}$ for any key $\mathfrak{c}:=\{i_{1},i_{2}\}\times\{\alpha_{1},\alpha_{2}\}$, with $\gamma_{1},\gamma_{2}\notin\mathcal{N}_{\mathfrak{c}_{r}}$ and $l_{1},l_{2}\notin\mathcal{N}^{\mathfrak{c}_{c}}$.
\end{prop}
\begin{proof}
The thesis follows from \citet[Thm. 1]{Angelelli2025} when $\mathfrak{G}(\mathbf{L})=\wp_{k}[n]$. Thus, we focus on cases where $\mathfrak{G}(\mathbf{L})$ is a proper subset of $\wp_{k}[n]$ and consider bases $\mathcal{J}$ that contain at least one pair $(s,\sigma)\in\mathcal{J}\times\mathcal{J}^{\mathtt{C}}$ such that $h(\mathcal{J}_{\sigma}^{s})=0$. Adopting the same notation as in Lemma \ref{lem: chain of non-zero}, we choose such a basis $\mathcal{L}^{(r)}$ with minimal exchange distance (in terms of symmetric difference) $r:=\frac{1}{2}\cdot |\mathcal{I}\Delta\mathcal{L}^{(r)}|$ from $\mathcal{I}=:\mathcal{L}^{(0)}$; note that $r+1$ is the minimal distance between $\mathcal{I}$ and any set in $\wp_{k}[n]\setminus\mathfrak{G}(\mathbf{L})$. The case $r=0$ is established by Lemma \ref{lem: constant 2-key under conditions}, so we assume $r>0$. Let $\mathcal{L}^{(u)}\setminus\mathcal{L}^{(u-1)}=:\{\gamma_{u}\}$ and $\mathcal{L}^{(u-1)}\setminus\mathcal{L}^{(u)}=:\{l_{u}\}$. 

Observe that $h((\mathcal{L}^{(r-1)})_{\gamma_{r}\omega}^{l_{r}m})=0$ for all $m,\omega$ would imply that columns $\gamma_{r}$ and $\omega$ of $\mathbf{L}(\mathbf{t})$ are proportional for all $\omega\notin\mathcal{L}^{(r-1)}$. It would follow that each basis must have the form $\mathcal{L}^{(r-1)}$ or $(\mathcal{L}^{(r-1)})_{\alpha}^{i}$ for some $i\in\mathcal{L}^{(r-1)}$ and $\alpha \notin \mathcal{L}^{(r-1)}$, implying that there are at most $k\cdot(n-k)+1$ bases. However, this contradicts the existence of a key in $\mathcal{I}$, since $h(\mathcal{I}_{\alpha}^{i})\neq0$ for all $i\in\mathcal{I}$ and $\alpha\in\mathcal{I}^{\mathtt{C}}$ by construction (given that $r>0$), and we also have $h(\mathcal{I})\neq 0 $ and $h(\mathcal{I}_{\alpha_{1}\alpha_{2}}^{i_{1}i_{2}})\neq0$, leading to at least $k\cdot(n-k)+2$ bases. Thus, we find a key $\{l_{r},\overline{m}\}\times\{\gamma_{r},\overline{\omega}\}$ for some $\overline{m}\in\mathcal{L}^{(r-1)}$ and $\overline{\omega}\notin\mathcal{L}^{(r-1)}$, which generates a key $\{\gamma_{r},\overline{m}\}\times\{l_{r},\overline{\omega}\}$ in $\mathcal{L}^{(r)}$. Furthermore, each 
$Y$-term in the basis $\mathcal{L}^{(r-1)}$ is observable by construction, and the terms $Y(\mathcal{L}^{(r)})_{l_{r}\overline{\omega}}^{\gamma_{r}m}
\in\mathbb{C}$ remain observable when moving from $\mathcal{L}^{(r-1)}$ to $(\mathcal{L}^{(r-1)})_{\gamma_{r}}^{l_{r}}=\mathcal{L}^{(r)}$. By Lemma \ref{lem: constant 2-key under conditions}, we obtain $Y(\mathcal{L}^{(r)})_{l_{r}\overline{\omega}}^{\gamma_{r}m}
\in\mathbb{C}$ for all $m$, and by (\ref{eq: Grassmann-Plucker translation exchange, vertical}), it follows that $Y(\mathcal{L}^{(r-1)})_{\gamma_{r}\overline{\omega}}^{l_{r}m}
\in\mathbb{C}$. Analogously, $Y(\mathcal{L}^{(r-1)})_{\gamma_{r}\omega}^{l_{r}\overline{m}}
\in\mathbb{C}$ for all $\omega$. The remaining pairs $(m,\omega)$ to be analysed, where 
the hypotheses of Lemma \ref{lem: constant 2-key under conditions} do not apply, are those for which, for every choice of the key $\{l_{r},\overline{m}\}\times\{\gamma_{r},\overline{\omega}\}$, neither $\{l_{r},\overline{m}\}\times\{\gamma_{r},\omega\}$ nor $\{l_{r},m\}\times\{\gamma_{r},\overline{\omega}\}$ are keys in $\mathcal{L}^{(r-1)}$. This implies that the 
columns $\omega$ and $\gamma_{r}$ of $\mathbf{L}(\mathbf{t})$ are proportional, as are, in the representation based on $\mathcal{L}^{(r-1)}$, the rows $m$ and $l_{r}$ restricted to $(\mathcal{L}^{(r-1)})^{\mathtt{C}}$. For these pairs, we obtain four keys in $\mathcal{L}^{(r-1)}$ of the form $\{j,\overline{m}\}\times\{\beta,\overline{\omega}\}$ for all $j\in\{l_{r},m\}$ and $\beta\in\{\gamma_{r},\omega\}$, as they arise from substituting a column or row from the key $\{l_{r},\overline{m}\}\times\{\gamma_{r},\overline{\omega}\}$ with a proportional one. Applying Lemma \ref{lem: constant 2-key under conditions} after passing to the bases $(\mathcal{L}^{(r-1)})_{\beta}^{j}$ and using (\ref{eq: Grassmann-Plucker translation exchange, vertical}), 
we obtain $Y(\mathcal{L}^{(r-1)})_{\beta\overline{\omega}}^{j\overline{m}}\in\mathbb{C}$. From (\ref{eq: quadrilateral decomposition}), it follows that $Y(\mathcal{L}^{(r-1)})_{\gamma_{r}\omega}^{l_{r}m}\in\mathbb{C}$. 

Since the pair $(m,\omega)$ is arbitrary, we can specify the previous argument to both $m=\gamma_{r-1}$ and any $m\neq\gamma_{r-1}$, as well as $\omega=l_{r-1}$ and any $\omega\neq l_{r-1}$. Applying once more (\ref{eq: quadrilateral decomposition}) for these choices, we obtain $Y(\mathcal{L}^{(r-1)})_{\omega l_{r-1}}^{m\gamma_{r-1}}\in\mathbb{C}$, and by (\ref{eq: Grassmann-Plucker translation exchange, vertical}), $Y(\mathcal{L}^{(r-2)})_{\omega\gamma_{r-1}}^{ml_{r-1}}\in\mathbb{C}$. Iterating this last step, we conclude that $Y(\mathcal{L}^{(u-1)})_{\alpha\beta}^{ij}\in\mathbb{C}$ for all $u\in[r]$, $i,j\in\mathcal{L}^{(u-1)}$, and $\alpha,\beta\notin\mathcal{L}^{(u-1)}$. In particular, this holds for $\mathcal{L}^{(0)}=\mathcal{I}$, which completes the proof. 
\end{proof}
\begin{lem}
\label{lem: quadratic extension from diagonal binomials} For  $Y(\mathcal{I})_{\delta_{3}\delta_{1}}^{a_{1}a_{3}},Y(\mathcal{I})_{\delta_{2}\delta_{1}}^{a_{1}a_{2}},Y(\mathcal{I})_{\delta_{3}\delta_{2}}^{a_{2}a_{3}},Y(\mathcal{I}_{\delta_{1}}^{a_{1}})_{\delta_{3}\delta_{2}}^{a_{2}a_{3}}\in\mathbb{F}$, the term $Y(\mathcal{I})_{\delta_{3}\delta_{1}}^{a_{1}a_{2}}$ is a root of a polynomial with coefficients in $\mathbb{F}$ and degree at most $2$. 
\end{lem}
\begin{proof}
Let us introduce 
 \begin{equation}
\varepsilon_{\delta_{1}\delta_{2}\delta_{3}}^{a_{1}a_{2}a_{3}} := -\mathrm{sign}\left[\prod_{u<w}(a_{u}-a_{w})\cdot\prod_{x<z}(\delta_{x}-\delta_{z})\cdot\prod_{r\neq s}(a_{r}-\delta_{s})\right],  
\label{eq: three-indices sign} 
\end{equation} 
\begin{equation}
\mathfrak{m}_{\delta_{1}\delta_{2}\delta_{3}}^{a_{1}a_{2}a_{3}} := \varepsilon_{\delta_{1}\delta_{2}\delta_{3}}^{a_{1}a_{2}a_{3}}\cdot\frac{\Delta_{\mathbf{R}(\mathbf{t})}(\mathcal{I})^{2}\cdot\Delta_{\mathbf{R}(\mathbf{t})}(\mathcal{I}_{\delta_{1}\delta_{2}\delta_{3}}^{a_{1}a_{2}a_{3}})}{\Delta_{\mathbf{R}(\mathbf{t})}(\mathcal{I}_{\delta_{1}}^{a_{1}})\cdot\Delta_{\mathbf{R}(\mathbf{t})}(\mathcal{I}_{\delta_{2}}^{a_{2}})\cdot\Delta_{\mathbf{R}(\mathbf{t})}(\mathcal{I}_{\delta_{3}}^{a_{3}})}. 
\label{eq: three Y-term}
 \end{equation}
Recalling (\ref{eq: cross-section}), we directly verify that the identity 
\begin{equation} 
1+Y_{\delta_{2}\delta_{1}}^{a_{1}a_{2}}+Y_{\delta_{3}\delta_{1}}^{a_{1}a_{3}}+Y_{\delta_{3}\delta_{2}}^{a_{2}a_{3}}+Y_{\delta_{3}\delta_{2}}^{a_{2}a_{3}}\cdot Y_{\delta_{3}\delta_{1}}^{a_{1}a_{2}}-(Y_{\delta_{3}\delta_{1}}^{a_{1}a_{2}})^{-1}\cdot Y_{\delta_{3}\delta_{1}}^{a_{1}a_{3}}\cdot Y_{\delta_{2}\delta_{1}}^{a_{1}a_{2}}=\mathfrak{m}_{\delta_{1}\delta_{2}\delta_{3}}^{a_{1}a_{2}a_{3}} 
\label{eq: diagonal minors 3-term expansion}
 \end{equation} 
holds. Consequently, $Y_{\delta_{3}\delta_{2}}^{a_{2}a_{3}}\cdot Y_{\delta_{3}\delta_{1}}^{a_{1}a_{2}}$ is a root of the quadratic polynomial   
 \begin{equation}
P_{\delta_{1}\delta_{2}\delta_{3}}^{a_{1}a_{2}a_{3}}(X):=X^{2}+\left(1+Y_{\delta_{2}\delta_{1}}^{a_{1}a_{2}}+Y_{\delta_{3}\delta_{1}}^{a_{1}a_{3}}+Y_{\delta_{3}\delta_{2}}^{a_{2}a_{3}}-\mathfrak{m}_{\delta_{1}\delta_{2}\delta_{3}}^{a_{1}a_{2}a_{3}}\right)\cdot X-Y_{\delta_{3}\delta_{1}}^{a_{1}a_{3}}\cdot Y_{\delta_{2}\delta_{1}}^{a_{1}a_{2}}\cdot Y_{\delta_{3}\delta_{2}}^{a_{2}a_{3}}. 
\label{eq: diagonal minors, quadratic polynomial}
 \end{equation}

Now, let $c_{1}^{(s,t)}$ denote the sign $c_{1}$ in (\ref{eq: signs for permutation}) for  $(i,j,\alpha,\beta):=(a_{s},a_{t},\delta_{s},\delta_{t})$. From (\ref{eq: three-indices sign}), we note that
 \begin{equation} 
 \varepsilon_{\delta_{1}\delta_{2}\delta_{3}}^{a_{1}a_{2}a_{3}} = c_{1}^{(1,2)}\cdot c_{1}^{(1,3)}\cdot c_{1}^{(2,3)}. 
\label{eq: three-index sign from three two-index signs}
 \end{equation} 
Additionally, from the three-term Grassmann-Pl\"{u}cker relations (\ref{eq: three-term Grassmann-Plucker relations}), we derive 
 \begin{equation}
    Y_{\delta_{t}\delta_{s}}^{a_{s}a_{t}}+1=c_{1}^{(s,t)}\frac{\Delta_{\mathbf{R}(\mathbf{t})}(\mathcal{I})\cdot\Delta_{\mathbf{R}(\mathbf{t})}(\mathcal{I}_{\delta_{s}\delta_{t}}^{a_{s}a_{t}})}{\Delta_{\mathbf{R}(\mathbf{t})}(\mathcal{I}_{\delta_{s}}^{a_{s}})\cdot\Delta_{\mathbf{R}(\mathbf{t})}(\mathcal{I}_{\delta_{t}}^{a_{t}})}.
\label{eq: Y+1 sign ratio}
 \end{equation}
Combining (\ref{eq: three Y-term}), (\ref{eq: three-index sign from three two-index signs}), and (\ref{eq: Y+1 sign ratio}), we obtain the identity
 \begin{equation}
    \mathfrak{m}_{\delta_{1}\delta_{2}\delta_{3}}^{a_{1}a_{2}a_{3}}=\left(Y(\mathcal{I}_{\delta_{1}}^{a_{1}})_{\delta_{3}\delta_{2}}^{a_{2}a_{3}}+1\right)\cdot\left(Y(\mathcal{I})_{\delta_{2}\delta_{1}}^{a_{1}a_{2}}+1\right)\cdot\left(Y(\mathcal{I})_{\delta_{3}\delta_{1}}^{a_{1}a_{3}}+1\right) 
\label{eq: g-factor for change of basis}
 \end{equation} 
which implies, under the stated hypothesis, that all coefficients in (\ref{eq: diagonal minors, quadratic polynomial}) belong to $\mathbb{F}$. 
\end{proof}
\begin{rem}
\label{rem: hyperdeterminant}  
It is observed that the term $Y_{\delta_{3}\delta_{2}}^{a_{2}a_{3}}\cdot Y_{\delta_{2}\delta_{1}}^{a_{1}a_{3}}$ also satisfies the quadratic equation (\ref{eq: diagonal minors, quadratic polynomial}). More precisely, denoting the roots of this equation as $Y_{\delta_{3}\delta_{2}}^{a_{2}a_{3}}\cdot X_{+}$ and $Y_{\delta_{3}\delta_{2}}^{a_{2}a_{3}}\cdot X_{-}$, where $X_{+}=Y_{\delta_{3}\delta_{1}}^{a_{1}a_{2}}$ by Lemma \ref{lem: quadratic extension from diagonal binomials}, the relation between roots and coefficients in a polynomial yields 
\begin{equation} 
Y_{\delta_{3}\delta_{2}}^{a_{2}a_{3}}\cdot X_{+} \cdot Y_{\delta_{3}\delta_{2}}^{a_{2}a_{3}}\cdot X_{-} = -Y_{\delta_{3}\delta_{1}}^{a_{1}a_{3}}\cdot Y_{\delta_{2}\delta_{1}}^{a_{1}a_{2}}\cdot Y_{\delta_{3}\delta_{2}}^{a_{2}a_{3}} 
\quad 
\Rightarrow 
\quad 
X_{-}=-Y_{\delta_{3}\delta_{1}}^{a_{2}a_{1}}\cdot Y_{\delta_{3}\delta_{1}}^{a_{1}a_{3}}\cdot Y_{\delta_{2}\delta_{1}}^{a_{1}a_{2}}\cdot Y_{\delta_{2}\delta_{3}}^{a_{2}a_{3}} 
= Y_{\delta_{2}\delta_{1}}^{a_{1}a_{3}}. 
\label{eq: conjugate root}
\end{equation}
This result is consistent with the expression of the polynomial (\ref{eq: diagonal minors, quadratic polynomial}), as its coefficients are invariant under permutations of the labels $i\in[3]$ of $(a_{i},\delta_{i})$. This includes the term (\ref{eq: g-factor for change of basis}), and hence the factor 
 \begin{equation}
g_{\delta_{1}\delta_{2}\delta_{3}}^{a_{1}a_{2}a_{3}}:=\frac{Y(\mathcal{I}_{\delta_{1}}^{a_{1}})_{\delta_{2}\delta_{3}}^{a_{3}a_{2}}+1}{Y(\mathcal{I})_{\delta_{2}\delta_{3}}^{a_{3}a_{2}}+1}=\frac{\mathfrak{m}_{\delta_{1}\delta_{2}\delta_{3}}^{a_{1}a_{2}a_{3}}}{\left(Y(\mathcal{I})_{\delta_{3}\delta_{2}}^{a_{2}a_{3}}+1\right)\cdot\left(Y(\mathcal{I})_{\delta_{2}\delta_{1}}^{a_{1}a_{2}}+1\right)\cdot\left(Y(\mathcal{I})_{\delta_{3}\delta_{1}}^{a_{1}a_{3}}+1\right)}
\label{eq: g-factor}
 \end{equation}
is invariant under label permutations as well. 

In conclusion, we also note that the discriminant of (\ref{eq: diagonal minors, quadratic polynomial}) is given by 
 \begin{equation}
\Delta_{\delta_{1}\delta_{2}\delta_{3}}^{a_{1}a_{2}a_{3}}:=\left(1+Y_{\delta_{2}\delta_{1}}^{a_{1}a_{2}}+Y_{\delta_{3}\delta_{1}}^{a_{1}a_{3}}+Y_{\delta_{3}\delta_{2}}^{a_{2}a_{3}}-\mathfrak{m}_{\delta_{1}\delta_{2}\delta_{3}}^{a_{1}a_{2}a_{3}}\right)^{2}+4\cdot Y_{\delta_{3}\delta_{1}}^{a_{1}a_{3}}\cdot Y_{\delta_{2}\delta_{1}}^{a_{1}a_{2}}\cdot Y_{\delta_{3}\delta_{2}}^{a_{2}a_{3}}
\label{eq: discriminant}
 \end{equation}
which, when expressed in terms of the quantities $Y_{\delta_{w}\delta_{u}}^{a_{u}a_{w}}+1$ for $1\leq u<w\leq 3$, is a non-homogeneous analogue of the $2\times2\times2$ hyperdeterminant \citep{GKZ1994,Shadbakht2008}. 
\end{rem}

\begin{prop}
\label{prop: constant strip and principal minors}
Let $\mathfrak{c}:=\{i_{1},i_{2}\}\times\{\alpha_{1},\alpha_{2}\}$ be a non-planar key. Then, for each $(m,\omega)$ such that $\chi(\mathcal{I}\mid_{\alpha\omega}^{im})$ is observable for some $(i,\alpha)\in\mathfrak{c}$, we have
$Y_{\alpha_{w}\omega}^{mi_{u}}\in\mathbb{C}$ for all $u,w\in\{1,2\}$.
\end{prop} 
\begin{proof} 
The proof is presented in Appendix \ref{app: constant strip and principal minors} for completeness. 
\end{proof} 

\subsection{Identification of a Separable Basis}
\label{subsec: identification of a separable basis}
\begin{prop}
\label{prop: no non-trivial 2-curvature} 
Let $\mathfrak{c}$ be a 
non-planar key. Then, we have $Y_{\omega_{2}\omega_{1}}^{m_{1}m_{2}}\in\mathbb{C}$ for all $(m_{1},\omega_{1}),(m_{2},\omega_{2})\in\mathcal{N}(\mathfrak{c})$ with $h(\mathcal{I}_{\omega_{s}}^{m_{s}})\neq 0$ for $s\in\{1,2\}$. 
\end{prop} 
\begin{proof} 
The proof is presented in Appendix \ref{app: no non-trivial 2-curvature} for completeness. 
\end{proof} 

\begin{thm}
\label{thm: equality phases, single-exchange set} Under Assumption \ref{claim: existence 3-key}, we can find a separable basis $\mathcal{I}$. 
\end{thm}
\begin{proof}
By Lemma
\ref{lem: case B^(ij)_(ab) square}, the thesis is equivalent to 
$Y(\mathcal{I})_{\omega_{1}\omega_{2}}^{m_{1}m_{2}}\in\mathbb{C}$ for each observable set $\chi(\mathcal{I}|_{\omega_{1}\omega_{2}}^{m_{1}m_{2}})$, for which we choose a labelling 
such that $h(\mathcal{I}_{\omega_{1}}^{m_{1}})\cdot h(\mathcal{I}_{\omega_{2}}^{m_{2}})\neq0$. 
Considering the proofs of Propositions \ref{prop: constant 2-key} and \ref{prop: constant strip and principal minors}, 
it suffices to check observable sets $\chi(\mathcal{I}\mid_{\omega_{1}\omega_{2}}^{m_{1}m_{2}})$ where at most one of the indices (either upper or lower) comes from a non-planar key $\mathfrak{c}:=\{i_{1},i_{2}\}\times\{\alpha_{1},\alpha_{2}\}$; we begin by considering observable sets where one index comes from $\mathfrak{c}_{r}$, say $m_{1}=i_{1}$, as an analogous argument applies when either $\omega_{1}$ or $\omega_{2}$ belongs to $\mathfrak{c}_{c}$. We obtain $Y_{\omega_{1}\omega_{2}}^{i_{1}m_{2}}\in\mathbb{C}$ if there exists an index $\sigma\in\{\alpha_{1},\alpha_{2},\omega_{1}\}$ such that $h(\mathcal{I}_{\sigma}^{m_{2}})\neq0$ since, for $m_{1}=i_{1}$, at least one $w\in\{1,2\}$ makes $\chi(\mathcal{I}\mid_{\alpha_{w}\omega_{1}}^{i_{1}m_{2}})$ observable. Since $\chi(\mathcal{I}\mid_{\alpha_{w}\omega_{2}}^{i_{1}m_{2}})$ is also observable, Proposition \ref{prop: constant strip and principal minors} yields $Y_{\omega_{1}\omega_{2}}^{i_{1}m_{2}}=-Y_{\omega_{1}\alpha_{w}}^{i_{1}m_{2}}\cdot Y_{\alpha_{w}\omega_{2}}^{i_{1}m_{2}}\in\mathbb{C}$. Thus, we focus on $m_{2}\in\mathcal{N}^{\{\alpha_{1},\alpha_{2},\omega_{1}\}}$, for which $\mathfrak{c}$ remains a non-planar key under the transformation $\mathcal{I}\mapsto\mathcal{I}_{\omega_{2}}^{m_{2}}$ as per Remark \ref{rem: planarity preservation}. Along with $m_{1}=i_{1}$, this implies $
h(\mathcal{I}_{\alpha_{1}}^{i_{2}})\cdot h(\mathcal{I}_{\omega_{1}}^{i_{1}})\neq0\neq h(\mathcal{I}_{\alpha_{1}\omega_{2}}^{i_{2}m_{2}})\cdot h(\mathcal{I}_{\omega_{1}\omega_{2}}^{i_{1}m_{2}})$, which makes the terms $Y(\mathcal{I})_{\omega_{1}\alpha_{1}}^{i_{1}i_{2}}$ and $Y(\mathcal{I}_{\omega_{2}}^{m_{2}})_{\omega_{1}\alpha_{1}}^{i_{1}i_{2}}$ observable and hence constant by Proposition \ref{prop: constant 2-key}. Analogously, the observable terms $Y(\mathcal{I})_{\alpha_{1}\omega_{2}}^{i_{1}m_{2}}$ and $ Y(\mathcal{I})_{\alpha_{1}\omega_{2}}^{m_{2}i_{2}}$ belong to $\mathbb{C}$ by Proposition \ref{prop: constant strip and principal minors}; from (\ref{eq: associativity}), it follows that  $Y(\mathcal{I})_{\alpha_{1}\omega_{2}}^{i_{1}i_{2}}
\in\mathbb{C}$, which also implies $Y(\mathcal{I})_{\omega_{1}\omega_{2}}^{i_{1}i_{2}}=-Y(\mathcal{I})_{\omega_{1}\alpha_{1}}^{i_{1}i_{2}}\cdot Y(\mathcal{I})_{\alpha_{1}\omega_{2}}^{i_{1}i_{2}}
\in\mathbb{C}\setminus\{0,-1\}$ (as per Assumption \ref{claim: no Y=-1}) and, by applying (\ref{eq: Grassmann-Plucker translation exchange, vertical})-(\ref{eq: Grassmann-Plucker translation exchange, diagonal}) to this factorisation, $Y(\mathcal{I}_{\alpha_{1}}^{i_{1}})_{\omega_{1}\omega_{2}}^{\alpha_{1}i_{2}}\in\mathbb{C}\setminus\{0,-1\}$. Finally, (\ref{eq: g-factor for change of basis}) and (\ref{eq: g-factor}) give 
\begin{equation}
    \frac{Y(\mathcal{I}_{\alpha_{1}}^{i_{1}})_{\omega_{2}\omega_{1}}^{i_{2}m_{2}}+1}{Y(\mathcal{I})_{\omega_{2}\omega_{1}}^{i_{2}m_{2}}+1}=g_{\alpha_{1}\omega_{1}\omega_{2}}^{i_{1}i_{2}m_{2}}=\frac{Y(\mathcal{I}_{\omega_{2}}^{m_{2}})_{\omega_{1}\alpha_{1}}^{i_{1}i_{2}}+1}{Y(\mathcal{I})_{\omega_{1}\alpha_{1}}^{i_{1}i_{2}}+1}\in\mathbb{C}.
    \label{eq: double expression g-factor for constant}
\end{equation} 

With these premises, we focus on the bases $\mathcal{I}$ and $\mathcal{I}_{\alpha_{1}}^{i_{1}}$. At least one of the conditions $h(\mathcal{I}_{\omega_{1}}^{i_{2}})\neq0$ or $h((\mathcal{I}_{\alpha_{1}}^{i_{1}})_{\omega_{1}}^{i_{2}})\neq0$ holds since $\chi(\mathcal{I}\mid_{\alpha_{1}\omega_{1}}^{i_{1}i_{2}})$ is observable for $m_{1}=i_{1}$, while $h(\mathcal{A}_{\omega_{1}}^{m_{2}})=0$ 
for both $\mathcal{A}\in\{\mathcal{I},\mathcal{I}_{\alpha_{1}}^{i_{1}}\}$, 
as $m_{2}\in\mathcal{N}^{\{\alpha_{1},\alpha_{2},\omega_{1}\}}$. 
We then choose such $\mathcal{A}$
satisfying $h(\mathcal{A}_{\omega_{1}}^{i_{2}})\neq0$ and set $\{\pi_{1}\}:=\mathcal{A}\cap\{i_{1},\alpha_{1}\}$ to denote the corresponding upper index. Thus, both $Y(\mathcal{A})_{\omega_{1}\omega_{2}}^{m_{2}\pi_{1}}$ and $Y(\mathcal{A})_{\omega_{1}\omega_{2}}^{m_{2}i_{2}}$ take the form given in (\ref{eq: special rational, binomial}). In light of relation (\ref{eq: associativity}), these forms are consistent with $Y(\mathcal{A})_{\omega_{1}\omega_{2}}^{\pi_{1}i_{2}}\in\mathbb{C}\setminus\{0,-1\}$, as previously derived, only if $Y(\mathcal{A})_{\omega_{1}\omega_{2}}^{m_{2}i_{2}}\in\mathbb{C}$ as well. Hence, 
the numerator or denominator in (\ref{eq: double expression g-factor for constant}) belongs to $\mathbb{C}$, as does their ratio, implying both are constant. 
Since we already obtained   $Y(\mathcal{I})_{\omega_{1}\omega_{2}}^{i_{1}i_{2}}\in\mathbb{C}$, we conclude that $Y_{\omega_{1}\omega_{2}}^{i_{1}m_{2}}=-Y_{\omega_{1}\omega_{2}}^{i_{1}i_{2}}\cdot Y_{\omega_{1}\omega_{2}}^{i_{2}m_{2}}\in\mathbb{C}$. 

We now consider all remaining observable sets.  
For $(m_{1},\omega_{1}),(m_{2},\omega_{2})\in\mathcal{N}(\mathfrak{c})$, we obtain $Y_{\omega_{1}\omega_{2}}^{m_{1}m_{2}}\in\mathbb{C}$ by Proposition \ref{prop: no non-trivial 2-curvature}. In the remaining case where $(m_{s},\omega_{s})\notin\mathcal{N}(\mathfrak{c})$ for some $s\in\{1,2\}$, say $s=1$, there exists $(i_{u},\alpha_{w})\in\mathfrak{c}$ such that $h(\mathcal{I}_{\alpha_{w}}^{m_{1}})\cdot h(\mathcal{I}_{\omega_{1}}^{i_{u}})\neq0$.   
Thus, the sets $\chi(\mathcal{I}\mid_{\alpha_{w}\omega_{2}}^{m_{1}m_{2}})$ and $\chi(\mathcal{I}\mid_{\omega_{1}\omega_{2}}^{i_{u}m_{2}})$ are observable with an index $\alpha_{w}$ from $\mathfrak{c}_{c}$ and $i_{u}$ from $\mathfrak{c}_{r}$,  respectively, and the argument above applies to both. This leads to $Y_{\alpha_{w}\omega_{2}}^{m_{1}m_{2}},Y_{\omega_{1}\omega_{2}}^{i_{u}m_{2}}\in\mathbb{C}$. By Proposition \ref{prop: constant strip and principal minors}, we also obtain $Y_{\omega_{1}\alpha_{w}}^{m_{1}i_{u}},Y_{\alpha_{w}\omega_{2}}^{m_{2}i_{u}}\in\mathbb{C}$ since  
these terms are observable, concluding 
$Y_{\omega_{1}\omega_{2}}^{m_{1}m_{2}}=-(Y_{\omega_{1}\alpha_{w}}^{m_{1}i_{u}}\cdot(Y_{\alpha_{w}\omega_{2}}^{m_{1}m_{2}}\cdot Y_{\alpha_{w}\omega_{2}}^{m_{2}i_{u}}))\cdot Y_{\omega_{1}\omega_{2}}^{i_{u}m_{2}}\in\mathbb{C} 
$. 
\end{proof} 

\section{From Local to Global Separability}
\label{sec: propagation of separability between different bases} 

The results in the previous section mainly focus on observable $Y$-terms within a given basis. By connecting such local information from different bases, terms not examined under a specific choice of reference dimensions can still contribute to the derivation of a canonical form of the factors based on extended tetrads (Theorem \ref{thm: retrieval of canonical forms} below). Furthermore, even without Assumption \ref{claim: existence 3-key}, knowing that a separable basis exists allows us to infer separability for other bases in accordance with Theorem \ref{thm: transfer of separability} stated later. We address this extension by relaxing Assumption \ref{claim: existence 3-key} to a weaker condition ensuring the propagation of separability across bases: the existence of $g\in\mathcal{I}$ and $\kappa_{1},\kappa_{2}\in\mathcal{I}^{\mathtt{C}}$, with $\kappa_{1}\neq\kappa_{2}$, such that 
 \begin{equation}
    h(\mathcal{I}_{\kappa_{1}}^{g})\cdot h(\mathcal{I}_{\kappa_{2}}^{g})\neq 0.
    \label{eq: weak key}
 \end{equation}

\begin{rem}
\label{rem: transmission weak key} 
The existence of two indices satisfying (\ref{eq: weak key}) holds for a basis $\mathcal{I}$ if and only if it holds for any basis $\mathcal{J}\in\mathfrak{G}(\mathbf{L})$. Indeed, negating this property means requiring injectivity for any mapping $\varrho:\,\mathcal{I}^{\mathtt{C}}\rightarrow \mathcal{I}$ such that $h(\mathcal{I}_{\kappa}^{\varrho(\kappa)})\neq 0$ for all $\kappa\in\mathcal{I}^{\mathtt{C}}$. The other bases share the same structure, arising from exchanges of a set $\mathcal{A}\subseteq\mathcal{I}^{\mathtt{C}}$ with $\varrho(\mathcal{A})$ for any such mapping $\varrho$. In line with Remark \ref{rem: Grassmann-Plucker coordinates and factor score representations}, a dual argument applies concerning the existence of $g_{1},g_{2}\in\mathcal{I}$ and $\kappa\in\mathcal{I}^{\mathtt{C}}$ such that $h(\mathcal{I}_{\kappa}^{g_{1}})\cdot h(\mathcal{I}_{\kappa}^{g_{2}})\neq 0$. In the remainder of this section, we focus on (\ref{eq: weak key}) while keeping this duality in mind. 
\end{rem}
\begin{thm}
\label{thm: transfer of separability}
If there exists a separable basis $\mathcal{I}$ and indices $g\in\mathcal{I}$, $\kappa_{1},\kappa_{2}\in\mathcal{I}^{\mathtt{C}}$ such that (\ref{eq: weak key}) holds, then every basis in $\mathfrak{G}(\mathbf{L})$ is separable. In particular, separability follows from Assumption \ref{claim: existence 3-key}. 
\end{thm}
\begin{proof}
Let $\mathcal{I}$ be separable, and assume, for the sake of contradiction, that the thesis fails for a given basis $\mathcal{J}\in\mathfrak{G}(\mathbf{L})$. Consider a finite sequence  $\mathcal{L}^{(0)}:=\mathcal{I},\,\mathcal{L}^{(1)},\dots,\mathcal{L}^{(r)}:=\mathcal{J}$ of elements of $\mathfrak{G}(\mathbf{L})$, where $r=|\mathcal{I}\setminus\mathcal{J}|$ and $|\mathcal{L}^{(u-1)}\Delta\mathcal{L}^{(u)}|=2$ for all $u\in[r]$, as
stated in Lemma \ref{lem: chain of non-zero}. Define 
\begin{equation}
q:=\min\left\{ u\in[r]:\,\mathcal{L}^{(u-1)}\text{ is separable and }\text{\ensuremath{\mathcal{L}}}^{(u)}\text{ is non-separable}\right\} .
\label{eq: separability separation}
 \end{equation}
To simplify notation, let $\mathcal{A}:=\mathcal{L}^{(q-1)}$ and $\mathcal{B}:=\mathcal{L}^{(q)}$, with indices $v,\omega$ such that $\mathcal{B}=\mathcal{A}_{\omega}^{v}$. 
Definition (\ref{eq: separability separation}) implies the existence of an observable set $\chi(\mathcal{B}\mid_{\alpha\beta}^{ij})$ that is not separable, i.e., $Y(\mathcal{B})_{\alpha\beta}^{ij}\notin\mathbb{C}$ by Lemma \ref{lem: case B^(ij)_(ab) square}. 
By applying (\ref{eq: quadrilateral decomposition}), we express 
 \begin{equation}
Y(\mathcal{B})_{\alpha\beta}^{ij}=-Y(\mathcal{B})_{\alpha v}^{i\omega}\cdot Y(\mathcal{B})_{v\beta}^{i\omega}\cdot Y(\mathcal{B})_{\alpha v}^{\omega j}\cdot Y(\mathcal{B})_{v\beta}^{\omega j}. 
\label{eq: quadrilateral decomposition for separability transmission}
 \end{equation}
Since $Y(\mathcal{B})_{\alpha\beta}^{ij}\notin\mathbb{C}$, there exists a non-constant term on the right-hand side of (\ref{eq: quadrilateral decomposition for separability transmission}), say $Y(\mathcal{B})_{v\beta}^{i\omega}\notin\mathbb{C}$ by appropriately labelling $\{i,j\}$ and $\{\alpha,\beta\}$. 
From the transformation rule 
(\ref{eq: Grassmann-Plucker translation exchange, diagonal}), we obtain  $Y(\mathcal{A})_{\omega\beta}^{iv}\notin\mathbb{C}$, which is compatible with the separability hypothesis for $\mathcal{A}$ only if $\chi(\mathcal{A}\mid_{\omega\beta}^{iv})$ is not observable. From the condition $h(\mathcal{A}_{\omega}^{v})=h(\mathcal{B})\neq 0$, it follows that 
 \begin{equation}
    h(\mathcal{A}_{\beta}^{i})=0,\quad h(\mathcal{A}_{\omega\beta}^{iv})=h(\mathcal{B}_{\beta}^{i})=0, \quad h(\mathcal{A}_{\omega}^{i})\cdot h(\mathcal{A}_{\beta}^{v})=0. 
    \label{eq: structural equations from non-separable}
 \end{equation}
Conversely, since $\chi(\mathcal{B}|_{\alpha\beta}^{ij})$ is observable by construction, from the second condition in (\ref{eq: structural equations from non-separable}) we deduce 
\begin{equation}
h(\mathcal{A}_{\omega\alpha}^{iv})\cdot h(\mathcal{A}_{\omega\beta}^{jv})=h(\mathcal{B}_{\alpha}^{i})\cdot h(\mathcal{B}_{\beta}^{j})\neq 0 
\label{eq: observability of non-separable}
\end{equation} 
which implies that $v\notin\{i,j\}$, $\omega\notin\{\alpha,\beta\}$, and the sets $\chi(\mathcal{A}\mid_{\omega\alpha}^{iv})$ and $\chi(\mathcal{A}\mid_{\omega\beta}^{jv})$ are also observable. 
From the third condition in (\ref{eq: structural equations from non-separable}), $h(\mathcal{A}_{\omega}^{i})\neq0$ would imply $h(\mathcal{A}_{\beta}^{v})=0$, hence $h(\mathcal{A}_{\beta}^{j})\neq0$ in order for $h(\mathcal{A}_{\omega\beta}^{jv})\neq0$ to hold, as per (\ref{eq: observability of non-separable}); this would make both $Y(\mathcal{A})_{\omega\beta}^{jv}$ and $Y(\mathcal{A})_{\omega\beta}^{ji}$ observable (of the form (\ref{eq: special rational, binomial})), and hence constant, as well as their ratio $-Y(\mathcal{A})_{\omega\beta}^{iv}$ by (\ref{eq: associativity}), contradicting the hypothesis $Y(\mathcal{A})_{\omega\beta}^{iv}\notin\mathbb{C}$. By dual reasoning, we can use the third condition in (\ref{eq: structural equations from non-separable}) also to exclude $h(\mathcal{A}_{\beta}^{v})\neq0$. 
So $h(\mathcal{A}_{\omega}^{i})=h(\mathcal{A}_{\beta}^{v})=0$, and, to satisfy (\ref{eq: observability of non-separable}), we find $h(\mathcal{A}_{\alpha}^{i})\neq0\neq h(\mathcal{A}_{\beta}^{j})$. It follows that we can associate each $\pi\in\{\alpha,\beta,\omega\}=:\mathfrak{C}_{c}$ with an index $p\in\{i,j,v\}=:\mathfrak{C}_{r}$ such that $h(\mathcal{A}_{\pi}^{p})\neq0$, and vice versa---as a note, it can be directly verified that 
exactly three such pairs $(p,\pi)\in\mathfrak{C}_{r}\times\mathfrak{C}_{c}$ occur to make the aforementioned dependence conditions compatible with $Y(\mathcal{A})_{\omega\beta}^{iv}\notin\mathbb{C}$, as this precludes expressing $Y(\mathcal{A})_{\omega\beta}^{iv}$ as a product of observable (hence, constant) terms via (\ref{eq: quadrilateral decomposition}). 
We formalise such a correspondence $\varrho$ between $\mathfrak{C}_{c}$ and $\mathfrak{C}_{r}$ as 
$\varrho(\alpha):=i$, $\varrho(\beta):=j$, and $\varrho(\omega):=v$. We deduce   
    $Y(\mathcal{A})_{\alpha\beta}^{ij},Y(\mathcal{A})_{\omega\alpha}^{iv},Y(\mathcal{A})_{\omega\beta}^{jv}\in\mathbb{C}\setminus\{0,-1\}$, 
as these $Y$-terms originate from observable sets; furthermore, combining (\ref{eq: observability of non-separable}) with $h(\mathcal{B}_{\beta}^{i})=0$, as per (\ref{eq: structural equations from non-separable}), leads to the form (\ref{eq: special rational, binomial}) for $Y(\mathcal{A}_{\omega}^{v})_{\beta\alpha}^{ij}$. It follows that all coefficients in the polynomial (\ref{eq: diagonal minors, quadratic polynomial}) 
for $(a_{1},a_{2},a_{3}):=(v,i,j)$ and
$(\delta_{1},\delta_{2},\delta_{3}):=(\omega,\alpha,\beta)$ belong to $\mathbb{C}[\mathbf{t},\mathbf{t}^{-1}]$ (and, hence, to $\mathbb{F}$), where we have used (\ref{eq: g-factor for change of basis}) to obtain $\mathfrak{m}_{\omega\alpha\beta}^{vij}$; moreover, the discriminant (\ref{eq: discriminant}) can be a perfect square in $\mathbb{C}[\mathbf{t},\mathbf{t}^{-1}]$ only if it is constant. But this would imply $\mathfrak{m}_{\omega\alpha\beta}^{vij}\in\mathbb{C}$ as well, making all coefficients of (\ref{eq: diagonal minors, quadratic polynomial}) constant; since $Y(\mathcal{A})_{\omega\beta}^{iv}$ is a root of the equation 
    $P_{\omega\beta}^{iv}(X):=P_{\omega\alpha\beta}^{vij}\left(Y(\mathcal{A})_{\beta\alpha}^{ij}\cdot X\right)=0$,  
it would be constant too, contradicting $Y(\mathcal{A})_{\omega\beta}^{iv}\notin\mathbb{C}$. Thus, (\ref{eq: discriminant}) is not a perfect square, and by Lemma \ref{lem: quadratic extension from diagonal binomials} and Remark \ref{rem: hyperdeterminant},  
the term $Y(\mathcal{A})_{\omega\alpha}^{jv}$ 
is the unique conjugate root of $Y(\mathcal{A})_{\omega\beta}^{iv}$ over $\mathbb{F}$. 

We now invoke elements $g\in\mathcal{A}$ and $\kappa_{1},\kappa_{2}\in\mathcal{A}^{\mathtt{C}}$ such that (\ref{eq: weak key}) holds, as the existence of such elements is assumed for $\mathcal{I}$, so they exist for all bases, specifically for $\mathcal{A}$, by Remark \ref{rem: transmission weak key}. 
For each $u,w\in\{1,2\}$ and $\gamma_{1},\gamma_{2}\in\{\alpha,\beta,\omega\}$ with $\gamma_{1}\neq \gamma_{2}$ (hence, $\varrho(\gamma_{1})\neq\varrho(\gamma_{2})$), we also have $Y(\mathcal{A})_{\kappa_{u}\gamma_{w}}^{\varrho({\gamma_{w}})g}\in\mathbb{C}$, as it is observable in the separable basis $\mathcal{A}$, or it equals $-1$ for $g=\varrho(\gamma_{w})$ or $\kappa_{u}=\gamma_{w}$. Suppose  $Y(\mathcal{A})_{\gamma_{2}\gamma_{1}}^{\varrho(\gamma_{1})g}\notin\mathbb{C}$; then, from $Y(\mathcal{A})_{\gamma_{2}\gamma_{1}}^{\varrho(\gamma_{2})\varrho(\gamma_{1})}\in\mathbb{C}$ and (\ref{eq: associativity}), we infer that $Y(\mathcal{A})_{\gamma_{2}\gamma_{1}}^{\varrho(\gamma_{2})g}$ is also non-constant, implying that $g\notin\{\varrho(\gamma_{1}),\varrho(\gamma_{2})\}$ and $\kappa_{u}\notin\{\gamma_{1},\gamma_{2}\}$ for both $u\in\{1,2\}$ to exclude the constant values $Y(\mathcal{A})_{\kappa_{u}\gamma_{w}}^{\varrho({\gamma_{w}})g}$ for these terms. As above, the conditions $h(\mathcal{A}_{\gamma_{1}}^{g})=0$ and $h(\mathcal{A}_{\gamma_{2}\gamma_{1}}^{\varrho(\gamma_{1})g})=0$ are required to make $Y(\mathcal{A})_{\gamma_{2}\gamma_{1}}^{\varrho(\gamma_{2})g}$ and $Y(\mathcal{A})_{\gamma_{2}\gamma_{1}}^{\varrho(\gamma_{1})g}$ non-observable (being non-constant), respectively; the first condition guarantees $h(\mathcal{A}_{\kappa_{u}\gamma_{1}}^{\varrho(\gamma_{1})g})\cdot h(\mathcal{A}_{\gamma_{1}\gamma_{2}}^{\varrho(\gamma_{2})\varrho(\gamma_{1})})\neq 0$, ensuring that $Y(\mathcal{A}_{\gamma_{1}}^{\varrho(\gamma_{1})})_{\kappa_{u}\gamma_{2}}^{\varrho(\gamma_{2})g}$ is observable and, taking the form (\ref{eq: special rational, binomial}) by the second condition, belongs to $\mathbb{C}[\mathbf{t},\mathbf{t}^{-1}]$. 
Thus, we can adapt the previous argument by replacing $\mathfrak{C}_{c}$ and $\mathfrak{C}_{r}$ with $\{\kappa_{u},\gamma_{1},\gamma_{2}\}$ and $\{g,\varrho(\gamma_{1}),\varrho(\gamma_{2})\}$, respectively, finding that each discriminant $\Delta_{\kappa_{u}\gamma_{1}\gamma_{2}}^{g\varrho(\gamma_{1})\varrho(\gamma_{2})}$ 
is not a perfect square in $\mathbb{C}[\mathbf{t},\mathbf{t}^{-1}]$. 
However, in this case, both 
$Y_{\kappa_{1}\gamma_{1}}^{\varrho(\gamma_{1})\varrho(\gamma_{2})}$ and $Y_{\kappa_{2}\gamma_{1}}^{\varrho(\gamma_{1})\varrho(\gamma_{2})}$ 
would coincide with the unique conjugate of $Y_{\gamma_{2}\gamma_{1}}^{\varrho(\gamma_{1})g}$ over $\mathbb{F}$, leading to $Y_{\kappa_{1}\gamma_{1}}^{\varrho(\gamma_{1})\varrho(\gamma_{2})}=Y_{\kappa_{2}\gamma_{1}}^{\varrho(\gamma_{1})\varrho(\gamma_{2})}$, and, by (\ref{eq: associativity}), to $Y_{\kappa_{1}\kappa_{2}}^{\varrho(\gamma_{1})\varrho(\gamma_{2})}=-1$. Having $\kappa_{1}\neq\kappa_{2}$ and $\varrho(\gamma_{1})\neq\varrho(\gamma_{2})$, this contradicts Assumption \ref{claim: no Y=-1}. It follows that $\Delta_{\kappa_{u}\gamma_{1}\gamma_{2}}^{g\varrho(\gamma_{1})\varrho(\gamma_{2})}$ belongs to $\mathbb{C}$, 
and all the coefficients of $P_{\kappa_{u}\gamma_{1}\gamma_{2}}^{g\varrho(\gamma_{1})\varrho(\gamma_{2})}$ are constant; therefore, 
$Y_{\gamma_{1}\gamma_{2}}^{\varrho(\gamma_{s})g},Y_{\gamma_{2}\kappa_{u}}^{g\varrho(\gamma_{1})}\in\mathbb{C}$ for all $u,s\in \{1,2\}$ and $\gamma_{1},\gamma_{2}\in\{\alpha,\beta,\omega\}$. All the terms on the right-hand side of 
\begin{equation}
    Y(\mathcal{A})_{\omega\beta}^{iv}= Y(\mathcal{A})_{\omega\alpha}^{iv} \cdot\left(Y(\mathcal{A})_{\alpha\kappa_{2}}^{ig}\cdot Y(\mathcal{A})_{\kappa_{2}\beta}^{ig}\right) \cdot \left(Y(\mathcal{A})_{\alpha\kappa_{2}}^{gv}\cdot Y(\mathcal{A})_{\kappa_{2}\beta}^{gv}\right). 
\label{eq: pentagonal decomposition from weak key}
\end{equation} 
have this form, come from an observable set, or coincide with $-1$ (e.g., if $g\in\{v,i\}$). 
Since we have $Y(\mathcal{A})_{\omega\beta}^{iv}\in\mathbb{C}$, which leads to a contradiction, the thesis follows. 
\end{proof}
We now leverage the previous proof to derive a canonical form for factors $\mathbf{L}$ and $\mathbf{R}$ in (\ref{eq: Cauchy-Binet expansion}). 
\begin{thm}
\label{thm: retrieval of canonical forms}
Consider the family of minor products $( h(\mathcal{I}):\,\mathcal{I}\in\mathfrak{G}(\mathbf{L}))$ as the available information. If there exist $g\in\mathcal{I}$ and $\kappa_{1},\kappa_{2}\in \mathcal{I}^{\mathtt{C}}$ satisfying (\ref{eq: weak key}), then we can recover a canonical form $\left\{\mathbf{L^{\star}},\mathbf{R^{\star}}\right\}$ for the factors in (\ref{eq: Cauchy-Binet expansion}), i.e., a unique pair of matrices up to simultaneous transposition and exchange of the factors. 
\end{thm}
\begin{proof}
Each observable term $Y_{\alpha\beta}^{ij}$ is a root of the polynomial derived from (\ref{eq: quartic from Grassmann-Plucker, 2}), namely:  
\begin{equation}
F_{\alpha\beta}^{ij}(X) := \left(h(\mathcal{I}_{\beta}^{i})\cdot h(\mathcal{I}_{\alpha}^{j})\right)\cdot X^{2} 
+ \left(h(\mathcal{I}_{\beta}^{i})\cdot h(\mathcal{I}_{\alpha}^{j})-h(\mathcal{I})\cdot h(\mathcal{I}_{\alpha\beta}^{ij})+h(\mathcal{I}_{\alpha}^{i})\cdot h(\mathcal{I}_{\beta}^{j})\right)\cdot X 
+ h(\mathcal{I}_{\alpha}^{i})\cdot h(\mathcal{I}_{\beta}^{j}).
\label{eq: polynomial h to Y}
\end{equation}
To resolve potential ambiguity in assigning a root of $F_{\alpha\beta}^{ij}$ to $Y_{\alpha\beta}^{ij}$, we begin by identifying two cases where this association 
is \emph{unambiguous}: one occurs when $F_{\alpha\beta}^{ij}$ has coinciding roots, including $-1$ for the degenerate cases $i=j$ or $\alpha=\beta$ (as per Assumption \ref{claim: no Y=-1}); the other arises from non-key observable sets, which yield unique admissible $Y$-terms by Lemma \ref{lem: algebraic forms Y-terms} and the exclusion of $0$ and $-1$ for distinct indices. Next, we introduce the family $\mathcal{Y}$ of observable terms $Y_{\gamma\delta}^{lm}$ for which $(-1)^{c+1}\cdot Y_{\gamma\delta}^{lm}$ cannot be expressed, through iterated applications of (\ref{eq: associativity}), as a product of $c$ unambiguous terms for some $c\in\mathbb{N}_{0}$. 
Thus, each $Y_{\gamma\delta}^{lm}\in\mathcal{Y}$ must derive from a key whose associated polynomial (\ref{eq: polynomial h to Y}) has distinct roots. For every key $\{m_{1},m_{2}\}\times\{\gamma_{1},\gamma_{2}\}$, if there exists $s\in\mathcal{I}$ such that $h(\mathcal{I}_{\gamma_{1}}^{s})=0$, then we can express $Y_{\gamma_{1}\gamma_{2}}^{m_{1}m_{2}}$ as a product of unambiguous terms derived from (\ref{eq: special rational, binomial}); this follows by applying (\ref{eq: associativity}) with $Y_{\gamma_{1}\gamma_{2}}^{m_{1}s}$ and $Y_{\gamma_{1}\gamma_{2}}^{sm_{2}}$ when $h(\mathcal{I}_{\gamma_{2}}^{s})\neq0$, or otherwise by using (\ref{eq: quadrilateral decomposition}) with $Y_{\sigma\gamma_{w}}^{m_{u}s}$ for $u,w\in\{1,2\}$ and any $\sigma\in\mathcal{I}^{\mathtt{C}}$ such that $h(\mathcal{I}_{\sigma}^{s}) \neq 0$, which exists by Section \ref{par: non-trivial dependence pattern}. Dually, the same argument also applies to lower indices. 
Thus, for any $Y_{\gamma_{1}\gamma_{2}}^{m_{1}m_{2}}\in\mathcal{Y}$, the $Y$-terms obtained by a \emph{single-index change}, i.e., $Y_{\gamma_{1}\sigma}^{m_{1}m_{2}}$ or $Y_{\gamma_{1}\gamma_{2}}^{m_{1}s}$, originate from (weak) keys satisfying (\ref{eq: local context for representation of 4-square}). By selecting $a\in\mathcal{I}$ and $\omega\in\mathcal{I}^{\mathtt{C}}$ such that $h(\mathcal{I}_{\omega}^{a})=0$, we can also determine any $Y_{\gamma_{1}\gamma_{2}}^{m_{1}m_{2}}$ satisfying this condition, applying (\ref{eq: quadrilateral decomposition}) with unambiguous terms $Y_{\omega\gamma_{w}}^{am_{u}}$ ($u,w\in\{1,2\}$) in the form (\ref{eq: special rational, binomial}). This resolves ambiguity from $\mathcal{Y}$ when $\mathfrak{G}(\mathbf{L})\neq\wp_{k}[n]$ moving to a basis for which such $a$ and $\omega$ exist, as described in the proof of Proposition \ref{prop: constant 2-key}. This argument also implies that, for any $Y_{\gamma_{1}\gamma_{2}}^{m_{1}m_{2}}\in\mathcal{Y}$, both terms associated by a single-index change, e.g., $Y_{\gamma_{1}\sigma}^{m_{1}m_{2}}$ and $Y_{\sigma\gamma_{2}}^{m_{1}m_{2}}$, are observable and, by definition, at least one of them belongs to $\mathcal{Y}$.

If $\mathcal{Y}\neq\emptyset$, we choose a root of $F_{\alpha\beta}^{ij}$ for a fixed $Y_{\alpha\beta}^{ij}\in\mathcal{Y}$. By considering the relations between the roots and coefficients of a polynomial, as in \citet[Eq. (116)]{Angelelli2025}, this choice determines any $Y$-term in $\mathcal{Y}$ obtained from $Y_{\alpha\beta}^{ij}$ by a single-index change. 
For any $Y_{\gamma\delta}^{lm}\in\mathcal{Y}$, we follow the proof of \citet[Prop. 18]{Angelelli2025} to obtain sequences of $Y$-terms in $\mathcal{Y}$ that begin with $Y_{\alpha\beta}^{ij}$ and end with $Y_{\gamma\delta}^{lm}$, where each $Y$-term is obtained from the preceding one by a single-index change. By the previous argument, the choice of the root of $F_{\alpha\beta}^{ij}$ for $Y_{\alpha\beta}^{ij}$ determines the terms in these sequences, including $Y_{\gamma\delta}^{lm}$. 

We now extend this reasoning to the remaining $Y$-terms. If $\mathcal{Y}\neq \emptyset$, we fix a root for a given element, ensuring that every other $Y$-term in $\mathcal{Y}$ is determined by the previous procedure. 
Thus, if a term $Y_{\omega_{1}\omega_{2}}^{q_{1}q_{2}}\notin\mathcal{Y}$ remains undetermined, it is non-observable and cannot be expressed as a product of unambiguous or observable $Y$-terms by iterating (\ref{eq: associativity}). In particular, when expressing $Y_{\omega_{1}\omega_{2}}^{q_{1}q_{2}}$ as a product of terms $Y_{\kappa_{1}\omega_{w}}^{gq_{u}}$, at least one choice of $q\in\{q_{1},q_{2}\}$ and $\omega\in\{\omega_{1},\omega_{2}\}$ must return an undetermined, hence non-observable, term. This implies $h(\mathcal{I}_{\omega}^{q})=0$ and $h(\mathcal{I}_{\kappa_{1}}^{q})\cdot h(\mathcal{I}_{\omega}^{g})=0$, where $q\neq g$, as $Y_{\kappa_{1}\omega}^{gq}=-1$ would otherwise be unambiguous. The assumption in Section \ref{par: non-trivial dependence pattern} allows us to find $w\in\mathcal{I}$ and $\varrho\in\mathcal{I}^{\mathtt{C}}$ such that $h(\mathcal{I}_{\omega}^{w})\cdot h(\mathcal{I}_{\varrho}^{q})\neq 0$, implying $q\neq w$ and $\varrho \neq \omega$. 
If $h(\mathcal{I}_{\kappa_{1}}^{q})\neq 0$, then from $h(\mathcal{I}_{\kappa_{1}}^{q})\cdot h(\mathcal{I}_{\omega}^{g})=0$ we deduce $h(\mathcal{I}_{\omega}^{g})=0$; but in this case, $Y_{\kappa_{1}\omega}^{gq}=-Y_{\kappa_{1}\omega}^{wq}/Y_{\kappa_{1}\omega}^{wg}$ would be determined, as both terms on the right-hand side would be observable and take the form (\ref{eq: special rational, binomial}). By the same token, we obtain $h(\mathcal{I}_{\omega}^{g})=0$. 
Given these conditions, the case $h(\mathcal{I}_{\kappa_{2}}^{q})\neq 0$ would imply 
$Y_{\kappa_{1}\omega}^{gq}=Y_{\kappa_{1}\kappa_{2}}^{gq}\cdot Y_{\kappa_{2}\omega}^{gw}\cdot Y_{\kappa_{2}\omega}^{wq}$, which once again determines $Y_{\kappa_{1}\omega}^{gq}$ as the product of observable terms identified by (\ref{eq: special rational, binomial}). Thus, we conclude $h(\mathcal{I}_{\kappa_{2}}^{q})=0$. The vanishing of $h(\mathcal{I}_{\kappa_{1}}^{q})$, $h(\mathcal{I}_{\omega}^{q})$, $h(\mathcal{I}_{\omega}^{g})$, and $h(\mathcal{I}_{\kappa_{2}}^{q})$ implies $g \neq q\neq w \neq g$, $\varrho \neq \omega \neq \kappa_{s} \neq \varrho$, and $h(\mathcal{I}_{\kappa_{s}}^{g})\cdot h(\mathcal{I}_{\kappa_{s}\omega\varrho}^{gwq})\neq 0$ for both $s\in\{1,2\}$. It follows that $Y(\mathcal{I}_{\kappa_{s}}^{g})_{\omega\varrho}^{qw}$ is observable in the basis $\mathcal{I}_{\kappa_{s}}^{g}$ and determined by (\ref{eq: special rational, binomial}), since $h(\mathcal{I}_{\omega}^{g})=h(\mathcal{I}_{\omega}^{q})=0$ implies $h((\mathcal{I}_{\kappa_{s}}^{g})_{\omega}^{q})=0$. Therefore, $Y_{\omega\varrho}^{qw}$, $Y_{\varrho\kappa_{s}}^{gq}$, and $Y_{\omega\kappa_{s}}^{gw}$, as well as $Y(\mathcal{I}_{\kappa_{s}}^{g})_{\omega\varrho}^{qw}$, are determined for both $s\in\{1,2\}$, as they are observable and unambiguous. These establish the coefficients of the polynomials $P_{\varrho\kappa_{1}\omega}^{qgw}$ and $P_{\varrho\kappa_{2}\omega}^{qgw}$ in (\ref{eq: diagonal minors, quadratic polynomial}), via (\ref{eq: g-factor}), which are also determined. 
Denoting their respective roots as 
$Y_{\omega\kappa_{u}}^{gw}\cdot X_{+}^{(u)}$ and $Y_{\omega\kappa_{u}}^{gw}\cdot X_{-}^{(u)}$ for $u\in\{1,2\}$, with $X_{+}^{(1)}=X_{+}^{(2)}=Y_{\omega\varrho}^{qg}$, it follows from (\ref{eq: conjugate root}) that $X_{-}^{(u)}=Y_{\varrho\kappa_{u}}^{wq}$. 
From $w\neq q$ and $\kappa_{1}\neq \kappa_{2}$, we obtain $Y_{\kappa_{1}\kappa_{2}}^{wq}\neq-1$, which requires $X_{-}^{(1)} \neq  X_{-}^{(2)}$. Thus, the only common root of $P_{\varrho\kappa_{1}\omega}^{qgw}(Y_{\omega\kappa_{1}}^{gw}\cdot X)$ and $P_{\varrho\kappa_{2}\omega}^{qgw}(Y_{\omega\kappa_{2}}^{gw}\cdot X)$ is the unique possible value for $Y_{\omega\varrho}^{qg}$, which in turn determines $Y_{\kappa_{1}\omega}^{gq}=-Y_{\kappa_{1}\varrho}^{gq}\cdot Y_{\varrho\omega}^{gq}$, i.e., a contradiction. Hence, the $Y$-terms are determined, and we use them to obtain a canonical form for $\mathbf{R}$ by applying (\ref{eq: gauge invariance}) along with suitable changes of basis in $\mathbb{C}^{k}$. Specifically, after relabelling $\mathcal{I}=[k]$ and 
adopting the choice in \citet[Eqs. (35)--(38)]{Angelelli2025}, the transformed factor $\mathbf{R^{\star}}$ 
satisfies $R_{i,j}^{\star}=\delta_{ij}$ and $R_{\alpha,i}^{\star} = 
-Y_{\kappa_{1}\alpha}^{gi}$ for all $i,j\in\mathcal{I}$ and $\alpha\in\mathcal{I}^{\mathtt{C}}$, with 
$Y=-1$ for repeated upper or lower indices. This allows us to recover the transformed pattern matrix $\mathbf{L^{\star}}$ via $\Delta_{\mathbf{L^{\star}}}(\mathcal{I})=h(\mathcal{I})\cdot\Delta_{\mathbf{R^{\star}}}(\mathcal{I})^{-1}$. 
We can also fix the remaining change of basis by requiring $\mathbf{L}^{\star}|_{\mathcal{I}} = \mathrm{diag}(h(\mathcal{I}),1,\dots,1)$. 

The only remaining ambiguity is the choice of a root for a given element in $\mathcal{Y}$ when $\mathcal{Y}\neq \emptyset$. As discussed above, there is a basis where this ambiguity is resolved, unless $\mathfrak{G}(\mathbf{L}^{\star})=\wp_{k}[n]$; in the latter case, both factors $\mathbf{L}^{\star}$ and $\mathbf{R}^{\star}$ satisfy Assumption \ref{claim: no Y=-1}, and their roles are interchangeable, with the same construction applied after the exchange. The two factor pairs, $(\mathbf{L}^{\star},\mathbf{R}^{\star})$ and $((\mathbf{R}^{\star})^{\mathtt{T}},(\mathbf{L}^{\star})^{\mathtt{T}})$, yield the same list of minor products. Since there are only two such pairs (corresponding to the two roots for a given $Y$-term in $\mathcal{Y}$), they are the only possible configurations. 
\end{proof}
The previous proof does not hold if (\ref{eq: weak key}) is not satisfied. In particular, let  $\mathbf{L^{\star}}:=\left(\idd_{k}|\idd_{k}\right)$; then both the assignments $\left(\idd_{k}|\mathbf{r}\right)^{\mathtt{T}}$ and $\left(\idd_{k}|\mathbf{r}^{\mathtt{T}}\right)^{\mathtt{T}}$ to $\mathbf{R^{\star}}$ yield the same minor products. 
\begin{lem}
\label{lem: independence on base-set} 
If all the bases in $\mathfrak{G}(\mathbf{L})$ are separable, then for all $i\triangledown_{\mathcal{J}}\alpha$, according to (\ref{eq: distinguishability relation on columns}), the function
 \begin{equation}
\psi_{2}(i;\alpha):=\Psi\left(h(\mathcal{J})^{-1}\cdot h(\mathcal{J}_{\alpha}^{i})\right)=\Psi\left(h(\mathcal{J}_{\alpha}^{i})\right)-\Psi\left(h(\mathcal{J})\right)
\label{eq: quantification distinguishable columns}
 \end{equation}
does not depend on the choice of the set $\mathcal{J}\in\mathfrak{G}(\mathbf{L})$ such that $\mathcal{J}_{\alpha}^{i}\in\mathfrak{G}(\mathbf{L})$ as well. Specifically, this holds under the hypotheses of Theorem \ref{thm: transfer of separability}, and in particular when Assumption \ref{claim: existence 3-key} is satisfied. 
\end{lem}
\begin{proof}
Take two different sets $\mathcal{I}_{1}:=\mathcal{I}$ and $\mathcal{J}_{1}:=\mathcal{J}$ in $\mathfrak{G}(\mathbf{L})$ such that $\mathcal{I}_{2}:=\mathcal{I}_{\alpha}^{i}$ and $\mathcal{J}_{2}:=\mathcal{J}_{\alpha}^{i}$ also belong to $\mathfrak{G}(\mathbf{L})$. Note that this implies $i\in\mathcal{I}\cap\mathcal{J}$, $\alpha\in\mathcal{I}^{\mathtt{C}}\cap\mathcal{J}^{\mathtt{C}}$, $\mathcal{I}_{1}\setminus\mathcal{J}_{1}=\mathcal{I}_{2}\setminus\mathcal{J}_{2}$, and $\mathcal{J}_{1}\setminus\mathcal{I}_{1}=\mathcal{J}_{2}\setminus\mathcal{I}_{2}$, so we can omit the subscripts in these difference sets. The proof proceeds by induction on $\kappa:=|\mathcal{I}\setminus\mathcal{J}|$. We begin with the base case $\kappa=1$, setting $\mathcal{I}\setminus\mathcal{J}=:\{j\}$ and $\mathcal{J}\setminus\mathcal{I}=:\{\beta\}$. 
The condition $\mathcal{I}_{1},\mathcal{J}_{1},\mathcal{I}_{2},\mathcal{J}_{2}\in\mathfrak{G}(\mathbf{L})$ implies that $\chi(\mathcal{I}|_{\alpha\beta}^{ij})$ is observable, and thus the separability of the bases in $\mathfrak{G}(\mathbf{L})$ 
establishes the thesis. 

Now, assume the thesis holds for all $\mathcal{I},\mathcal{J}\in\mathfrak{G}(\mathbf{L})$ such that $|\mathcal{I}\setminus\mathcal{J}|\leq\kappa$, and consider any $4$-tuple of bases $(\mathcal{I},\mathcal{I}_{\alpha}^{i},\mathcal{J},\mathcal{J}_{\alpha}^{i})$ with $|\mathcal{I}\setminus\mathcal{J}|=\kappa+1$. 
If there exist $j\in\mathcal{I}\setminus\mathcal{J}$ and $\beta\in\mathcal{J}\setminus\mathcal{I}$ such that $h(\mathcal{I}_{\beta}^{j})\cdot h(\mathcal{I}_{\alpha\beta}^{ij})\neq0$, 
then we obtain the thesis as 
\begin{equation}
\Psi\left(h(\mathcal{I})^{-1}\cdot h(\mathcal{I}_{\alpha}^{i})\right) =  \Psi\left(h(\mathcal{I}_{\beta}^{j})^{-1}\cdot h((\mathcal{I}_{\beta}^{j})_{\alpha}^{i})\right) 
=  \Psi\left(h(\mathcal{J})^{-1}\cdot h(\mathcal{J}_{\alpha}^{i})\right) 
\label{eq: split of the lemma thesis}
\end{equation}
where the first equality 
holds as in the base case, while the second follows from the induction hypothesis, as $\mathcal{I}_{\beta}^{j},\mathcal{I}_{\beta\alpha}^{ji}\in\mathfrak{G}(\mathbf{L})$ and $|\mathcal{I}_{\beta}^{j}\setminus\mathcal{J}|=\kappa$. Analogous expressions hold whenever 
$h(\mathcal{J}_{j}^{\beta})\cdot h(\mathcal{J}_{\alpha j}^{i\beta})\neq0$, by substituting $(\mathcal{I},\mathcal{J})\mapsto(\mathcal{J},\mathcal{I})$ and $(j,\beta)\mapsto(\beta,j)$. Otherwise, the absence of such $j$ and $\beta$ means that 
\begin{equation}
\text{for all } j\in\mathcal{I}\setminus\mathcal{J},\,\beta\in\mathcal{J}\setminus\mathcal{I}:\quad h(\mathcal{I}_{\beta}^{j})\cdot h(\mathcal{I}_{\alpha\beta}^{ij})=h(\mathcal{J}_{j}^{\beta})\cdot h(\mathcal{J}_{\alpha j}^{i\beta})=0.
\label{eq: complementary pattern under (i,a)}
\end{equation} 
We can apply (\ref{eq: symmetric exchange relation}) to both the bases $(\mathcal{I}_{s},\mathcal{J}_{s})$ with $s\in\{1,2\}$, finding that, for each $j\in\mathcal{I}\setminus\mathcal{J}$, there exists an element $\beta_{s}\in\mathcal{J}\setminus\mathcal{I}$ such that $h((\mathcal{I}_{s})_{\beta_{s}}^{j})\cdot h((\mathcal{J}_{s})_{j}^{\beta_{s}})\neq0$. Considering (\ref{eq: complementary pattern under (i,a)}), $h(\mathcal{J}_{\alpha j}^{i\beta_{2}})\neq0$ requires $h(\mathcal{J}_{j}^{\beta_{2}})=0$, and combining these two conditions implies $h(\mathcal{J}_{j}^{i})\neq0$; similarly, $h(\mathcal{I}_{\beta_{1}}^{j})\neq0$ forces $h(\mathcal{I}_{\alpha\beta_{1}}^{ij})=0$ by (\ref{eq: complementary pattern under (i,a)}), which, along with $h(\mathcal{I}_{\alpha}^{i})\neq0$, yields 
$h(\mathcal{I}_{\alpha}^{j})\neq0$. As before, an analogous argument applies starting with any $\beta\in\mathcal{J}\setminus\mathcal{I}$, and we conclude   
\begin{equation} 
\text{for all } j\in\mathcal{I}\setminus\mathcal{J},\,\beta\in\mathcal{J}\setminus\mathcal{I}:\quad \mathcal{I}_{\alpha}^{j},\mathcal{I}_{\beta}^{i},\mathcal{J}_{\alpha}^{\beta},\mathcal{J}_{j}^{i}\in\mathfrak{G}(\mathbf{L}). 
\label{eq: new bases from complementary pattern}
\end{equation} 

Given these premises, we take $\pi\in\mathcal{J}\setminus\mathcal{I}$, seen as an element of $\mathcal{J}_{\alpha}^{i}$, and apply (\ref{eq: symmetric exchange relation}) to identify an index $p\in\mathcal{I}\setminus\mathcal{J}$ such that $h((\mathcal{J}_{\alpha}^{i})_{p}^{\pi})\neq0$. From (\ref{eq: complementary pattern under (i,a)}), we infer 
$h(\mathcal{J}_{p}^{\pi})=0$, while (\ref{eq: new bases from complementary pattern}) yields $h(\mathcal{J}_{\alpha}^{\pi})\cdot h(\mathcal{I}_{\alpha}^{p})\cdot h(\mathcal{J}_{p}^{i})\neq0$; combining these conditions leads to $h(\mathcal{J}_{\alpha p}^{\pi i})\neq0$. So we concentrate on the bases $\mathcal{I}_{\alpha}^{p}$, $(\mathcal{I}_{\alpha}^{p})_{p}^{i}=\mathcal{I}_{\alpha}^{i}$, $\mathcal{J}_{\alpha}^{\pi}$, and $\mathcal{J}_{\alpha p}^{\pi i}$. From $|\mathcal{I}_{\alpha}^{p}\setminus\mathcal{J}_{\alpha}^{\pi}|=\kappa$, we can apply the induction hypothesis to the bases $(\mathcal{I}_{\alpha}^{p},\mathcal{J}_{\alpha}^{\pi})$ and indices $(i,p)$, establishing 
\begin{equation} 
\Psi\left(h(\mathcal{I}_{\alpha}^{i})\right)-\Psi\left(h(\mathcal{I}_{\alpha}^{p})\right) 
= \Psi\left(h(\mathcal{J}_{\alpha p}^{\pi i})\right)-\Psi\left(h(\mathcal{J}_{\alpha}^{\pi})\right)
=  \Psi\left(h(\mathcal{J}_{p}^{i})\right)-\Psi\left(h(\mathcal{J})\right)
\label{eq: pivoting 1 c} 
\end{equation}
where the first equality follows from the induction hypothesis, and the second from the base case. Analogously, applying (\ref{eq: symmetric exchange relation}) to the bases $\mathcal{I}$, $\mathcal{J}$, and the same index $p\in\mathcal{I}\setminus\mathcal{J}$ identified above, we can find $\varrho\in\mathcal{J}\setminus\mathcal{I}$ with $ h(\mathcal{J}_{p}^{\varrho})\neq0$. Now we consider  
the bases $\mathcal{I}_{\alpha}^{p}$, $(\mathcal{I}_{\alpha}^{p})_{p}^{\alpha}=\mathcal{I}$, $\mathcal{J}_{\alpha}^{\varrho}$ (which belongs to $\mathfrak{G}(\mathbf{L})$ by (\ref{eq: new bases from complementary pattern})), and $(\mathcal{J}_{\alpha}^{\varrho})_{p}^{\alpha}=\mathcal{J}_{p}^{\varrho}$. Focusing on the pair $(\alpha,p)$, the induction hypothesis applies to the bases $(\mathcal{I}_{\alpha}^{p},\mathcal{J}_{\alpha}^{\varrho})$ since $|\mathcal{I}_{\alpha}^{p}\setminus\mathcal{J}_{\alpha}^{\varrho}|=\kappa$, and combining it with the base case applied to $(\mathcal{J}_{\alpha}^{\varrho},\mathcal{J}_{\alpha}^{i})$ yields 
\begin{equation} 
\Psi\left(h(\mathcal{I})\right)-\Psi\left(h(\mathcal{I}_{\alpha}^{p})\right) 
= \Psi\left(h(\mathcal{J}_{p}^{\varrho})\right)-\Psi\left(h(\mathcal{J}_{\alpha}^{\varrho})\right) 
= \Psi\left(h(\mathcal{J}_{p}^{i})\right)-\Psi\left(h(\mathcal{J}_{\alpha}^{i})\right).
\label{eq: pivoting 2 c}
\end{equation}
Comparing 
(\ref{eq: pivoting 1 c}) and 
(\ref{eq: pivoting 2 c}), we obtain $\Psi\left(h(\mathcal{I}_{\alpha}^{i})\right)+\Psi\left(h(\mathcal{J})\right)$ $ = \Psi\left(h(\mathcal{I}_{\alpha}^{p})\right)+\Psi\left(h(\mathcal{J}_{p}^{i})\right) $ $= \Psi\left(h(\mathcal{I})\right)+\Psi\left(h(\mathcal{J}_{\alpha}^{i})\right)$. Therefore, $\Psi\left(h(\mathcal{I}_{\alpha}^{i})\right) - \Psi\left(h(\mathcal{I})\right)= \Psi\left(h(\mathcal{J}_{\alpha}^{i})\right)-\Psi\left(h(\mathcal{J})\right)$, which concludes the proof. 
\end{proof}
\begin{thm}
\label{thm: equality phases monomial generalised}
When all the bases in $\mathfrak{G}(\mathbf{L})$ are separable---especially under the hypothesis of Theorem \ref{thm: transfer of separability} and, in particular, when Assumption \ref{claim: existence 3-key} is verified---(\ref{eq: separability of set function}) 
holds.
\end{thm}
\begin{proof} 
For each $\mathcal{I}\in\mathfrak{G}(\mathbf{L})$, denote by $\bar{\triangledown}_{\mathcal{I}}$ the \emph{transitive closure} of the relation $\triangledown_{\mathcal{I}}$ introduced in (\ref{eq: distinguishability relation on columns})---namely, $\alpha\bar{\triangledown}_{\mathcal{I}}\omega$ means that there exists a finite sequence $(\delta_{1},\dots,\delta_{\kappa})$
with $\delta_{1}=\alpha$, $\delta_{\kappa}=\omega$, and $\delta_{i}\triangledown_{\mathcal{I}}\delta_{i+1}$
for all $i\in[\kappa-1]$. For each such sequence, we extend (\ref{eq: quantification distinguishable columns}) by introducing  
 \begin{equation}
\bar{\psi}_{2}(\alpha;\omega):=\sum_{i=1}^{\kappa-1}\psi_{2}(\delta_{i};\delta_{i+1}).
\label{eq: extension of exponent differences}
 \end{equation}
Multiple sequences may connect $\alpha$ with $\omega$ and serve to evaluate $\bar{\psi}_{2}(\alpha;\omega)$. Given two such sequences, we can concatenate the first with the \emph{reversal} of the second, which returns from $\omega$ to $\alpha$; the reversed sequence is valid due to the symmetry of relation (\ref{eq: distinguishability relation on columns}), and by Lemma \ref{lem: independence on base-set}, we have 
\begin{equation} 
\text{for all }\mathcal{I},\mathcal{I}_{\alpha}^{i}\in\mathfrak{G}(\mathbf{L}):\quad\psi_{2}(i;\alpha)=\Psi(\mathcal{I}_{\alpha}^{i})-\Psi(\mathcal{I})=\Psi(\mathcal{I}_{\alpha}^{i})-\Psi((\mathcal{I}_{\alpha}^{i})_{i}^{\alpha})=-\psi_{2}(\alpha;i)  
\label{eq: psi2 reversal}
\end{equation} 
so the reversal of a sequence connecting $\alpha$ with $\omega$ changes the sign of each summand in (\ref{eq: extension of exponent differences}). Thus, two distinct sequences return the same value $\bar{\psi}_{2}(\alpha;\omega)$ only if this concatenation, which yields a \emph{closed path} starting and ending with $\alpha$, satisfies $\bar{\psi}_{2}(\alpha;\alpha)=0$. We use this observation to prove that (\ref{eq: extension of exponent differences}) does not depend on the choice of the specific sequence $(\delta_{1},\dots,\delta_{\kappa})$, by focusing on such closed paths. 

Note that $\delta_{u}\in\mathcal{I}$ if and only if $\delta_{u+1}\in\mathcal{I}^{\mathtt{C}}$. Thus, each closed path contains an odd number of indices, including the coinciding endpoints. Let this number be $2\cdot p+1$, with $\delta_{1}=\delta_{2p+1}$. We now prove that $\bar{\psi}_{2}(\alpha;\alpha)=0$, for all bases and closed paths, by induction on $p$. The base case $p=1$ follows from (\ref{eq: psi2 reversal}) by associating each of the sequences $(\delta_{1},\delta_{2})$ and $(\delta_{2},\delta_{1})$ with a distinct pair $(i,\alpha)$ and $(\alpha,i)$. For $p=2$, the claim follows from separability, with (\ref{eq: psi2 reversal}) accounting for repeated indices. Even in the subsequent inductive step, repeated indices beyond the coinciding endpoints allow splitting the path into shorter closed paths. 
Next, we assume the claim holds for all $u\leq p$ and consider any closed path $(\delta_{1},\dots,\delta_{2p+3})$ with $\delta_{1}=\delta_{2p+3}$ and $\delta_{i}\triangledown_{\mathcal{I}}\delta_{i+1}$ for all $i\in[2p+2]$. To simplify notation, we consider indices modulo $2p+2$, i.e., $\delta_{2p+2+l}=\delta_{l}$ for all $l\in[2p+2]$, and, if necessary, perform a cyclic shift of labels to ensure $\delta_{1}\in\mathcal{I}$. If we can find $\delta_{u},\delta_{w}$ such that $1<w-u<2p+1$ and
$\delta_{u}\triangledown_{\mathcal{I}}\delta_{w}$, then by (\ref{eq: psi2 reversal}) we can write 
\begin{equation*}
\bar{\psi}_{2}(\delta_{1};\delta_{1}) 
= \left(\sum_{i=u+1}^{w}\psi_{2}(\delta_{i-1};\delta_{i})+\psi_{2}(\delta_{w};\delta_{u})\right) 
+\left(\psi_{2}(\delta_{u};\delta_{w})+\sum_{j=w+1}^{2p+2+u}\psi_{2}(\delta_{j-1};\delta_{j})\right).
\end{equation*}
Each of the two bracketed sums is over a closed path with length at most $2p$. Therefore, the induction hypothesis applies, and we have $\bar{\psi}_{2}(\delta_{1};\delta_{1})=0$. Conversely, the absence of related pairs $\delta_{u}\triangledown_{\mathcal{I}}\delta_{w}$ with non-adjacent indices---i.e., when $1 < \lvert u-w \rvert < 2\cdot p + 1$---implies, for all $a,b\in\{1,-1\}$ and $M\neq0$ with $\delta_{2u+a}\neq\delta_{2u+2M+b}$ 
\begin{equation} 
h(\mathcal{I}_{\delta_{2u}}^{\delta_{2u+a}})\cdot h(\mathcal{I}_{\delta_{2u+2M}}^{\delta_{2u+2M+b}})\neq0,\quad h(\mathcal{I}_{\delta_{2u+2M}}^{\delta_{2u+a}})\cdot h(\mathcal{I}_{\delta_{2u}}^{\delta_{2u+2M+b}})=0, \quad h(\mathcal{I}_{\delta_{2u}\delta_{2u+2M}}^{\delta_{2u+a}\delta_{2u+2M+b}})\neq0 
\label{eq: chain self-correlation}
\end{equation} 
where the first follows from the condition defining the sequence, the second product vanishes since at least one factor involves two non-adjacent indices, and the third follows from the previous two conditions by the three-term Grassmann-Pl\"{u}cker relations. Fixing $u=1$, the choice $a=1$ in (\ref{eq: chain self-correlation}) yields $h(\mathcal{I}_{\delta_{2}\delta_{2M+2}}^{\delta_{3}\delta_{2M+2+b}})\neq0$ for all $b\in\{-1,1\}$ and $M\notin\{0,1\}$, while taking $M=1$ and $a=b$ gives $h(\mathcal{I}_{\delta_{2}\delta_{4}}^{\delta_{3}\delta_{5}})\neq 0$ (when $a=1$) and $h(\mathcal{I}_{\delta_{2}\delta_{4}}^{\delta_{1}\delta_{3}})\neq 0$ (when $a=-1$). Thus, the sequence $(\delta_{1},\delta_{4},\dots,\delta_{2p+2},\delta_{1})$ is valid for the basis $\mathcal{I}_{\delta_{2}}^{\delta_{3}}$, and, together with Lemma \ref{lem: independence on base-set}, the latter condition gives 
\begin{align}
\psi_{2}(\delta_{1};\delta_{2})+\psi_{2}(\delta_{2};\delta_{3})+\psi_{2}(\delta_{3};\delta_{4}) 
& = \Psi\left(\frac{h(\mathcal{I}_{\delta_{2}}^{\delta_{1}})}{h(\mathcal{I})}\right)+\Psi\left(\frac{h(\mathcal{I})}{h(\mathcal{I}_{\delta_{2}}^{\delta_{3}})}\right)+\Psi\left(\frac{h(\mathcal{I}_{\delta_{4}}^{\delta_{3}})}{h(\mathcal{I})}\right) \nonumber \\ 
& = \Psi\left(\frac{h(\mathcal{I}_{\delta_{2}}^{\delta_{1}})}{h(\mathcal{I}_{\delta_{2}}^{\delta_{3}})}\right)+\Psi\left(\frac{h(\mathcal{I}_{\delta_{4}}^{\delta_{3}})}{h(\mathcal{I})}\right) 
= 
\Psi\left(\frac{h(\mathcal{I}_{\delta_{2}}^{\delta_{1}})}{h(\mathcal{I}_{\delta_{2}}^{\delta_{3}})}\right)+\Psi\left(\frac{h((\mathcal{I}_{\delta_{2}}^{\delta_{1}})_{\delta_{4}}^{\delta_{3}})}{h(\mathcal{I}_{\delta_{2}}^{\delta_{1}})}\right)
\nonumber \\
& = \Psi\left(\frac{h((\mathcal{I}_{\delta_{2}}^{\delta_{3}})_{\delta_{4}}^{\delta_{1}})}{h(\mathcal{I}_{\delta_{2}}^{\delta_{3}})}\right) = \psi_{2}(\delta_{1};\delta_{4}) 
\label{eq: renormalise sub-sequence}
\end{align} 
yielding the same value for $\bar{\psi}_{2}(\delta_{1};\delta_{1})$ as in the basis $\mathcal{I}$. The length of $(\delta_{1},\delta_{4},\dots,\delta_{2p+2},\delta_{1})$ is $2p+1$, so the induction hypothesis applies in $\mathcal{I}_{\delta_{2}}^{\delta_{3}}$, yielding $\bar{\psi}_{2}(\delta_{1};\delta_{1})=0$ and proving the claim. 

Finally, we construct the function $\psi$ in (\ref{eq: separability of set function}). Fix an arbitrary set $\mathcal{I}\in\mathfrak{G}(\mathbf{L})$; the relation $\bar{\triangledown}_{\mathcal{I}}$ is reflexive, allowing paths of length zero, by (\ref{eq: distinguishability relation on columns}), symmetric due to the validity of reversed sequences, and transitive by construction, so it is an equivalence. Select  
a representative $\overline{i_{c}}$ for each equivalence class $c$ of $\bar{\triangledown}_{\mathcal{I}}$, and assign a $d$-tuple $\psi(\overline{i_{c}})\in\mathbb{Z}^{d}$ to each such index. 
Next, for each $\alpha\in[n]$ belonging to the same class as $\overline{i_{c}}$, define 
 \begin{equation}
\psi(\alpha):=\psi(\overline{i_{c}})+\bar{\psi}_{2}(\overline{i_{c}};\alpha).
\label{eq: set-to-element construction}
 \end{equation}
If $\mathcal{I}_{\alpha}^{i}\in\mathfrak{G}(\mathbf{L})$, then $i$ and $\alpha$ belong to the same class, and so does $\overline{i_{c}}$. Thus, (\ref{eq: extension of exponent differences}) and (\ref{eq: set-to-element construction}) give 
\begin{equation}
\bar{\psi}_{2}(i;\alpha) 
= \bar{\psi}_{2}(i;\overline{i_{c}})+\bar{\psi}_{2}(\overline{i_{c}};\alpha) 
= \psi(\alpha)-\psi(i)\quad \text{for }\mathcal{I}_{\alpha}^{i}\in\mathfrak{G}(\mathbf{L}). 
\label{eq: additivity of evaluation of paths}
\end{equation} 
For any other $\mathcal{J}\in\mathfrak{G}(\mathbf{L})$, choose an ordering $(\alpha_{1},\dots,\alpha_{r})$ of $\mathcal{J}\setminus\mathcal{I}$. We construct a sequence of bases by setting $\mathcal{T}_{0}:=\mathcal{J}$ and $\mathcal{T}_{u}:=(\mathcal{T}_{u-1})_{i_{u}}^{\alpha_{u}}$ for $u\in[r]$, where each $i_{u}$ satisfies $(\mathcal{T}_{u-1})_{i_{u}}^{\alpha_{u}},\mathcal{I}_{\alpha_{u}}^{i_{u}}\in\mathfrak{G}(\mathbf{L})$ as per (\ref{eq: symmetric exchange relation}): such $i_{u}$ is distinct from any $\alpha_{s}$ since it belongs to $\mathcal{I}$, and for any subsequent step $w>u$, we have $i_{u}\in\mathcal{T}_{w-1}$ and $i_{w}\notin\mathcal{T}_{w-1}$, so $i_{u}$ is distinct from all later $i_{w}$. We conclude 
\begin{align}
\Psi(\mathcal{J}) 
= \Psi(\mathcal{I})+\sum_{u=0}^{r-1}\Psi(\mathcal{T}_{u})-\Psi((\mathcal{T}_{u})_{i_{u+1}}^{\alpha_{u+1}}) 
& 
= \Psi(\mathcal{I})+\sum_{u=0}^{r-1}\Psi(\mathcal{I}_{\alpha_{u+1}}^{i_{u+1}})-\Psi(\mathcal{I})\nonumber && \text{(by Lemma \ref{lem: independence on base-set})}\\
& 
= \Psi(\mathcal{I})+\sum_{u=1}^{r}\bar{\psi}_{2}(i_{u};\alpha_{u}) && \text{(by (\ref{eq: quantification distinguishable columns}))} \nonumber \\ 
& = \Psi(\mathcal{I})+\sum_{u=1}^{r}\psi(\alpha_{u})-\psi(i_{u}) && \text{(by (\ref{eq: additivity of evaluation of paths}))}.
\label{eq: construct J from I}
\end{align}
Hence, (\ref{eq: separability of set function}) follows by setting 
$\mathbf{m_{0}}:=\Psi(\mathcal{I})-\sum_{i\in\mathcal{I}\setminus\mathcal{J}}\psi(i) -\sum_{j\in\mathcal{I}\cap\mathcal{J}}\psi(j) = \Psi(\mathcal{I})-\sum_{i\in\mathcal{I}}\psi(i)$. 
\end{proof}


\begin{rem}
\label{rem: constant Y-terms}
It is straightforward to verify that (\ref{eq: separable deformations}) produces constant $Y$-terms. Conversely, when all $Y$-terms are constant, each basis in $\mathfrak{G}(\mathbf{L})$ is separable by Lemma \ref{lem: case B^(ij)_(ab) square}. Lemma \ref{lem: independence on base-set} and the proof of Theorem \ref{thm: equality phases monomial generalised} then apply, yielding a representation of the minor products of the form (\ref{eq: separable deformations}), up to the common monomial factor. This characterises the configurations under consideration as those where the $Y$-terms remain invariant under the weighting (\ref{eq: toric deformation}). 
\end{rem}

\section{Counterexamples}
\label{sec: counterexamples} 

When the conditions underlying the results in Sections \ref{sec: principal minors and hyperdeterminants}--\ref{sec: propagation of separability between different bases} are not satisfied, separability may fail. In what follows, we construct counterexamples by relaxing the individual conditions that define non-planar keys. A Mathematica notebook \citep{Mathematica} is available, containing additional verifications and details on these constructions. Throughout this section, subscripts are used, where appropriate, to explicitly indicate the dimensions of block matrices. 

\subsection{Reduction to Principal Minors}
\label{subsubsec: Reduction to Principal Minor Assignment}

We begin with an example where $\mathfrak{G}(\mathbf{L})$ provides minimal structural information. Consider 
a generic skew-symmetric constant matrix $\mathbf{S}\in\mathbb{C}^{k\times k}$ and define  
\begin{equation}
\mathbf{L}_{0}:=(\idd_{k}\mid\idd_{k})\in\mathbb{C}^{k\times(2k)},\quad 
\mathbf{r}_{0}(\tau) := \tau\cdot\mathbf{1}_{k}\cdot\mathbf{1}_{k}^{\mathtt{T}}+\mathbf{S},\quad \mathbf{R}_{0}(\tau) := (\idd_{k}\mid\mathbf{r}_{0})^{\mathtt{T}}
\label{eq: special non-separable control matrix}
\end{equation}
where $\tau$ is a non-constant monomial satisfying $\tau(\mathbf{1})=1$ and $\mathbf{1}_{k}:=(1,\dots,1)\in\mathbb{C}^{k}$. As we now illustrate, these matrices satisfy (\ref{eq: monomial terms of Cauchy-Binet expansion}), but $\mathfrak{G}(\mathbf{L}_{0})$ does not satisfy Assumption \ref{claim: existence 3-key}. In particular, we state the following: 
\begin{prop}
\label{prop: principal minors odd non-separable}
For a configuration defined by 
(\ref{eq: special non-separable control matrix}), where $\mathbf{S}$ is a generic $(k\times k)$ skew-symmetric matrix, the following holds: 
\begin{equation}
    \text{for all }\mathcal{I}\in\mathfrak{G}(\mathbf{L}_{0}): \quad \Delta_{\mathbf{R}_{0}(\tau)}(\mathcal{I})\text{ is constant}  \Leftrightarrow 
    \left|[k]\setminus\mathcal{I}\right| \text{ is even}.
    \label{eq: alternating constantness}
\end{equation}
This configuration, whose observable terms arise from the principal minors of $\mathbf{r}_{0}(\tau)$ by Remark \ref{rem: transmission weak key}, is not separable. 
\end{prop}
\begin{proof}
Set $R(\mathcal{I}):=[k]\setminus\mathcal{I}$ and $C(\mathcal{I}):=\mathcal{I}\setminus[k]$, with $|R(\mathcal{I})|=:r$. Then, the minor $\Delta_{\mathbf{R}_{0}(\tau)}(\mathcal{I})$ coincides, up to a permutation sign depending only on $\mathcal{I}$, with $\det\left(\mathbf{r}_{0}(\tau)_{C(\mathcal{I})}^{R(\mathcal{I})}\right)$, where $\mathbf{r}_{0}(\tau)_{\mathcal{B}}^{\mathcal{A}}$ denotes the submatrix of 
$\mathbf{r}_{0}(\tau)$ with row indices indexed by $\mathcal{A}\subseteq [k]$ and column indices indexed by $\{\beta-k:\,\beta\in\mathcal{B}\}$, for $\mathcal{B}\subseteq [k]^{\mathtt{C}}$. 
Whenever $r$ is even, the matrix determinant lemma \citep[Eq. (0.8.5.11)
]{Horn2012} yields 
\begin{equation}
\Delta_{\mathbf{R}_{0}(\tau)}(\mathcal{I})=\left(1+\tau\cdot\mathbf{1}_{r}^{\mathtt{T}}\cdot\left(\mathbf{r}_{0}(0)_{C(\mathcal{I})}^{R(\mathcal{I})}\right)^{-1}\cdot\mathbf{1}_{r}\right)\cdot\Delta_{\mathbf{R}_{0}(0)}(\mathcal{I})=\Delta_{\mathbf{R}_{0}(0)}(\mathcal{I})\in\mathbb{C},\quad \mathcal{I}\in\mathfrak{G}(\mathbf{L}_{0})
\label{eq: matrix determinant lemma}
\end{equation} 
since a skew-symmetric matrix $\left(\mathbf{r}_{0}(0)_{C(\mathcal{I})}^{R(\mathcal{I})}\right)^{-1}$ satisfies $\mathbf{v}^{\mathtt{T}}\cdot\left(\mathbf{r}_{0}(0)_{C(\mathcal{I})}^{R(\mathcal{I})}\right)^{-1}\cdot\mathbf{v}=0$ for all $\mathbf{v}\in\mathbb{C}^{r}$, in particular when $\mathbf{v}=\mathbf{1}_{r}$. Note that the same permutation sign results for $\Delta_{\mathbf{R}_{0}(\tau)}(\mathcal{I})$ and $\Delta_{\mathbf{R}_{0}(0)}(\mathcal{I})$, so it has no effect on (\ref{eq: matrix determinant lemma}). 
For sets $\mathcal{I}\in\mathfrak{G}(\mathbf{L}_{0})$ such that $r$ is odd, choose any $\alpha\in C(\mathcal{I})$, recall the mapping $\varrho$ from Remark \ref{rem: transmission weak key}, and apply the Schur formula \citep[Eq. (0.8.5.1)]{Horn2012} 
\begin{equation} 
\det\left(\mathbf{r}_{0}(\tau)_{C(\mathcal{I})}^{R(\mathcal{I})}\right)
=\det(\mathbf{r}_{(\alpha)}(\tau))\cdot \det\left(\mathbf{r}_{0}(\tau)_{\{\alpha\}}^{\{\varrho(\alpha)\}}\right)
\label{eq: Schur formula}
\end{equation} 
where $\mathbf{r}_{(\alpha)}(\tau)$ denotes the Schur complement \citep[Eq. (0.8.5.2)]{Horn2012} of the invertible submatrix $\mathbf{r}_{0}(\tau)_{\{\alpha\}}^{\{\varrho(\alpha)\}}=(\tau)$ in $\mathbf{r}_{0}(\tau)_{C(\mathcal{I})}^{R(\mathcal{I})}$. 
Based on (\ref{eq: special non-separable control matrix}), and defining $\mathbf{s}:=\mathbf{r}_{0}(0)_{\{\alpha\}}^{R(\mathcal{I})\setminus\{\varrho(\alpha)\}}$, we directly obtain the following expression for the Schur complement: 
\begin{equation} 
\mathbf{r}_{(\alpha)}(\tau)=
\mathbf{r}_{0}(0)_{C(\mathcal{I})\setminus\{\alpha\}}^{R(\mathcal{I})\setminus\{\varrho(\alpha)\}}+(\mathbf{1}_{r-1}\cdot\mathbf{s}^{\mathtt{T}}-\mathbf{s}\cdot\mathbf{1}_{r-1}^{\mathtt{T}})+\tau^{-1}\cdot\mathbf{s}\cdot\mathbf{s}^{\mathtt{T}}.
\label{eq: Schur complement, b}
\end{equation}
The matrix $\mathbf{r}_{0}(0)_{C(\mathcal{I})\setminus\{\alpha\}}^{R(\mathcal{I})\setminus\{\varrho(\alpha)\}}+(\mathbf{1}_{r-1}\cdot\mathbf{s}^{\mathtt{T}}-\mathbf{s}\cdot\mathbf{1}_{r-1}^{\mathtt{T}})$ is constant, skew-symmetric, and even-dimensional, so we can invoke the matrix determinant lemma again and find $\det(\mathbf{r}_{(\alpha)}(\tau))\in\mathbb{C}$. Since $\det\left(\mathbf{r}_{0}(\tau)_{\{\alpha\}}^{\{\varrho(\alpha)\}}\right)=\tau$, we conclude that both sides of (\ref{eq: Schur formula}), and hence $\Delta_{\mathbf{R}_{0}(\tau)}(\mathcal{I})$, are non-constant. 

We observe that non-principal minors of $\mathbf{r}_{0}(\tau)$ also do not vanish for a generic choice of $\tau$ and $\mathbf{S}$. Indeed, for any vanishing minor $\det(\mathbf{r}_{0}(\tau)_{C(\mathcal{I})}^{R(\mathcal{I})})=0$ of minimal order $r$ among those that violate Assumption \ref{claim: no Y=-1}, its Laplace expansion along the $\alpha$-th column of $\mathbf{r}_{0}(\tau)_{C(\mathcal{I})}^{R(\mathcal{I})}$, for any $\alpha \in C(\mathcal{I})$ such that $\varrho(\alpha) \notin R(\mathcal{I})$,  
would entail a non-trivial linear relation among the column's entries, contradicting the genericity of $\mathbf{S}$. 

We conclude by considering $\alpha\neq \beta$ in $[k]^{\mathtt{C}}$, so that $h([k])\cdot h([k]_{\alpha\beta}^{\varrho(\alpha)\varrho(\beta)})\in\mathbb{C}$ by (\ref{eq: matrix determinant lemma}), whereas $h([k]_{\alpha}^{\varrho(\alpha)})\cdot h([k]_{\beta}^{\varrho(\beta)})\notin\mathbb{C}$. Thus, $Y([k])_{\beta\alpha}^{\varrho(\alpha)\varrho(\beta)}\notin\mathbb{C}$ by (\ref{eq: special rational, binomial}), at odds with (\ref{eq: separability of set function}) by Lemma \ref{lem: case B^(ij)_(ab) square}. 
\end{proof}

We use this counterexample to illustrate the incompatibility of orderings on $[n]$ and $\mathfrak{G}(\mathbf{L})$ arising in non-separable configurations, as discussed in Section \ref{subsec: counterexamples and relations to choice theory}. Taking $\mathcal{I}:=[k]\in\mathfrak{G}(\mathbf{L}_{0})$, for any $\alpha\in\mathcal{I}^{\mathtt{C}}$ we find weight ratios defined via the function $\Psi_{0}$, derived as in (\ref{eq: monomial terms of Cauchy-Binet expansion}) from $(\mathbf{L}_{0},\mathbf{R}_{0})$ in (\ref{eq: special non-separable control matrix}): 
\begin{equation} 
\left\Vert\frac{\mathbf{t}^{\Psi_{0}(\mathcal{I}_{\alpha}^{\varrho(\alpha)})}}{\mathbf{t}^{\Psi_{0}(\mathcal{I})}}\right\Vert 
= \left\Vert \frac{\Delta_{\mathbf{R}_{0}(\mathbf{t})}(\mathcal{I}_{\alpha}^{\varrho(\alpha)})\cdot \Delta_{\mathbf{R}_{0}(\mathbf{1})}(\mathcal{I}_{\alpha}^{\varrho(\alpha)})^{-1}}{\Delta_{\mathbf{R}_{0}(\mathbf{t})}(\mathcal{I})\cdot \Delta_{\mathbf{R}_{0}(\mathbf{1})}(\mathcal{I})^{-1}}\right\Vert \\ 
= \Vert\tau\Vert 
\label{eq: contextuality, basis 1}
\end{equation}
as follows from the proof of Proposition \ref{prop: principal minors odd non-separable}. Conversely, for any $\beta\in\mathcal{I}^{\mathtt{C}}$ with $\beta\neq\alpha$, equations (\ref{eq: matrix determinant lemma}) and (\ref{eq: Schur formula}) imply that the basis $\mathcal{J}:=\mathcal{I}_{\beta}^{\varrho(\beta)}$ yields 
\begin{equation}
\left\Vert \frac{\mathbf{t}^{\Psi_{0}(\mathcal{J}_{\alpha}^{\varrho(\alpha)})}}{\mathbf{t}^{\Psi_{0}(\mathcal{J})}}\right\Vert 
= \left\Vert \frac{\Delta_{\mathbf{R}_{0}(\mathbf{t})}(\mathcal{J}_{\alpha}^{\varrho(\alpha)})\cdot\Delta_{\mathbf{R}_{0}(\mathbf{1})}(\mathcal{J}_{\alpha}^{\varrho(\alpha)})^{-1}}{\Delta_{\mathbf{R}_{0}(\mathbf{t})}(\mathcal{J})\cdot\Delta_{\mathbf{R}_{0}(\mathbf{1})}(\mathcal{J})^{-1}}\right\Vert 
= \Vert\tau\Vert^{-1} 
\label{eq: contextuality, basis 2}
\end{equation} 
since $|\mathcal{J}\setminus\mathcal{I}|$ is odd while $|\mathcal{J}_{\alpha}^{\varrho(\alpha)}\setminus\mathcal{I}|$ is even. Therefore, for any evaluation $\mathbf{t}_{0}$ yielding a value $\tau_{0}$ with $\Vert \tau_{0} \Vert > 1$, exchanging $\alpha$ with $\varrho(\alpha)$ results in a new basis $\mathcal{I}_{\alpha}^{\varrho(\alpha)}$ with greater weight norm, $\Vert \mathbf{t}_{0}^{\Psi_{0}(\mathcal{I}_{\alpha}^{\varrho(\alpha)})} \Vert  > \Vert  \mathbf{t}_{0}^{\Psi_{0}(\mathcal{I})} \Vert$, as shown in (\ref{eq: contextuality, basis 1}); however, the same exchange within the basis $\mathcal{J}$ reverses the ordering, as it decreases the weight norm $\Vert \mathbf{t}_{0}^{\Psi_{0}(\mathcal{J}_{\alpha}^{\varrho(\alpha)})} \Vert < \Vert \mathbf{t}_{0}^{\Psi_{0}(\mathcal{J})} \Vert$ by (\ref{eq: contextuality, basis 2}). Hence, the comparison between components $\varrho(\alpha)$ and $\alpha$ is not entirely determined by their pairwise relation alone but depends on the other elements composing the basis used for the comparison; this form of contextuality clearly does not occur for separable configurations (\ref{eq: separability of set function}) and reflects the non-additive behaviour of $\Psi_{0}$. 

This inconsistency implies that comparisons derived from such weight norms are incompatible with any dimension-specific analogue. More precisely, given an injective function $\psi:\,[n]\longrightarrow\mathbb{Z}^{d}$, choose an evaluation $\mathbf{t}_{0}$ of $\mathbf{t}$ so that the values of $\Vert \mathbf{t}_{0}^{\psi(\alpha)}\Vert$, $\alpha\in[n]$, are pairwise distinct; then, define $m := \underset{\alpha\in[n]}{\text{arg min}}\Vert\mathbf{t}_{0}^{\psi(\alpha)}\Vert$ and $M := \underset{\alpha\in[n]}{\text{arg max}}\Vert\mathbf{t}_{0}^{\psi(\alpha)}\Vert$, and set $\tau_{0}:=\mathbf{t}_{0}^{\psi(M)-\psi(m)}$. Then, considering $\mathrm{diag}(\mathbf{t}_{0}^{\psi(1)},\dots,\mathbf{t}_{0}^{\psi(n)})\cdot\mathbf{R}_{0}(\tau_{0})$ in place of $\mathbf{R}_{0}(\tau_{0})$ and proceeding as above, the weight ratios in (\ref{eq: contextuality, basis 1}) and (\ref{eq: contextuality, basis 2}) become $\Vert\mathbf{t}_{0}^{\psi(M)-\psi(m)+\psi(\alpha)-\psi(\varrho(\alpha))}\Vert$ and $\Vert\mathbf{t}_{0}^{\psi(m)-\psi(M)+\psi(\alpha)-\psi(\varrho(\alpha))}\Vert$, respectively. By construction, only one of these ratios induces the same ordering as $\Vert\mathbf{t}_{0}^{\psi(\cdot)}\Vert$, which is thus incompatible with the order resulting from $\Vert \mathbf{t}_{0}^{\Psi_{0}(\cdot)} \Vert$ under the context-independence condition as specified in Section \ref{subsec: counterexamples and relations to choice theory}. 

\subsection{Multiple Weak Keys}
\label{subsubsec: Multiple weak keys}
Several weak keys, none of which is a key, do not suffice to guarantee separability. Based on the proofs in the previous sections, we construct a counterexample with 
\begin{align} 
\mathbf{L}_{\mathrm{w}} & :=
\begin{pmatrix}
\idd_{k} \quad \vline & 
\begin{array}{cc}
\mathbf{1}_{2\times (p+1)} & \mathbf{0}_{2\times (k-2)}\\
\mathbf{0}_{(k-2)\times (p+1)} & \idd_{k-2}
\end{array}
\end{pmatrix}, 
\nonumber \\  
\mathbf{R}_{\mathrm{w}}(\xi) & :=
\begin{pmatrix}
\idd_{k} \quad \vline & 
\begin{array}{ccc}
1 & \mathbf{1}_{1\times p} & -\mathbf{1}_{1\times(k-2)}\\
\xi & \mathbf{d}_{1\times p} & \mathbf{a}_{1\times(k-2)}\\
\mathbf{r}_{(k-2)\times1} & \mathbf{C}_{(k-2)\times p} & \mathbf{S}_{(k-2)\times(k-2)}
\end{array}
\end{pmatrix}^{\mathtt{T}}
\label{eq: counterexample all weak keys}
\end{align}
where $\mathbf{1}_{e\times f}$ and $\mathbf{0}_{e\times f}$ denote the $(e\times f)$-block matrices with all entries equal to $1$ and $0$, respectively; $\xi$ is a non-constant monomial; $\mathbf{d}\in\mathbb{C}^{1\times p}$, $\mathbf{a}\in\mathbb{C}^{1\times(k-2)}$, and $\mathbf{C}\in\mathbb{C}^{(k-2)\times p}$ are generic; $r_{s}:=\xi\cdot a_{s}^{-1}-1$ for all $s\in[k-2]$; and $\mathbf{S}\in\mathbb{C}^{(k-2)\times(k-2)}$ satisfies $S_{s,u}=a_{s}^{-1}\cdot a_{u}\cdot(1-S_{u,s})+1$ and $S_{s,s}=1$ for all $s,u\in[k-2]$. 
Hence, we obtain $n=2\cdot k+p-1$. 
For $k=5$ and $p=4$, the matrix $\mathbf{R}_{\mathrm{w}}(\xi)$ 
takes the form 
\begin{equation*}
    \left(
\begin{smallmatrix}
 1 & 0 & 0 & 0 & 0 & 1 & 1 & 1 & 1 & 1 & -1 & -1 & -1 \\
 0 & 1 & 0 & 0 & 0 & \xi  & d_{1} & d_{2} & d_{3} & d_{4} & a_{1} & a_{2} & a_{3} \\
 0 & 0 & 1 & 0 & 0 & \frac{\xi }{a_{1}}-1 & C_{1,1} & C_{1,2} & C_{1,3} & C_{1,4} & 1 & \frac{a_{2}\cdot \left(1-S_{2,1}\right)}{a_{1}}+1 & \frac{a_{3}\cdot \left(1-S_{3,1}\right)}{a_{1}}+1 \\
 0 & 0 & 0 & 1 & 0 & \frac{\xi }{a_{2}}-1 & C_{2,1} & C_{2,2} & C_{2,3} & C_{2,4} & S_{2,1} & 1 & \frac{a_{3}\cdot \left(1-S_{3,2}\right)}{a_{2}}+1 \\
 0 & 0 & 0 & 0 & 1 & \frac{\xi }{a_{3}}-1 & C_{3,1} & C_{3,2} & C_{3,3} & C_{3,4} & S_{3,1} & S_{3,2} & 1 \\ 
\end{smallmatrix}
\right)^{\mathtt{T}} 
\end{equation*} 
where we can verify that all the maximal minors 
are non-zero for generic choices of the constants parametrising the matrix, while $\Delta_{\mathbf{R}_{\mathrm{w}}(\xi)}(\mathcal{I})$, for $\mathcal{I}\in\mathfrak{G}(\mathbf{L}_{\mathrm{w}})$, is a monomial. Some $Y$-terms are non-constant; specifically, $Y([k])_{k+1\,k+2}^{1\,2}=-d_{1}\cdot \xi^{-1}\notin\mathbb{C}$, which conflicts with separability by Lemma~\ref{lem: case B^(ij)_(ab) square}. 


\subsection{Planar Key}
\label{subsubsec: Unexplainable key}

Finally, we present a counterexample where a violation of (\ref{eq: separability of set function}) and separability can emerge due to the planarity of keys. In this configuration, a structural matrix $\mathbf{L_{p}}$ exhibits a key $\mathfrak{c}=\{i_{1},i_{2}\}\times\{\alpha_{1},\alpha_{2}\}\subseteq\mathcal{I}\times\mathcal{I}^{\mathtt{C}}$, with $\mathcal{I}\in\mathfrak{G}(\mathbf{L_{p}})$, such that for all $\omega\in\mathcal{I}^{\mathtt{C}}\setminus\mathfrak{c}_{c}$, there exists a unique $\varrho(\omega)\in\mathcal{I}$ for which $h(\mathcal{I}_{\omega}^{\varrho(\omega)})\neq0$. We take a non-constant monomial $\zeta$ to obtain $Y_{\omega\alpha_{1}}^{i_{1}\varrho(\omega)} = c_{\omega}^{2}\cdot\zeta-1$, where $c_{\omega}\in\mathbb{C}\setminus\{0\}$, for all $\omega\in(\mathcal{I}_{\alpha_{1}\alpha_{2}})^{\mathtt{C}}$. Such terms arise from the following non-separable configuration: 
\begin{align} 
\mathbf{L_{\mathrm{p}}} & := 
\begin{pmatrix}
\idd_{k} \quad \vline & 
\begin{array}{ccc}
1 & 1 & \mathbf{0}_{1\times (k-2)}\\ 
1 & -1 & \mathbf{0}_{1\times (k-2)}\\ 
\mathbf{0}_{(k-2)\times 1} & \mathbf{0}_{(k-2)\times 1} & \idd_{k-2} 
\end{array}
\end{pmatrix}, 
\nonumber \\  
\mathbf{R_{p}}(\zeta) & := \begin{pmatrix}-\zeta^{-1} & \zeta^{-1} & \mathbf{0}_{1\times(k-2)} & 0 & -1 & \mathbf{c^{(r)}}_{1\times(k-2)}\\
\vartheta & 1 & \mathbf{0}_{1\times(k-2)} & 1 & 0 & \mathbf{1}_{1\times(k-2)}\\
\mathbf{0}_{(k-2)\times1} & \mathbf{0}_{(k-2)\times 1} & \idd_{k-2} & \mathbf{1}_{(k-2)\times1} & \mathbf{c^{(c)}}_{(k-2)\times 1} & \mathbf{Z}_{(k-2)\times(k-2)} 
\end{pmatrix}^{\mathtt{T}}
\label{eq: counterexample planar key}
\end{align}
with $\vartheta=Y_{\alpha_{1}\alpha_{2}}^{i_{1}i_{2}}\in\mathbb{C}$ as in the proof of Proposition \ref{prop: constant strip and principal minors}, and the entries 
\begin{equation*} 
c_{\omega}^{(r)} = -\frac{1}{c_{\varrho(\omega)}^{(c)}}:=c_{\omega}^{2}, \quad 
Z_{\varrho(\nu),\varrho(\omega)} :=  1+\mathrm{sign}(\omega-\nu)\cdot\frac{c_{\omega}\cdot \mu_{\{\nu,\omega\}}}{c_{\nu}},\quad \nu,\omega\in(\mathcal{I}_{\alpha_{1}\alpha_{2}})^{\mathtt{C}}  
\end{equation*} 
are expressed in terms of constants $c_{\omega}$ and $\mu_{\{\nu,\omega\}}$, where $\mu_{\{\nu,\omega\}}^{2}=\mathfrak{m}_{\alpha_{1}\nu\omega}^{i_{1}\varrho(\nu)\varrho(\omega)}$ introduced in (\ref{eq: three Y-term}). Taking $k=7$ as a concrete instance, similarly to the previous example, we find that (\ref{eq: counterexample planar key}) satisfies (\ref{eq: monomial terms of Cauchy-Binet expansion}) and Assumption \ref{claim: no Y=-1}, but $Y([k])_{\omega\alpha_{1}}^{i_{1}\varrho(\omega)}=c_{\omega}^{2}\cdot\zeta -1 \notin\mathbb{C}$, which is incompatible with separability. 

\section{Concluding Remarks}
\label{sec: conclusion}

This work proposes a novel framework for analysing multidimensional representations of factor models, explicitly capturing dependence patterns via determinantal characterisations. By recasting the factor pairing problem within a Grassmannian formulation, we sharpen geometric conditions that support the methodological treatment of uncertainty-related conceptual issues within instrumental variable methods and factor score indeterminacy. This approach enables a sensitivity analysis based on an algebraic parameterisation, aimed at assessing whether component-wise contributions can be disentangled across structurally observable subsets of dimensions, as expressed by separability of variation weights. 

These findings yield two key implications. First, they facilitate the definition of invariants for studying the intrinsic uncertainty in factor decompositions. In particular, such a representation makes order-theoretic ambiguity manifest in terms of context-dependent comparisons of dimensional contributions, paving the way for more general approaches to encode uncertainty in multidimensional modelling. Second, the framework leads to new graphical descriptions associated with the model, enhancing the interpretability of separability conditions as outlined in Remark \ref{rem: analogy with Kuratowski}.

While our analysis focuses on a specific parameterisation, future work may explore alternative weighting schemes suited to particular matrix decomposition problems, along with the corresponding graphical characterisations. Extending the study beyond a single structural factor, as noted in Remark \ref{rem: self-compatibility score matrices}, subspace pairing can also be specified to connect different models---via their score solutions and solution differences---offering an informative view on their alignment and comparison. This perspective includes deepening dual representations, substructures related to submodels or other notions of observability \citep{krijnen1998conditions}, and consistency conditions under adding variables \citep{mulaik1978effect}. 

Aside from the main scope of this work, our formulation has revealed potential connections to the geometry of uncertainty related to evidence theory \citep{Cuzzolin2020}, particularly regarding relative beliefs of singletons \citep{Cuzzolin2010}. Finally, as anticipated in earlier sections, broader applications of the expansion (\ref{eq: Cauchy-Binet expansion}) may benefit the investigation of principal angles, canonical correlations, or cross-correlations between given sets of variables and scores, offering refined tools for uncertainty quantification in psychological measurement. 


\section{Proofs of Auxiliary Results} 
\label{app: Proofs of auxiliary results} 

\subsection{Proof of Proposition \ref{prop: constant strip and principal minors}} 
\label{app: constant strip and principal minors} 

\begin{lem}
\label{lem: one configuration from 2-key} 
For each weak key $\mathfrak{c}:=\{i,j\}\times\{\alpha,\beta\}$ satisfying (\ref{eq: local context for representation of 4-square}) and for any $(m,\omega)\in\mathcal{N}(\mathfrak{c})$ with $h(\mathcal{I}_{\omega}^{m})\neq0$, there exist at most two allowed polynomial configurations for the set $\left\{Y_{\alpha\omega}^{mi},Y_{\beta\omega}^{mi},Y_{\alpha\omega}^{mj},Y_{\beta\omega}^{mj}\right\}$, provided that it contains at least one non-constant term. 
\end{lem}
\begin{proof}
We define 
\begin{equation}
\Upsilon_{\omega}^{+}:=\left\{ Y_{\alpha\omega}^{mi},Y_{\beta\omega}^{mj}\right\} ,\quad\Upsilon_{\omega}^{-}:=\left\{ Y_{\alpha\omega}^{mj},Y_{\beta\omega}^{mi}\right\} ,\quad\Upsilon_{\omega}:=\Upsilon_{\omega}^{+}\cup\Upsilon_{\omega}^{-}
\label{eq: directed pairs}
\end{equation}
and introduce the notation $\Upsilon_{\omega}^{\sigma}=:\{Y_{\sigma,1},Y_{\sigma,2}\}$ for each $\sigma\in\{+,-\}$. From 
(\ref{eq: quadrilateral decomposition}), we obtain 
 \begin{equation}
Y_{\alpha\beta}^{ij}\cdot\left(Y_{\alpha\omega}^{mi}\cdot Y_{\beta\omega}^{mj}\right)=-\left(Y_{\alpha\omega}^{mj}\cdot Y_{\beta\omega}^{mi}\right).
\label{eq: monomial scaling of product of binomials}
 \end{equation}
By applying Lemma \ref{lem: algebraic forms Y-terms} and the assumptions that $(m,\omega)\in\mathcal{N}(\mathfrak{c})$, $h(\mathcal{I}_{\omega}^{m})\neq 0$, and $\mathfrak{c}$ satisfies (\ref{eq: local context for representation of 4-square}), we deduce that each term in $\Upsilon_{\omega}$ takes the form (\ref{eq: special rational, binomial}). Furthermore, since 
(\ref{eq: local context for representation of 4-square}) holds for $\mathfrak{c}$ and $Y_{\alpha\beta}^{ij}\in\mathbb{F}$ by (\ref{eq: monomial scaling of product of binomials}), either (\ref{eq: special rational, monomial}) or Remark \ref{rem: algebraic key yields constant} applies, ensuring that $Y_{\alpha\beta}^{ij}=:\vartheta$ is a non-zero (possibly constant) monomial. 

In the case $\Upsilon_{\omega}^{\sigma}\cap\mathbb{C}=\emptyset$, the factors of $Y_{\sigma,1}\cdot Y_{\sigma,2}$ uniquely determine 
the set $\left\{\Psi^{(1)}(Y_{\sigma,1}+1),\right.$ $\left.\Psi^{(1)}(Y_{\sigma,2}+1)\right\}$, up to independent changes of sign, where $\Psi^{(1)}(\cdot)$ is defined in Remark \ref{rem: notation for Psi}. Starting from (\ref{eq: monomial scaling of product of binomials}) and considering Assumption \ref{claim: no Y=-1},  
when $\Upsilon_{\omega}\cap\mathbb{C}=\emptyset$ we infer 
 \begin{equation}
\{\Psi^{(1)}(Y_{+,1}+1),\Psi^{(1)}(Y_{+,2}+1)\}=\{-\Psi^{(1)}(Y_{-,1}+1),-\Psi^{(1)}(Y_{-,2}+1)\} 
\label{eq: even number of constant Y-terms}
 \end{equation}
which also holds in the general case where $|\Upsilon_{\omega}\cap\mathbb{C}|$ is even. This follows trivially when $\Upsilon_{\omega}\subset\mathbb{C}$ and from Assumption \ref{claim: no Y=-1} when $|\Upsilon_{\omega}\cap\mathbb{C}|=2$. In particular, when $\Upsilon_{\omega}\nsubseteq\mathbb{C}$, there exists a non-zero monomial $\tau\in\mathbb{C}[\mathbf{t},\mathbf{t}^{-1}]$ such that 
 \begin{equation}
\Upsilon_{\omega}^{+}=\left\{\tau^{-1}-1,\tau\vartheta^{-1}\cdot\Omega\right\},\quad\Upsilon_{\omega}^{-}=\left\{\tau-1,\Omega\right\},\quad \Omega\in\{-\vartheta\tau^{-1}-1\}\cup\mathbb{C}.
\label{eq: binomial configuration from 5 Y terms, even-typed}
 \end{equation} 
Moreover, since $\Upsilon_{\omega}\nsubseteq\mathbb{C}$, at least two $Y$-terms in (\ref{eq: monomial scaling of product of binomials}) must not be monomials in $\mathbb{C}[\mathbf{t},\mathbf{t}^{-1}]$. Thus, when $|\Upsilon_{\omega}\cap\mathbb{C}|$ is odd, we must have $|\Upsilon_{\omega}\cap\mathbb{C}|=1$; specifically, there exist a permutation $\sigma$ of $\{+,-\}$ and $Y\in\Upsilon_{\omega}^{\sigma(+)}$ such that each element of $\Upsilon_{\omega}^{\sigma(-)}$ is a polynomial dividing $Y$. This condition allows us to express $Y_{\alpha\beta}^{ij}$ as $\vartheta = \varepsilon_{2}\cdot C\cdot\tau^{(1+\varepsilon_{1}-2\varepsilon_{2})/2}$ or its reciprocal, according to $\sigma$, for suitable choices of $\varepsilon_{1},\varepsilon_{2}\in\{1,-1\}$ and a non-zero constant $C\in\mathbb{C}$. This leads to the second configuration:  
\begin{equation} 
\Upsilon_{\omega}^{\sigma(+)}=\left\{\tau^{2\cdot\varepsilon_{2}}-1,C^{-1}\right\},\quad \Upsilon_{\omega}^{\sigma(-)} =  \left\{\tau-1,-\tau^{\varepsilon_{1}}-1\right\}.
\label{eq: binomial configuration from 5 Y terms, odd-typed}
\end{equation}
\end{proof} 
\vspace{-.5cm}
We now proceed with the proof of Proposition \ref{prop: constant strip and principal minors}, which further examines the properties of 
these configurations. 
\begin{proof}
The thesis holds for $\mathcal{I}$ if and only if it also holds for $\mathcal{J}:=\mathcal{I}_{\alpha}^{i}$, where $(i,\alpha)\in\mathfrak{c}$ and $\mathcal{J}\in\mathfrak{G}(\mathbf{L})$. Indeed, by Remark \ref{rem: invariance null arrays}, $\mathfrak{c}_{\mathcal{J}}$ remains a key, and this change of basis preserves the sets $\mathcal{N}^{\mathfrak{c}_{c}}$,  $\mathcal{N}_{\mathfrak{c}_{r}}$, and hence the planarity condition (\ref{eq: equivocal condition}). Moreover, $(m,\omega)\in\mathcal{N}(\mathfrak{c}_{\mathcal{J}})$ whenever $(m,\omega)\in\mathcal{N}(\mathfrak{c})$. Then, the transformation rule (\ref{eq: Grassmann-Plucker translation exchange, diagonal}) implies 
\begin{equation}
Y(\mathcal{J})_{i\omega}^{\alpha j},Y(\mathcal{J})_{i\omega}^{m\alpha},Y(\mathcal{J})_{\beta i}^{m\alpha},Y(\mathcal{J})_{\beta i}^{\alpha j}\in\mathbb{C}\Leftrightarrow Y(\mathcal{I})_{\alpha\omega}^{ij},Y(\mathcal{I})_{\alpha\omega}^{mi},Y(\mathcal{I})_{\beta\alpha}^{mi},Y(\mathcal{I})_{\beta\alpha}^{ij}\in\mathbb{C}
\label{eq: equivalence under shift of basis}
\end{equation}
for all $i,j\in\mathfrak{c}_{r}$ and $\alpha,\beta\in\mathfrak{c}_{c}$. Thus, applying (\ref{eq: quadrilateral decomposition}), we deduce that 
\begin{equation} 
Y(\mathcal{I})_{\beta\omega}^{mj}=-Y(\mathcal{I})_{\beta\alpha}^{mi}\cdot Y(\mathcal{I})_{\alpha\omega}^{mi}\cdot Y(\mathcal{I})_{\beta\alpha}^{ij}\cdot Y(\mathcal{I})_{\alpha\omega}^{ij}\in\mathbb{C},\quad (j,\beta)\in\mathfrak{c}.
\label{eq: equivalence pivot indices in key}
\end{equation}

For each $\lambda\notin\mathcal{N}_{\mathfrak{c}_{r}}$, we can, by definition, find an index $i\in\mathfrak{c}_{r}$   
such that $h(\mathcal{I}_{\lambda}^{i})\neq0$, say $i=i_{1}$. If $\mathfrak{c}_{\lambda}:=\mathfrak{c}_{r}\times \{\alpha,\lambda\}$ is not a key for either $\alpha\in\{\alpha_{1},\alpha_{2}\}$, then we use relations (\ref{eq: equivalence under shift of basis})-(\ref{eq: equivalence pivot indices in key}) to move to   
$\mathcal{I}_{\alpha_{1}}^{i_{1}}$ and confirm that $\mathfrak{c}_{\lambda}:=\{\alpha_{1},i_{2}\}\times\{\alpha_{2},\lambda\}$ is a key in $\mathcal{I}_{\alpha_{1}}^{i_{1}}$. Therefore, we can always find a basis $\mathcal{H}\in\{\mathcal{I},\mathcal{I}_{\alpha_{1}}^{i_{1}}\}$ such that this construction yields a new key $\mathfrak{c}_{\lambda}$ in $\mathcal{H}$, in addition to $\mathfrak{c}_{\mathcal{H}}$. 
When $\omega\notin\mathcal{N}_{\mathfrak{c}_{r}}$, 
we specify $\lambda:=\omega$ in the construction of $\mathfrak{c}_{\lambda}$ to produce a key $\mathfrak{c}_{\omega}$ 
in a properly chosen basis $\mathcal{H}$. Applying Proposition \ref{prop: constant 2-key} to each key $\mathfrak{c}_{\mathcal{H}}$ and $\mathfrak{c}_{\omega}$ separately and, where needed, the binomial argument used in Lemma \ref{lem: constant 2-key under conditions}, we conclude that each term on the right-hand side of (\ref{eq: equivalence pivot indices in key}) is constant. Using (\ref{eq: equivalence under shift of basis}) if necessary, we obtain the thesis for $\mathcal{I}$. The same conclusion holds by starting with an index $m\notin\mathcal{N}^{\mathfrak{c}_{c}}$, exchanging the roles of upper and lower indices in the previous construction. 
Therefore, we restrict our attention to $(m,\omega)\in \mathcal{N}(\mathfrak{c})$. 

Since $\mathfrak{c}$ is non-planar, there exists an index $\alpha_{3}\in[n]$ that falsifies (\ref{eq: equivocal condition}); we can assume $\alpha_{3}\in\mathcal{I}^{\mathtt{C}}$, transposing indices otherwise. Adapting the previous procedure by setting $\lambda:=\alpha_{3}$, we choose a basis $\mathcal{H}$ to obtain $\mathfrak{c}_{\mathcal{H}}$ and a new key, denoted $\mathfrak{c}_{3}$. As above, 
we have $(m,\omega)\in \mathcal{N}(\mathfrak{c}_{\mathcal{H}})$ by Remark \ref{rem: invariance null arrays}. Thus, we can assume $\mathcal{H}=\mathcal{I}$ without loss of generality, simplifying notation. The entire reasoning above also applies if $(m,\omega)\notin \mathcal{N}(\mathfrak{c}_{3})$, proving the thesis. 

Thus, we focus on $(m,\omega)\in \mathcal{N}(\mathfrak{c})\cap\mathcal{N}(\mathfrak{c}_{3})$. For all $u,w\in[3]$ with $u\neq w$, we have $Y_{\alpha_{u}\alpha_{w}}^{i_{1}i_{2}}\in\mathbb{C}\setminus\{-1\}$ by Proposition \ref{prop: constant 2-key}. It follows that the corresponding (weak) key $\mathfrak{c}_{(u,w)}:=\mathfrak{c}_{r}\times \{\alpha_{u},\alpha_{w}\}$ generates sets $\Upsilon_{\omega, (u,w)}^{+}$ and $\Upsilon_{\omega, (u,w)}^{-}$ as in (\ref{eq: directed pairs}) and, whether they contain only constants or take the form (\ref{eq: binomial configuration from 5 Y terms, even-typed}) or (\ref{eq: binomial configuration from 5 Y terms, odd-typed}), the two elements of each of these sets 
do not share any non-zero common root. 
Thus, any two terms sharing the same non-zero roots, belonging to different sets $\Upsilon_{\omega,(u,w)}^{+}$ and $\Upsilon_{\omega,(u,w)}^{-}$, differ by a single (upper or lower) index, and the opposite of their ratio is a $Y$-term by (\ref{eq: associativity}). Consequently, the configuration (\ref{eq: binomial configuration from 5 Y terms, even-typed}) does not contain both constant and non-constant terms, as the two non-constant ones would be proportional over $\mathbb{C}$ and, taking the form (\ref{eq: special rational, binomial}), yield a $Y$-term (the opposite of their ratio) equal to $-1$, contradicting Assumption \ref{claim: no Y=-1}. 
Thus, 
$\left|\Upsilon_{\omega, (u,w)}\cap\mathbb{C}\right|\leq 1$ unless $\Upsilon_{\omega, (u,w)}\subset \mathbb{C}$. 

Now, expressing each $Y_{\alpha_{u}\omega}^{i_{1}i_{2}}=-Y_{\alpha_{u}\omega}^{i_{1}m}/Y_{\alpha_{u}\omega}^{i_{2}m}$, for $u\in\{1,2,3\}$, as the opposite of a ratio of two terms from $\Upsilon_{\omega,(u,w)}^{+}$ and $\Upsilon_{\omega,(u,w)}^{-}$ 
(not necessarily in that order), we observe the following: the term $Y_{\alpha_{u}\omega}^{i_{1}i_{2}}$, as a function of $\mathbf{t}$ derived from (\ref{eq: binomial configuration from 5 Y terms, even-typed}), has non-zero roots (i.e., it is not a monomial) if and only if its reciprocal $Y_{\alpha_{u}\omega}^{i_{2}i_{1}}$ also has non-zero roots. 
Conversely, when starting from (\ref{eq: binomial configuration from 5 Y terms, odd-typed}), we obtain 
non-zero roots for either $Y_{\alpha_{u}\omega}^{i_{1}i_{2}}$ or $Y_{\alpha_{u}\omega}^{i_{2}i_{1}}$, but not both. This property of $Y_{\alpha_{u}\omega}^{i_{1}i_{2}}$ uniquely identifies, for each $u\in[3]$, the configuration type, either as given by (\ref{eq: binomial configuration from 5 Y terms, even-typed}) or (\ref{eq: binomial configuration from 5 Y terms, odd-typed}), 
associated with the sets $\Upsilon_{\omega,(u,w)}^{+}$ and $\Upsilon_{\omega,(u,w)}^{-}$. 

Given these premises, suppose that there exists a non-constant term, say $Y_{\alpha_{1}\omega}^{mi}$ (with $i\in\mathfrak{c}_{r}$) by appropriately labelling the lower indices. Both this term and $Y_{\alpha_{1}\omega}^{i_{1}i_{2}}$ are associated with two different sets $\Upsilon_{\omega,(1,u)}$ (for $u\in\{2,3\}$). The term $Y_{\alpha_{1}\omega}^{mi}$ generates at least one other non-constant term in the form $Y_{\alpha_{w}\omega}^{mj}$ with $w\neq 1$, which also belongs to the remaining set $\Upsilon_{\omega,(2,3)}$. 
This implies that all three sets $\Upsilon_{\omega, (u,w)}$ satisfy $\left|\Upsilon_{\omega, (u,w)}\cap\mathbb{C}\right|\leq 1$, and hence each takes one of the two forms (\ref{eq: binomial configuration from 5 Y terms, even-typed}) or (\ref{eq: binomial configuration from 5 Y terms, odd-typed}). Furthermore, this bound implies that at most one of the terms $Y_{\alpha_{u}\omega}^{mi}$ (with $u\in[3]$ and $i\in\mathfrak{c}_{r}$) is constant, so we can find at least one set $\Upsilon_{\omega,(w_{1},w_{2})}$ 
without constant terms, i.e., of the type (\ref{eq: binomial configuration from 5 Y terms, even-typed}). On the other hand, the term $Y_{\alpha_{1}\omega}^{i_{1}i_{2}}$ determines the same configuration for both $\Upsilon_{\omega,(1,w)}$ with $w\in\{2,3\}$ as observed above, and then, by the same token, for $\Upsilon_{\omega,(2,3)}$. Therefore, all three $\Upsilon$-sets arise from (\ref{eq: binomial configuration from 5 Y terms, even-typed}). Finally, we focus on $Y_{\omega\alpha_{u}}^{i_{1}m}$ for $u\in[3]$ and note that one of these terms does not share non-zero roots with the remaining two; otherwise, by (\ref{eq: binomial configuration from 5 Y terms, even-typed}) they would take the form $\xi^{\varepsilon_{u}}-1$ with a non-constant monomial $\xi$ and $\varepsilon_{u}\in\{1,-1\}$. In that case, the same value of $\varepsilon_{u}$ would occur for two of these  $Y$-terms, making them coincide and the opposite of their ratio (a $Y$-term) equal to $-1$, at odds with Assumption \ref{claim: no Y=-1}. We label the indices so that $Y_{\omega\alpha_{1}}^{i_{1}m}$ does not share non-zero roots with either $Y_{\omega\alpha_{2}}^{i_{1}m}$ or $Y_{\omega\alpha_{3}}^{i_{1}m}$. Setting 
$Y_{\omega\alpha_{1}}^{i_{1}m}:=\xi-1$ and $\vartheta_{u}^{-1}:=Y_{\alpha_{u}\alpha_{1}}^{i_{1}i_{2}}$ for $u\in\{2,3\}$,  
we obtain  
$Y_{\alpha_{u}\alpha_{1}}^{i_{1}m}=-Y_{\alpha_{u}\omega}^{i_{1}m}\cdot Y_{\omega\alpha_{1}}^{i_{1}m}=\frac{\xi-1}{\vartheta_{u}\xi+1}$ and 
$Y_{\alpha_{3}\alpha_{2}}^{i_{2}m}=
    -\frac{(\vartheta_{2}\xi)^{-1}+1}{(\vartheta_{3}\xi)^{-1}+1}$.   
Evaluating (\ref{eq: diagonal minors 3-term expansion}) at $(a_{1},a_{2},a_{3}):=(i_{1},i_{2},m)$ and $(\delta_{1},\delta_{2},\delta_{3}):=(\alpha_{1},\alpha_{2},\alpha_{3})$, we find $\mathfrak{m}_{\alpha_{1}\alpha_{2}\alpha_{3}}^{i_{1}i_{2}m}=0$. By (\ref{eq: three Y-term}) and since the indices are pairwise distinct, this contradicts Assumption \ref{claim: no Y=-1} once more. Thus, $\xi\in\mathbb{C}$, and the thesis follows. 
\end{proof} 

\subsection{Proof of Proposition \ref{prop: no non-trivial 2-curvature}} 
\label{app: no non-trivial 2-curvature} 
\begin{proof}
Let $\mathfrak{c}:=\{i_{1},i_{2}\}\times\{\alpha_{1},\alpha_{2}\}$. For any $(i,\alpha)\in\mathfrak{c}$ and $s\in\{1,2\}$, 
$Y_{\alpha\omega_{s}}^{im_{s}}$ is constant by Proposition \ref{prop: constant strip and principal minors}. This observation extends to the case where  $h(\mathcal{I}_{\omega_{2}}^{m_{1}})\cdot h(\mathcal{I}_{\omega_{1}}^{m_{2}})\neq 0$, since Proposition \ref{prop: constant strip and principal minors} applies to all observable sets $\chi(\mathcal{I}\mid_{\alpha\omega_{s}}^{im_{u}})$ with $s,u\in\{1,2\}$. Consequently, $Y_{\alpha\omega_{s}}^{im_{u}}\in\mathbb{C}$, and from (\ref{eq: quadrilateral decomposition}) it follows that $Y_{\omega_{1}\omega_{2}}^{m_{1}m_{2}}\in\mathbb{C}$. 
Thus, in the remainder of this proof, we assume $h(\mathcal{I}_{\omega_{2}}^{m_{1}})\cdot h(\mathcal{I}_{\omega_{1}}^{m_{2}})=0$, which implies that $h(\mathcal{I}_{\omega_{1}\omega_{2}}^{m_{1}m_{2}})\neq 0$. 

We consider the change of basis $\mathcal{I}\mapsto\mathcal{I}_{\omega_{1}}^{m_{1}}$, observing that $\mathfrak{c}$ remains a non-planar key by Remark \ref{rem: planarity preservation} since $(m_{1},\omega_{1})\in\mathcal{N}(\mathfrak{c})$. The set $\chi(\mathcal{I}_{\omega_{1}}^{m_{1}}\mid_{\alpha\omega_{2}}^{im_{2}})$ is observable because $h(\mathcal{I}_{\omega_{1}\omega_{2}}^{m_{1}m_{2}})\cdot h(\mathcal{I}_{\omega_{1}\alpha}^{m_{1}i})\neq 0$ and, as before, we obtain $Y(\mathcal{I}_{\omega_{1}}^{m_{1}})_{\alpha\omega_{2}}^{im_{2}}\in\mathbb{C}$ by Proposition \ref{prop: constant strip and principal minors}. Furthermore, the assumption $h(\mathcal{I}_{\omega_{2}}^{m_{1}})\cdot h(\mathcal{I}_{\omega_{1}}^{m_{2}})=0$ leads to the form (\ref{eq: special rational, binomial}) for $Y_{\omega_{2}\omega_{1}}^{m_{1}m_{2}}=:\tau-1$, where $\tau$ denotes a non-zero monomial. The fact that $Y(\mathcal{I}_{\omega_{1}}^{m_{1}})_{\alpha\omega_{2}}^{m_{2}i}$, 
$Y_{\alpha\omega_{2}}^{m_{2}i}$, and 
$Y_{\alpha\omega_{1}}^{m_{1}i}$ 
are constant allows us to express, for $u\in\{1,2\}$,  $Y_{\omega_{1}\alpha}^{i_{u}m_{1}}=:c_{u,1}-1$, $Y_{\omega_{2}\alpha}^{i_{u}m_{2}}=:c_{u,2}-1$, and, specifying (\ref{eq: g-factor for change of basis}) at $(a_{1},a_{2},a_{3})=(m_{1},m_{2},i)$ and $(\delta_{1},\delta_{2},\delta_{3})=(\omega_{1},\omega_{2},\alpha)$, $\mathfrak{m}_{\alpha\omega_{1}\omega_{2}}^{i_{u}m_{1}m_{2}}=c_{u,3}\cdot\tau$, where $c_{u,1},c_{u,2},c_{u,3}\in\mathbb{C}\setminus\{0\}$.  
Lemma \ref{lem: quadratic extension from diagonal binomials} and Remark \ref{rem: hyperdeterminant} assert that 
$\zeta_{u}^{(1)}:=Y_{\omega_{2}\omega_{1}}^{m_{1}m_{2}}\cdot Y_{\omega_{2}\alpha}^{i_{u}m_{1}}$ and $\zeta_{u}^{(2)}:=Y_{\omega_{2}\omega_{1}}^{m_{1}m_{2}}\cdot Y_{\omega_{1}\alpha}^{i_{u}m_{2}}$ 
are the roots of the quadratic equation $P_{\alpha\omega_{1}\omega_{2}}^{i_{u}m_{1}m_{2}}(X)=0$, where $P_{\alpha\omega_{1}\omega_{2}}^{i_{u}m_{1}m_{2}}$ is a specialisation of (\ref{eq: diagonal minors, quadratic polynomial}) with discriminant $\Delta_{u}:=\Delta_{\alpha\omega_{1}\omega_{2}}^{i_{u}m_{1}m_{2}}$ defined in (\ref{eq: discriminant}). Observe that $\zeta_{1}^{(1)}$ 
and $\zeta_{2}^{(1)}$ are proportional over $\mathbb{C}$ because $\zeta_{1}^{(1)}=-Y_{\omega_{2}\alpha}^{i_{1}i_{2}}\cdot \zeta_{2}^{(1)}$ and $Y_{\omega_{2}\alpha}^{i_{1}i_{2}}=-Y_{\omega_{2}\alpha}^{i_{1}m_{2}}\cdot  Y_{\omega_{2}\alpha}^{m_{2}i_{2}}\in\mathbb{C}$. Analogously, $\zeta_{1}^{(2)}$ 
and $\zeta_{2}^{(2)}$ are proportional over $\mathbb{C}$ with coefficient $-Y_{\omega_{1}\alpha}^{i_{1}i_{2}}\in\mathbb{C}$. Using the notation introduced above, the coefficients of these two polynomials belong to $\mathbb{F}$, and we obtain the explicit form 
 \begin{equation}
\Delta_{u}=(-c_{u,1}+c_{u,2}+\tau-c_{u,3}\cdot\tau)^{2}-4\cdot c_{u,2} \cdot (c_{u,1}-1)\cdot(c_{u,3}\cdot c_{u,2}^{-1}-1)\cdot\tau. 
\label{eq: discriminant, explicit form}
 \end{equation}
Since $Y_{\omega_{2}\alpha}^{i_{1}m_{1}}\cdot Y_{\omega_{2}\alpha}^{m_{1}i_{2}}=-Y_{\omega_{2}\alpha}^{i_{1}i_{2}}=Y_{\omega_{2}\alpha}^{i_{1}m_{2}}\cdot Y_{\omega_{2}\alpha}^{m_{2}i_{2}}\in\mathbb{C}$, it follows that $\mathbb{C}[\mathbf{t},\mathbf{t}^{-1}]$ contains either both $\sqrt{\Delta_{1}}$ and $\sqrt{\Delta_{2}}$ or neither. 
If the two discriminants are not perfect squares, then $\zeta_{u}^{(1)}\notin\mathbb{F}$, which is compatible with the aforementioned conditions $\zeta_{u}^{(1)}+\zeta_{u}^{(2)}\in\mathbb{F}$ and $\zeta_{1}^{(s)}/\zeta_{2}^{(s)}\in\mathbb{C}$, for all $s,u\in\{1,2\}$, only if $Y_{\omega_{1}\alpha}^{i_{1}i_{2}}=Y_{\omega_{2}\alpha}^{i_{1}i_{2}}$. This, in turn, implies $Y_{\omega_{1}\omega_{2}}^{i_{1}i_{2}}=-1$, contradicting Assumption \ref{claim: no Y=-1}; it follows that $\Delta_{1}$ and $\Delta_{2}$ are perfect squares in $\mathbb{C}[\mathbf{t},\mathbf{t}^{-1}]$. 

On the other hand, evaluating the  discriminant of $\Delta_{u}$ with respect to $\tau$, we find that it vanishes only if $c_{u,s}-1=0$ or $c_{u,s} - c_{u,3} = 0$ for some $s\in\{1,2\}$. From (\ref{eq: g-factor for change of basis}), we directly verify that 
$\frac{c_{u,3}}{c_{u,2}}-1=Y(\mathcal{I}_{\omega_{2}}^{m_{2}})_{\omega_{1}\alpha}^{i_{u}m_{1}}$, and analogously $\frac{c_{u,3}}{c_{u,1}}-1=Y(\mathcal{I}_{\omega_{1}}^{m_{1}})_{\omega_{2}\alpha}^{i_{u}m_{2}}$. Along with $c_{u,s}-1=Y(\mathcal{I})_{\omega_{s}\alpha}^{i_{u}m_{s}}$, Assumption \ref{claim: no Y=-1} prevents each of these terms from vanishing and,  
as seen from (\ref{eq: discriminant, explicit form}), the vanishing of the coefficient of $\tau^{2}$ implies that the coefficient of $\tau$ is non-zero.  
It follows that $\Delta_{u}$ is a non-constant polynomial in $\tau$ having no multiple roots. Thus, it can be a perfect square in $\mathbb{C}[\mathbf{t},\mathbf{t}^{-1}]$ only if it is a monomial in $\tau$ and $\sqrt{\tau}\in\mathbb{C}[\mathbf{t},\mathbf{t}^{-1}]$. Hence, the constant term $(c_{u,1}-c_{u,2})^{2}$ in (\ref{eq: discriminant, explicit form}) must vanish, i.e., $Y_{\omega_{1}\alpha}^{i_{u}m_{1}}=Y_{\omega_{2}\alpha}^{i_{u}m_{2}}$ for $u\in\{1,2\}$. We conclude  
$Y_{\omega_{2}\omega_{1}}^{i_{1}i_{2}}\ensuremath{=-Y_{\omega_{2}\alpha}^{i_{1}m_{2}}\cdot Y_{\omega_{2}\alpha}^{m_{2}i_{2}}}\cdot Y_{\alpha\omega_{1}}^{i_{1}m_{1}}\cdot Y_{\alpha\omega_{1}}^{m_{1}i_{2}}=-1$, contradicting Assumption \ref{claim: no Y=-1} again. Thus, $\tau\in\mathbb{C}$, and the thesis follows. 
\end{proof} 

{\small 
\bibliographystyle{plainnat}
\bibliography{references}
}

\end{document}